\documentclass[bj,preprint]{imsart}

\RequirePackage[OT1]{fontenc}
\RequirePackage{amsthm,amsmath}
\RequirePackage[numbers]{natbib}
\RequirePackage[colorlinks,citecolor=blue,urlcolor=blue]{hyperref}

\usepackage{enumitem}
\usepackage{amsfonts,graphicx,bbm}
\usepackage{array,booktabs}
\usepackage[margin=1.25in]{geometry}

\def\P{{\mathbb P}}
\def\E{{\mathbb E}}

\newcommand{\var}{{\rm {Var}}}
\newcommand{\cov}{{\rm {Cov}}}

\def\N{\mathbb N}
\def\R{\mathbb R}

\usepackage{bm}

\newcommand\comment[1]{{}}

\DeclareMathOperator*{\argmax}{arg\,max}
\DeclareMathOperator*{\argmin}{arg\,min}

\startlocaldefs
\numberwithin{equation}{section}
\theoremstyle{plain}
\newtheorem{theorem}{Theorem}[section]
\newtheorem{proposition}{Proposition}[section]
\newtheorem{lemma}{Lemma}[section]
\newtheorem{corollary}{Corollary}[section]
\theoremstyle{remark}
\newtheorem{remark}{Remark}[section]
\newtheorem{definition}{Definition}[section]

\newtheorem{example}{Example}[section]
\endlocaldefs

\arxiv{1811.05124}

\begin{document}

\begin{frontmatter}
\title{FUNDAMENTAL LIMITS OF EXACT SUPPORT RECOVERY IN HIGH DIMENSIONS\thanksref{T1}}
\runtitle{EXACT SUPPORT RECOVERY IN HIGH DIMENSIONS}
\thankstext{T1}{This work was partially supported by the NSF ATD grant DMS--1830293.}

\begin{aug}
\author{\fnms{Zheng} \snm{Gao}\ead[label=e1]{gaozheng@umich.edu}}
\and
\author{\fnms{Stilian} \snm{Stoev}\ead[label=e2]{sstoev@umich.edu}}

\runauthor{Z. Gao and S. Stoev}

\affiliation{University of Michigan, Ann Arbor}

\address{Department of Statistics\\
University of Michigan\\
1085 S. University Ave.\\
Ann Arbor, MI, 48105\\
\printead*{e1}\\
\printead*{e2}}

\end{aug}

\begin{abstract}
%
%
\comment{ We study the problem of exact support recovery of high dimensional sparse vectors under dependence. We characterize a {\em strong classification} boundary which describes the required signal sizes, as a function of the sparsity level, in order for the support recovery to be asymptotically exact.  Specifically, we show that when the signal is above a so-called strong classification boundary, several classes of well-known procedures can achieve asymptotically perfect support recovery.  This is so under arbitrary error dependence assumptions, provided that the marginal error distribution has rapidly varying tails.  

 We then show that under a very general, but not arbitrary class of dependence structures, the said boundary is tight. In particular, under this general class of error dependence, if the signal sizes are below the strong classification boundary, no thresholding procedure can achieve asymptotically exact support recovery. The concept of uniform relative stability (URS) -- a type of concentration of maxima phenomenon --  plays a key role. A complete characterization of the URS property is given for Gaussian triangular arrays in terms of their covariance structure.  Examples of non-URS errors are also given, where exact support recovery is possible for strictly smaller signal sizes.

 Finally we demonstrate, perhaps surprisingly, that thresholding procedures may not be optimal in support recovery problems especially in the regime of heavy-tailed super-exponential error distributions.}
 
 %
 %
\comment{
We study the problem of the estimation of the support (set of non-zero components) of a sparse high-dimensional signal when observed with dependent noise. 
We characterize a phase-transition phenomenon for the asymptotically exact estimation of the signal support. 
Specifically, we show that when the signal size is above a so-called strong classification boundary, several classes of well-known procedures can achieve asymptotically perfect support recovery. 
This is so under arbitrary error dependence assumptions, provided that the marginal error distribution has rapidly varying tails.  
Conversely, we show that no thresholding estimators can achieve perfect support recovery if the signal is below the boundary, under very mild, but not arbitrary dependence structures on the noise. 
Examples of errors are given, where the dependence conditions are violated and exact support recovery is possible for strictly smaller signal sizes.

The results are based on an important and little-understood concentration of maxima phenomenon, known as relative stability in the case of independent errors. 
As an important probabilistic result, we establish the complete characterization of the relative stability phenomenon for dependent Gaussian triangular arrays. 
Consequently, we obtain a rather complete understanding of support recovery problem for sparse signals observed with dependent Gaussian noise. The methods of proof involve Sudakov-Fernique and Slepian lemma techniques alongside a curious application of Ramsey theory.

Finally we demonstrate, perhaps surprisingly, that thresholding procedures may not be optimal in support recovery problems especially in the regime of heavy-tailed super-exponential error distributions.
In the case of log-concave error densities, the thresholding estimators are shown to be minimax optimal and hence in this setting, the strong classification boundary is universal.}

%
%
We study the support recovery problem for a high-dimensional signal observed with additive noise.
With suitable parametrization of the signal sparsity and magnitude of its non-zero components, we characterize a phase-transition phenomenon akin to the signal detection problem studied by Ingster in 1998. 
Specifically, if the signal magnitude is above the so-called {\em strong classification 
boundary}, we show that several classes of well-known procedures achieve asymptotically perfect support recovery as the dimension goes to infinity. 
This is so, for a very broad class of error distributions with light, rapidly varying tails which may have arbitrary dependence.  
Conversely, if the signal is below the boundary, then for a very broad class of error dependence structures, no thresholding estimators (including ones with data-dependent thresholds) can achieve perfect support recovery.
The proofs of these results exploit a certain \emph{concentration of maxima} phenomenon known as relative stability. 
We provide a complete characterization of the relative stability phenomenon for Gaussian triangular arrays in terms of their correlation structure. 
The proof uses classic Sudakov-Fernique and Slepian lemma arguments along with a curious application of Ramsey's coloring theorem. 

We note that our study of the strong classification boundary is in a finer, point-wise, rather than minimax, sense. 
We also establish results on the finite-sample Bayes optimality and sub-optimality of thresholding procedures. 
Consequently, we obtain a minimax-type characterization of the strong classification boundary for errors with log-concave densities. 
\end{abstract}



\end{frontmatter}
\section{Introduction}
\label{sec:intro}

Consider the canonical signal-plus-noise model where the observation $x$ is a high-dimensional vector in $\R^p$,
\begin{equation} \label{eq:model}
    x = \mu + \epsilon.
\end{equation}
The signal, $\mu = (\mu(j))_{j=1}^p$, is a vector with $s$ non-zero components supported on the set $S=\{j:\mu(j)\neq 0\}$; the second term $\epsilon$ is a random error vector. 
The goal of high-dimensional statistics is usually two-fold: 
\begin{enumerate}
    \item To detect the presence of non-zero components in $\mu$. That is, to test the global hypothesis $\mu = 0$, which we call the \emph{detection} problem, and
    \item To estimate the support set $S$, which we call the \emph{support recovery} problem.
\end{enumerate}
An archetypal application where the two problems arise is in cyber security \citep{kallitsis2016amon}.
Internet service providers routinely collect statistics of server to determine if there are abnormal surges or blackouts.
While this monitoring is performed over a large number of servers, only very few servers are believed to be experiencing problems at any time.
The signal detection problem is then equivalent to determining if there are any anomalies among all servers; the support recovery problem is equivalent to identifying the servers with anomalies.

The same two questions are pursued in, for example, large-scale microarray experiments \citep{dudoit2003multiple}, brain imaging and fMRI analysis \citep{nichols2003controlling}, and numerous other applications.

A common theme in such applications is that the error terms are correlated, and that the signal vectors are believed to be sparse: the number of non-zero components in $\mu$ is small compared to the number of test performed.
Under such dependence and  sparsity assumptions, it is natural to ask if and when one can reliably {(1)} detect the signals, and {(2)} recover the support set $S$.
In this paper, we focus on the support recovery problem, and seek minimal conditions under which the support set can be consistently estimated.

\subsection{Set-up} \label{subsec:set-up}

We now describe the models and assumptions that we are going to work with for the rest of the paper.

\subsubsection{Assumptions on the error distributions}

We assume that the errors come from a triangular array
\begin{equation}
\label{eq:error-array}
{\cal E} = \left\{ (\epsilon_p(j))_{j=1}^p,\ p=1,2,\dots\right\},
\end{equation}
such that the $\epsilon_p(j)$'s have common distribution $F$.

We shall consider light-tailed error distributions with {\em rapidly varying} tails (see e.g.,\ Definition \ref{def:rapid-variation} below). 
To be concrete and better convey the main ideas, we will focus on the class of asymptotically generalized Gaussian laws (see Definition \ref{def:AGG}), which is still a fairly general class of models commonly used in the literature on high-dimensional testing \citep{cai2007estimation, arias2017distribution}.
Extensions to other models are deferred to Appendix \ref{sec:other-boundaries}.

\begin{definition} \label{def:AGG}
A distribution $F$ is called asymptotic generalized Gaussian with parameter $\nu>0$ (denoted $\text{AGG}(\nu)$) if
\begin{enumerate}
    \item $F(x)\in(0,1)$ for all $x\in\R$, and \smallskip
    \item $\log{\overline{F}(x)} \sim -\frac{1}{\nu}x^\nu$ and $\log{F(-x)} \sim -\frac{1}{\nu}(-x)^\nu,$ \label{eq:AGG}
\end{enumerate}
where $\overline{F}(x) = 1 - F(x)$ is the survival function, and $a(x)\sim b(x)$ is taken to mean $\lim_{x\to\infty} a(x)/b(x) = 1$.
\end{definition}

The AGG models include, for example, the Gaussian distribution ($\nu = 2$), and the Laplace distribution ($\nu = 1$) as special cases. 
Since the requirement is only placed on the tail behavior, this class encompasses a large variety of light-tailed models.

On the other hand, the AGG model are themselves special cases of a more general class of tail models as we will see in Example \ref{exmp:AGG}.
For simplicity of exposition, however, we shall focus on the $\text{AGG}(\nu)$ distributions, where the quantiles have explicit expressions.
\begin{proposition} \label{prop:quantile}
The $(1/p)$-th upper quantile of $\text{AGG}(\nu)$ is
\begin{equation} \label{eq:AGG-quantiles}
    u_{p} := F^\leftarrow(1-1/p) \sim \left(\nu\log{p}\right)^{1/\nu},
\end{equation}
where $F^\leftarrow(q) = \inf_x\{x:F(x)\ge q\},\ q\in (0,1)$.
\end{proposition}
The proof of Proposition \ref{prop:quantile} can be found in Appendix \ref{sec:AGG}.

Note that the errors are only assumed to have common marginal distributions, and may have potentially arbitrary dependence.
We will study the role of dependence in support recovery problems in Section \ref{sec:URS}. 

\subsubsection{Assumptions on the signals}
We assume in model \eqref{eq:model} that $\mu = \left(\mu(j)\right)_{j=1}^p$ is a sparse signal vector with non-zero entries only at the support of the signal $S_p\subseteq \{1,\ldots,p\}$. 
We denote the size of the support set as $s = |S_p|$, and assume that the non-zero entries of $\mu$ are positive and take values in the interval $\left[\underline{\Delta},\overline{\Delta}\right]\subset(0,\infty)$.

Following \citep{ingster1998minimax, donoho2004higher, cai2007estimation, haupt2011distilled, arias2017distribution}, we parametrize the signal sparsity as
\begin{equation} \label{eq:sparsity-parametrized}
    s = s(p) = \lfloor p^{1-\beta} \rfloor, 
\end{equation}
with $0 < \beta \le 1$ fixed. We also parametrize the signal sizes $\underline{\Delta}$ and $\overline{\Delta}$ as
\begin{equation} \label{eq:signal-size-parametrized}
    \underline{\Delta} = \underline{\Delta}(p) = (\nu \underline{r} \log{p})^{1/\nu} \quad \text{and} \quad
    \overline{\Delta} = \overline{\Delta}(p)  = (\nu \overline{r} \log{p})^{1/\nu},
\end{equation}
with parameters $0 < \underline{r} \le \overline{r}$.

\subsection{Thresholding procedures} \label{subsec:thresholding-procedures}

We review four procedures that will be analyzed in this paper.
All of them fall under the class of thresholding procedures, which we define as follows.
\begin{definition}[Thresholding procedures]
A thresholding procedure for estimating the support 
$S_p:=\{j\, :\, \mu(j)\neq0\}$ is one that takes on the form
\begin{equation} \label{eq:thresholding-procedure}
    \widehat{S}_p = \left\{j\, :\, x(j) > t_p(x)\right\}.
\end{equation}
We note here that the threshold $t_p(x)$ may depend on the data $x$.
\end{definition}
A well-known (deterministic) thresholding procedure is Bonferroni's procedure.
\begin{definition}[Bonferroni's procedure]
Suppose the errors $\epsilon(j)$'s have a common marginal distribution $F$, Bonferroni's procedure with family-wise error rate (FWER) at most $\alpha$ is the thresholding procedure that uses the threshold
\begin{equation} \label{eq:Bonferroni-procedure}
    t_p = F^{\leftarrow}(1 - \alpha/p).
\end{equation}
\end{definition}
A closely related procedure is Sid\'ak's procedure \citep{vsidak1967rectangular}
which is a more aggressive (and also deterministic) thresholding procedure that uses the 
threshold
\begin{equation} \label{eq:Sidak-procedure}
    t_p = F^{\leftarrow}((1 - \alpha)^{1/p}).
\end{equation}

Another procedure, which is strictly more powerful than Bonferroni's, is Holm's procedure \citep{holm1979simple}.
On observing the data $x$, its coordinates can be ordered from largest to smallest
$x(j_1) \ge x(j_2)  \ge \ldots \ge x(j_p)$,
where $(j_1, \ldots, j_p)$ is a permutation of $\{1, \ldots, p\}$. 
\begin{definition}[Holm's procedure]
Let $k$ be the largest index such that
$$
\overline{F}(x(j_i)) \le \alpha / (p-i+1),\quad \text{for all}\;i\le k.
$$
Holm's procedure with FWER controlled at $\alpha$ is the thresholding procedure that uses
\begin{equation} \label{eq:Holm-procedure}
    t_p(x) = x(j_{k}),
\end{equation}
\end{definition}
In contrast to the Bonferroni procedure, Holm's procedure is data-dependent.
A closely related, more aggressive (data-dependent) thresholding procedure is Hochberg's procedure \citep{hochberg1988sharper},
which replaces the index $k$ in Holm's with the largest index $i$ such that
$$
\overline{F}(x(j_i)) \le \alpha / (p-i+1).
$$

We will analyze the performance of these thresholding procedures in Section \ref{sec:sufficient}.
The (sub)optimality of general data-dependent thresholding procedures will be established in Section \ref{sec:optimality}.
We now return to the discussion of support recovery.

\subsection{The Hamming-loss approach and its deficiencies}
Recall our goal is to establish minimal conditions under which the support set can be consistently estimated, i.e.,
\begin{equation} \label{eq:goal}
    \P[\widehat{S}_p = S_p] \longrightarrow 1 \quad \text{as } \; p\to\infty, 
\end{equation}
where $\widehat{S}_p$ is an estimate of the true set of signal support $S_p$.
Previously, the probability of exact support recovery has been studied via the Hamming loss, defined as the number of mismatches between the estimated and true support set,
\begin{equation} \label{eq:Hamming-loss}
    H(\widehat S, S)
    = \left|\widehat S \setminus S\right| + \left|S \setminus \widehat S\right|
    = \sum\limits_{j=1}^p\left|\mathbbm{1}_{\widehat{S}}(j)- \mathbbm{1}_{S}(j)\right|.
\end{equation}
As pointed out in, e.g. \citet{butucea2018variable}, there is a natural lower bound for the probability of support recovery by the expected Hamming loss,
\begin{equation} \label{eq:Hamming-loss-lower-bound}
    \mathbb{P}[\widehat S = S] 
    \ge 1 - \mathbb{E}[H(\widehat S, S)]
    = 1 - \sum\limits_{j=1}^p\E\left|\mathbbm{1}_{\widehat{S}}(j)- \mathbbm{1}_{S}(j)\right|.
\end{equation}
However, vanishing Hamming loss is only sufficient, not necessary for support recovery \eqref{eq:goal}.
A key observation in Relation \eqref{eq:Hamming-loss-lower-bound} is that the expected Hamming loss decouples into $p$ terms, and dependence of the estimates $\mathbbm{1}_{\widehat{S}}(j)$ among the $p$ locations no longer plays a role in the sum.
Therefore, studying support recovery problems via the expected Hamming loss is not very informative especially under severe dependence, as the bound \eqref{eq:Hamming-loss-lower-bound} may become \emph{very} loose.

Despite this deficiency, some important results have been derived for the independent or near-independent case.
For example, \citet{genovese2012comparison} and \citet{ji2012ups} studied the problem of support recovery in linear models under the Hamming loss.
Under the parametrization in \eqref{eq:sparsity-parametrized} and \eqref{eq:signal-size-parametrized}, a minimax-type phase transition result was established. In \citep{ji2012ups}, it was shown that the  Hamming loss diverges to $+\infty$ when signal sizes $r$ fall below the threshold
\begin{equation} \label{eq:strong-classification-boundary-Gaussian}
    g(\beta) = (1 + (1 - \beta)^{1/2})^2,
\end{equation}
for any method of support estimation.
Conversely, under orthogonal, or near-orthogonal random designs, if $r>g(\beta)$, their proposed method achieves vanishing Hamming loss.

Very recently, \citet{butucea2018variable} studied both asymptotics and non-asymptotics of the support recovery problem in the additive noise model \eqref{eq:model} with Gaussian errors under the Hamming loss.
It was again shown that the boundary \eqref{eq:strong-classification-boundary-Gaussian} exists in a minimax sense.
That is, when errors are \emph{independent}, the Hamming loss cannot be made to vanish if signal sizes fall below the boundary \eqref{eq:strong-classification-boundary-Gaussian}. 
Conversely, if $r>g(\beta)$, the Hamming loss can be made to vanish with a particular thresholding procedure.


So far in the literature, the role of dependence, and that of the distributional assumptions in the exact support recovery problem remain largely unexplored.  
This is arguably a consequence of the inherent limitation of the Hamming-loss approach. 
The minimax formulation, although elegant, focuses on the independent cases and obviates a fuller exploration of the phenomenon under various dependence conditions.
For practitioners, a minimax statement in the Gaussian case seems to offer little guidance.
Indeed, in applications, independence is an exception rather than the rule; Gaussianity of the errors may also be unnecessarily restrictive.

These considerations motivate us to study the support recovery problem \eqref{eq:goal} directly, under general distributional and dependence assumptions.
In particular, it is of interest to see if the phase transition continues to hold when we shift our focus from the Hamming loss to the exact support recovery problem \eqref{eq:goal}.
It is also important to know how dependence may affect the problem of support recovery, and when boundaries such as $g$ in \eqref{eq:strong-classification-boundary-Gaussian} remain tight under dependent, and possibly non-Gaussian errors.
We provide answers to these questions in this paper. 
Our contributions are summarized in Subsections \ref{subsec:role-of-dependence}, \ref{subsec:role-of-URS}, and \ref{subsec:role-of-thresholding} below.

\subsection{Role of dependence and point-wise phase transitions}
\label{subsec:role-of-dependence}
Consider the function
\begin{equation} \label{eq:strong-classification-boundary}
    g(\beta) = (1 + (1 - \beta)^{1/\nu})^\nu,
\end{equation}
which we refer to as the {\em strong classification boundary}.
In Theorem \ref{thm:sufficient} we show that, under the general distributional assumptions in Section \ref{subsec:set-up}, if the signal size is above the boundary (i.e., $\underline{r}> g(\beta)$), procedures described in Section \ref{subsec:thresholding-procedures} with appropriately calibrated levels achieve perfect support recovery, that is,
\begin{equation} \label{eq:exact-recovery}
    \mathbb{P}\left[\widehat{S}_p=S_p\right]\longrightarrow 1,\quad \mbox{ as }p\to \infty.
\end{equation}

Conversely, we show in Theorem \ref{thm:necessary}, that for a surprisingly large class of dependence structures characterized by the concept of \emph{uniform relative stability} (URS, see Definition \ref{def:URS}), when the signal size is below the boundary  (i.e., $r<g(\beta)$), no thresholding procedure can achieve the asymptotically perfect support recovery \eqref{eq:exact-recovery}. In fact,
\begin{equation} \label{eq:exact-recovery-failure}
    \mathbb{P}\left[\widehat{S}_p=S_p\right]\longrightarrow 0,\quad \mbox{ as }p\to \infty,
\end{equation}
for all $\widehat{S}_p$ in the form of \eqref{eq:thresholding-procedure}.

Complementing results in \citet{butucea2018variable}, these two results show that the thresholding procedures obey a phase transition phenomenon in a strong, \emph{point-wise} sense over the class of URS dependence structures, and over the class of AGG$(\nu),\ \nu>0$ error distributions. 

The strong classification boundary $g$ characterizes a phase-transition phenomenon similar to that of the signal detection and approximate support recovery. 
A preview of this result is presented in Figure \ref{fig:phase}, which shows the boundaries for signal detection, approximate support recovery, and exact support recovery, in the special case of Gaussian and Laplace errors, respectively. 

\subsection{Uniform relative stability and concentration of maxima}
\label{subsec:role-of-URS}
The key probabilistic concept behind our characterization of the dependence structure under which \eqref{eq:exact-recovery-failure} takes place is a certain {\em concentration of maxima} phenomenon, known as {\em relative stability} (see Section \ref{sec:URS}).
We introduce and study an extension of this concept referred to as uniform relative stability (URS). 
Broadly speaking, the strong classification boundary phenomenon holds for dependent light-tailed errors, provided that they are uniformly relatively stable. 

In the case of dependent Gaussian errors, we establish in Theorem \ref{thm:Gaussian-weak-dependence} a complete characterization of URS in terms of a simple condition on the covariance structure, using the Sudakov-Fernique bounds and a curious application of Ramsey theory.
These results may be of independent interest, since they characterize Gaussian triangular arrays whose maxima concentrate at the same rate as in the case of independence.

\begin{figure}
      \centering
      \includegraphics[width=0.45\textwidth]{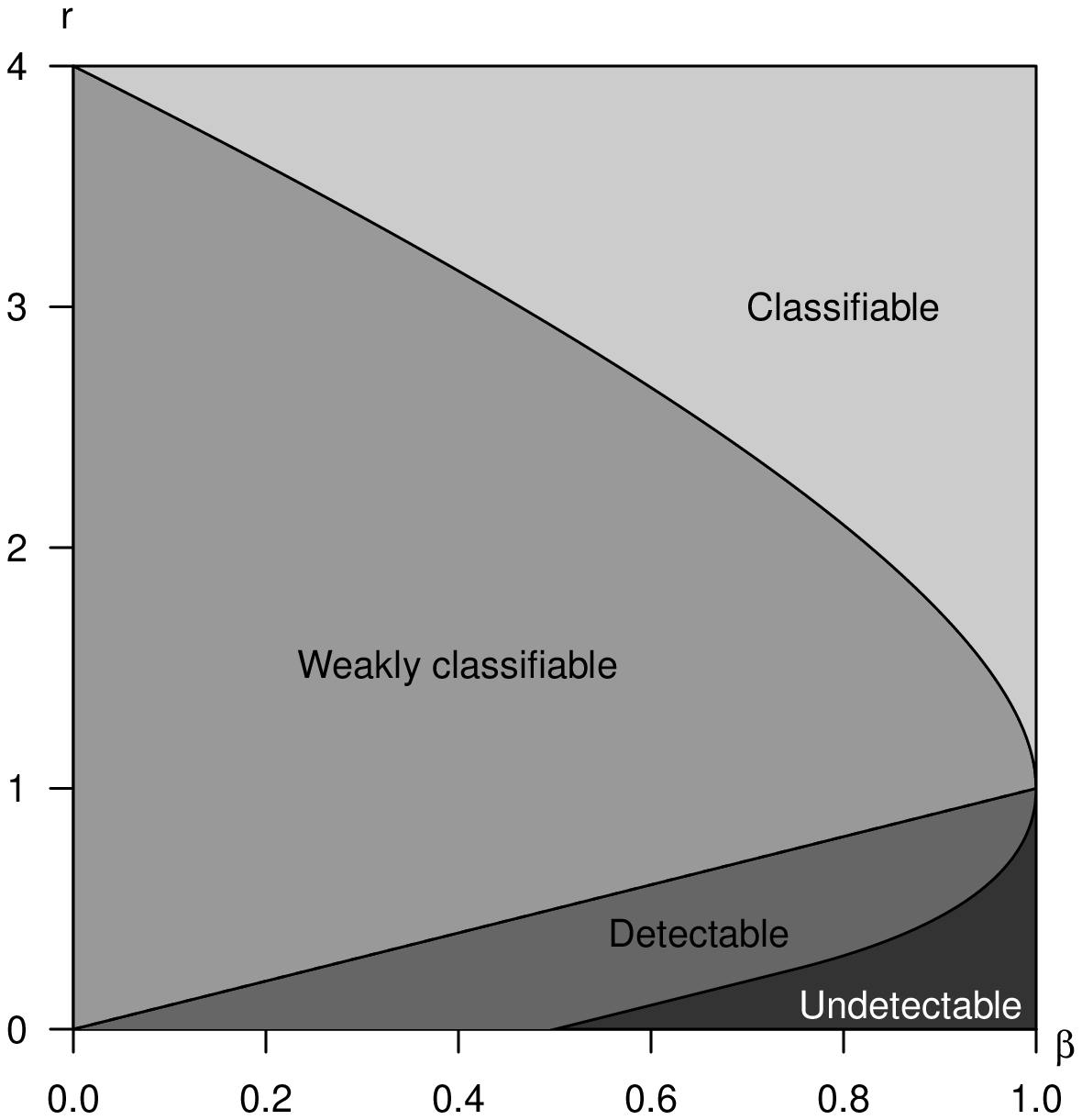}
      \includegraphics[width=0.45\textwidth]{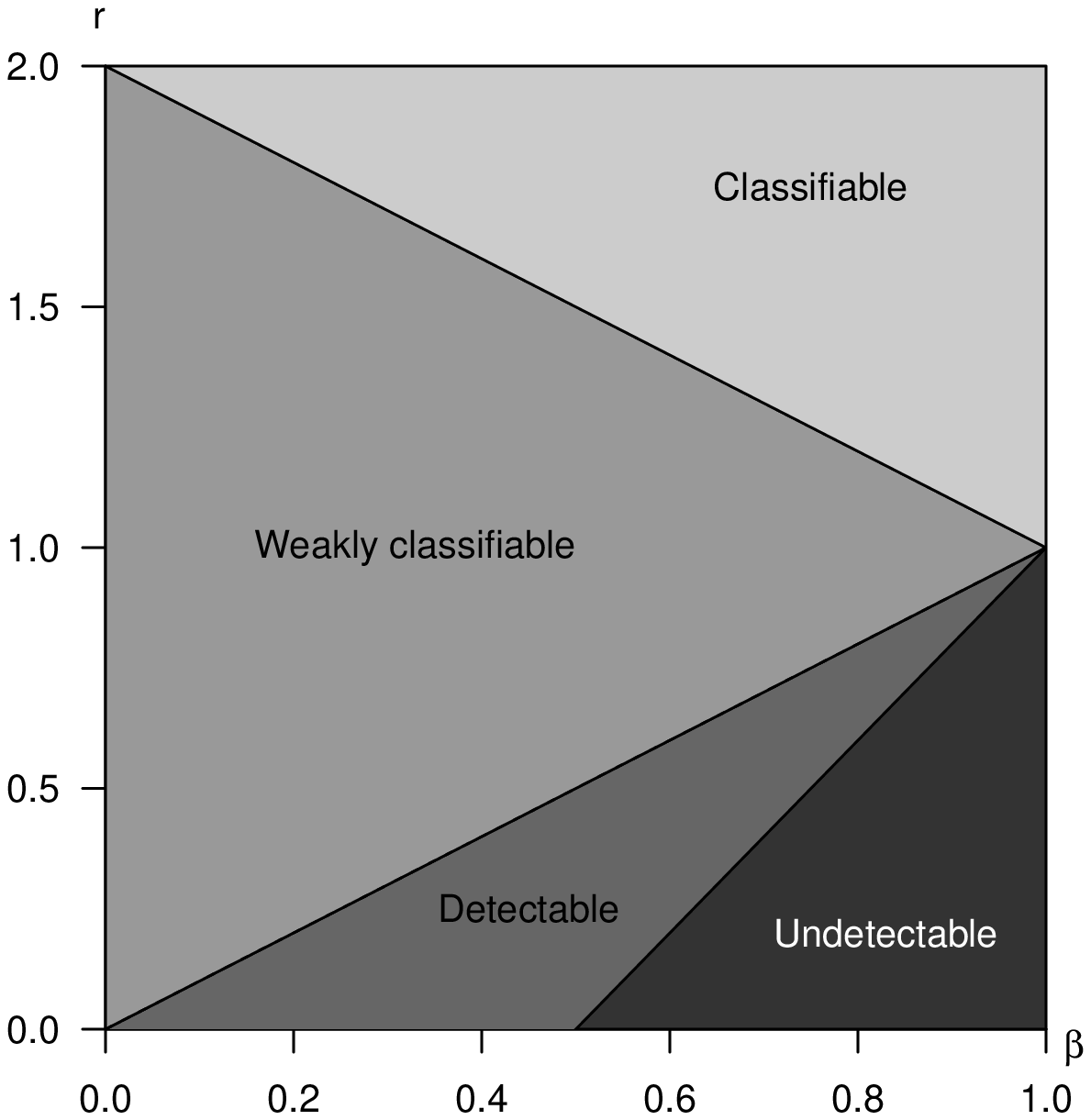}
      \caption{The phase diagrams of the detection, weak classification, and strong classification boundaries against sparse alternatives under Gaussian (left) and Laplace distributed (right) errors. The parameters $\beta$ and $r$ parametrize the signal sparsity and signal size, respectively.
      We study in this paper the strong classification boundary, above which the support recovery can be achieved \emph{exactly} in the \emph{Classifiable} region. In a large class of dependence structures characterized by URS, when signal sizes fall below the strong classification boundary, no thresholding procedure succeeds in the exact support recovery problem.
      For the detection and weak classification boundaries, see, e.g., \citep{haupt2011distilled, arias2017distribution, ingster1998minimax, donoho2004higher}, and Sec \ref{sec:sufficient}.}
      \label{fig:phase}
\end{figure}

\subsection{Role of thresholding procedures and minimax optimality}
\label{subsec:role-of-thresholding}
We show in Section \ref{sec:optimality} that data thresholding procedures are optimal when the errors are independent and identically distributed (i.i.d.) with log-concave density. 
In this case, no estimator can achieve perfect support recovery when the signal is below the strong classification boundary \eqref{eq:strong-classification-boundary}. 
Consequently, in the case of AGG errors with $\nu\ge 1$, the strong 
classification boundary is shown to hold in the minimax 
sense for \emph{all} procedures. This is formalized in Theorem \ref{thm:necessary-strengthened}. 

We show that thresholding procedures (including data-dependent ones) are not optimal in general in the support recovery problem when the errors have heavy, regularly-varying tails. 
In this case, we also demonstrate the absence of a phase-transition phenomenon in exact support recovery by thresholding, in Theorem \ref{thm:heavy-tails} in the appendix.



\subsection{Contents of this paper}
The results summarized in Subsections \ref{subsec:role-of-dependence}, \ref{subsec:role-of-URS}, and \ref{subsec:role-of-thresholding} are detailed in Sections \ref{sec:sufficient}, \ref{sec:URS}, \ref{sec:necessary}, and \ref{sec:optimality}. Numerical illustrations are given in Section \ref{sec:numerical}. 
Appendices \ref{sec:proofs} and \ref{sec:AGG} contain technical proofs and auxiliary results.
Generalizations of the phase transition phenomena to other classes of error distributions can be found in Appendix \ref{sec:other-boundaries}.
Performance of thresholding procedures under heavy, regularly varying tails is analyzed in Appendix \ref{sec:heavy-tailed}.

\section{Sufficient conditions for exact support recovery}
\label{sec:sufficient}

Our first result shows that, if $F\in \text{AGG}(\nu)$ with $\nu>0$, then regardless of the error dependence structure, (asymptotic) perfect support recovery is achieved by applying Bonferroni's procedure with appropriately calibrated FWER, as long as the signal sizes are above the boundary in \eqref{eq:strong-classification-boundary}.

\begin{theorem} \label{thm:sufficient}
Suppose $F\in \text{AGG}(\nu)$ with $\nu>0$;
let the signal $\mu$ have $|S_p| = \lfloor p^{(1-\beta)} \rfloor$ non-zero entries where $\beta\in(0,1]$, where the magnitudes of non-zero signal entries are at least $\underline{\Delta} = \left(\nu\underline{r}\log p\right)^{1/\nu}$.
Let also $\widehat{S}_p$ be the Bonferroni's procedure (defined in \eqref{eq:Bonferroni-procedure}) with vanishing FWER $\alpha = \alpha(p) \to 0$, such that 
$\alpha p^\delta\to \infty$ for every $\delta>0$.
If
\begin{equation} \label{eq:signal-above-boundary}
    \underline{r} > g(\beta) = (1 + (1 - \beta)^{1/\nu})^\nu,
\end{equation}
then we have
\begin{equation} \label{eq:exact-supporot-recovery}
    \lim_{p\to\infty}\mathbb P[\widehat{S}_p = S_p] = 1.
\end{equation}
\end{theorem}
\begin{corollary}[Classes of procedures attaining the boundary]
\label{cor:FWER-controlling_procedures}  Relation \eqref{eq:exact-supporot-recovery} holds for any FWER-controlling procedure that is strictly more powerful than Bonferroni's procedure. 
This includes Holm's procedure \citep*{holm1979simple}, and in the case of independent errors, Hochberg's procedure \citep*{hochberg1988sharper}, and the {\v{S}}id{\'a}k procedure \citep*{vsidak1967rectangular}.
\end{corollary}

\begin{example} \label{exmp:FWER-controlling_procedures}
Under Gaussian errors, the particular choice of the thresholding at $t_p = \sqrt{2\log{p}}$ in \eqref{eq:Bonferroni-procedure} corresponds to a Bonferroni's procedure with FWER decreasing at a rate of $(\log{p})^{-1/2}$. By the Corollary \ref{cor:FWER-controlling_procedures}, Holm's procedure, and when the errors are independent, the {\v{S}}id{\'a}k, and Hochberg procedures with FWER controlled at $(\log{p})^{-1/2}$ all achieve perfect support recovery provided that $r>g(\beta)$.
\end{example}

The claims in Example \ref{exmp:FWER-controlling_procedures} are proved in Appendix \ref{subsec:proofs-examples}.
We now turn to the proof of Theorem \ref{thm:sufficient}.

\begin{proof}[Proof of Theorem \ref{thm:sufficient}]
Under the AGG model, it is easy to see from equation \eqref{eq:AGG-quantiles} that the thresholds in Bonferroni's procedure are 
\begin{equation}\label{e:AGG-threshold}
t_p = F^{\leftarrow}(1 - \alpha/p) = (\nu\log{(p/\alpha)})^{1/\nu}(1+o(1)).
\end{equation}
Define $\widehat{S}_p = \left\{j:x(j)>t_p\right\}$ as our estimator for the support set. 
Dependence on $p$ will be suppressed to simplify notations when such omissions do not lead to ambiguity.

The Bonferroni's procedure controls the FWER.  Indeed, we have
\begin{align}\label{eq:Bonferroni-FWER-control}
    \P\left[\widehat{S} \subseteq S\right] 
        &= 1 - \P\left[\max_{j\in S^c}x(j) > t_p\right] = 1 - \P\left[\max_{j\in S^c}\epsilon(j) > t_p\right]\nonumber \\
        &\ge 1 - \sum_{j\in\{1,\ldots,p\}}\P\left[\epsilon(j)>t_p\right] \ge 1 - \alpha(p) \to 1,
\end{align}
where we used the union bound in the second inequality. This shows that the probability of no false inclusion converges to $1$.

On the other hand, for the probability of no missed detection, we have:
\begin{equation*}
    \P\left[\widehat{S} \supseteq S\right] 
        = \P\left[\min_{j\in S}x(j) > t_p\right]
        = \P\left[\min_{j\in S}x(j) - (\nu\underline{r}\log p)^{1/\nu} > t_p - (\nu\underline{r}\log p)^{1/\nu} \right].
\end{equation*}
Since the signal sizes are no smaller than $(\nu\underline{r}\log p)^{1/\nu}$, we have
\begin{equation*}
    x(j) - \left(\nu\underline{r}\log{p}\right)^{1/\nu} \ge \epsilon(j), \quad \text{for all }j\in S,
\end{equation*}
and hence we obtain
\begin{equation} \label{eq:sufficient-proof-eq1}
    \P\left[\widehat{S} \supseteq S\right] \ge 
    \P\left[\min_{j\in S}\epsilon(j) > (\nu\log{(p/\alpha)})^{1/\nu}(1+o(1)) - (\nu\underline{r}\log p)^{1/\nu} \right],
\end{equation}
where we plugged in the expression for $t_p$ in \eqref{e:AGG-threshold}.
Now, since the minimum signal size is bounded below by $\underline{r} > \left(1 + (1-\beta)^{1/\nu}\right)^\nu$, we have $\underline{r}^{1/\nu}-(1-\beta)^{1/\nu}>1$, and so we can pick a $\delta > 0$ such that 
\begin{equation} \label{eq:choice-of-delta}
    \delta < \left(\underline{r}^{1/\nu} - (1-\beta)^{1/\nu}\right)^\nu - 1.
\end{equation}
Since, by assumption, for all $\delta>0$, we have $p^{-\delta} = o\left(\alpha(p)\right)$, there is an
$M=M(\delta)$ such that $p/\alpha(p) < p^{1+\delta}$ for all $p\ge M$. Thus, from \eqref{eq:sufficient-proof-eq1}, we further conclude that for $p\ge M$, and $\beta\in(0,1)$,
\begin{align}
    \P\Big[\widehat{S} \supseteq S\Big]
      &\ge \P\Big[\min_{j\in S}\epsilon(j) > \left((1+\delta)\nu\log{p}\right)^{1/\nu}(1+o(1)) - (\nu\underline{r}\log p)^{1/\nu} \Big] \nonumber \\
      &= \P\Big[\max_{j\in S}\left(-\epsilon(j)\right) < \left(\underline{r}^{1/\nu} - (1+\delta)^{1/\nu}\right)(\nu\log{p})^{1/\nu}(1+o(1)) \Big]. \label{eq:sufficient-proof-eq2} \\
      &= \P\Big[\frac{\max_{j\in S}(-\epsilon(j))}{u_{|S|}} < \underbrace{\frac{\underline{r}^{1/\nu} - (1+\delta)^{1/\nu}}{(1-\beta)^{1/\nu}}\left(1+o(1)\right)}_{=:\text{A}}\Big], \label{eq:sufficient-proof-eq3}
\end{align}
where in the last line we used the parametrization $|S_p| = \lfloor p^{1-\beta}\rfloor$, and the expression for the quantiles \eqref{eq:AGG-quantiles} to obtain
\begin{equation} \label{eq:sufficient-proof-eq4}
    u_{|S_p|} = u_{\lfloor p^{1-\beta}\rfloor} \sim \left(\nu(1-\beta)\log{p}\right)^{1/\nu}.
\end{equation}
Observe that for the expression $\text{A}$ in the right-hand-side of \eqref{eq:sufficient-proof-eq3} we have $\text{A}\to c>1$ by our choice of $\delta$ in \eqref{eq:choice-of-delta}.
Since the $-\epsilon(j)$'s are also AGG, by Proposition \ref{prop:rapid-varying-tails} (to be stated next in Section \ref{sec:URS}), we conclude that the last probability in \eqref{eq:sufficient-proof-eq3}, and hence,
$\P[\widehat{S} \supseteq S]$ converges to 1. 

Notice that Relation \eqref{eq:sufficient-proof-eq4}, and hence \eqref{eq:sufficient-proof-eq3} is only valid for $\beta\in(0,1)$.
However, the case where $\beta=1$ is easily handled: 
since the right-hand-side of \eqref{eq:sufficient-proof-eq2} diverges to $+\infty$, while the left-hand-side is tight, the probability again converges to 1.
\end{proof}

Before further commenting on the implications of Theorem \ref{thm:sufficient}, we shall review some related work on sparse signal detection and approximate support recovery.

The problem of sparse signal detection was first studied by \citet*{ingster1998minimax} under independent Gaussian errors.
Specifically, under the parametrization \eqref{eq:sparsity-parametrized} and \eqref{eq:signal-size-parametrized} (with $\nu=2$), it was shown that the detection problem can be answered perfectly as $p\to\infty$, when signal size $r$ is above a threshold $f(\beta)$, where
\begin{equation} \label{eq:detection-boundary}
f(\beta) = \begin{cases}
\left(1 - \sqrt{1 - \beta}\right)^2 &,\ \beta\ge3/4\\
\beta - 0.5 &,\ 1/2<\beta<3/4.
\end{cases}
\end{equation}
Conversely, when $r<f(\beta)$, we can do no better than random guessing. 
Thus, the function $f(\beta)$ fully characterizes the so-called \emph{detection boundary} of the signal detection problem, and demonstrates a phase-transition phenomenon. 
\citet*{donoho2004higher} showed that the Higher Criticism statistic, originally proposed by Tukey, is a procedure that achieves the detection boundary asymptotically, without prior knowledge of the sparsity and the signal size.

In this context of approximate support recovery, the corresponding concept to type I error is the \emph{False Discovery Proportion} (FDP), defined as $\text{FDP} = |\widehat{S}\setminus S|\big/|\widehat{S}|$ (with \emph{False Discovery Rate} (FDR) being the expectation of FDP), 
and the counterpart of type II error in this context being the \emph{False Non-discovery Proportion} (FNP), defined $\text{NDP} = |S\setminus\widehat{S}|\big/|S|$.
\citet*{haupt2011distilled} showed that for signal size $r > \beta$, the sum of FDP and FNP can be made to vanish as $p\to\infty$; while if $r < \beta$, no thresholding procedure can succeed asymptotically. We shall call this boundary 
\begin{equation} \label{eq:weak-classification-boundary}
    h(\beta) = \beta
\end{equation}
the \emph{weak classification boundary}. Recently, \citet*{arias2017distribution} showed that the Benjamini-Hochberg procedure \citep*{benjamini1995controlling} and the Barber-Cand{\`e}s procedure \citep*{barber2015controlling} are practical procedures that achieve this weak classification boundary.
These procedures are special cases of \emph{thresholding procedures}.

An even more stringent notion of false discovery is the \emph{Family-Wise Error Rate} (FWER), defined as $1 - \mathbb P[\widehat{S} \subseteq S]$, which is probability of falsely reporting any signal not in the support set.
Observe that vanishing FDP and NDP does not imply $\P\left[\widehat{S} = S\right]\to 1$;
the latter, stronger, notion of set consistency requires that the FWER vanish as well.

\begin{remark}[Gap between FDR and FWER under sparsity] \label{rmk:gap-when-signal-sparse}
Although it is believed that FWER control is sometimes a requirement too stringent compared to, say, FDR control in support recovery problems, 
the fact that all three thresholds (detection, weak, and strong classification) involve the same scaling indicates that the difficulties of the three problems (signal detection, approximate, and exact support recovery) are comparable when signals are truly sparse.
This is illustrated with the next example.
\end{remark}

\begin{example}[Power analysis for variable selection] \label{exmp:gap-when-signal-sparse}
For Gaussian errors (AGG with $\nu = 2$), when $\beta = 3/4$, for signal detection the boundary \eqref{eq:detection-boundary} says that signals will have to be at least of magnitude $\sqrt{\log{p}/2}$, 
while approximate support recovery \eqref{eq:weak-classification-boundary} requires signal sizes of at least $\sqrt{3\log{p}/2}$, 
and exact support recovery \eqref{eq:strong-classification-boundary} calls for signal sizes of at least $\sqrt{9\log{p}/2}$. 

If $m$ (independent) observations $x_1,\ldots,x_m$ were made on the same set of $p$ locations, then by taking location-wise averages, $$\overline{x}_{m}(j) = \frac{1}{m}\sum_{i=1}^{m} x_i(j), \quad j=1,\ldots,p$$
we can reduce error standard deviation, and hence boost the signal-to-noise ratio by a factor of $\sqrt{m}$.
By the simple calculations above, if one were to perform a power analysis to determine the sample size needed to detect (sparse) signals of a certain magnitude, increasing the sample size by a factor of $3$ will enable approximate support recovery with FDR control; and in fact, exact support recovery with FWER control can be achieved by increasing the sample size by a factor of $9$.
\end{example}

\begin{remark}[Gap between FDR and FWER under dense signals]
If the signals are only \emph{approximately} sparse, i.e., having a few components above \eqref{eq:strong-classification-boundary-Gaussian} but many smaller components above \eqref{eq:weak-classification-boundary}, then FDR-controlling procedures will discover substantially larger proportion of signals than FWER-controlling procedures.

Indeed, as $\beta\to0$, the required signal size for approximate support recovery \eqref{eq:weak-classification-boundary} tends to 0, while the required signal size for exact support recovery \eqref{eq:strong-classification-boundary-Gaussian} tends to $2^\nu$ in AGG models.
While Example \ref{exmp:gap-when-signal-sparse} indicates that the exact support recovery is not much more stringent than approximate support recovery when signals are sparse, the gap between required signal sizes widens when signals are dense. 
\end{remark}

Finally, we emphasize that Theorem \ref{thm:sufficient} holds for errors with \emph{arbitrary} 
dependence structures. Intuitively, this is because the maxima of the errors grow at their 
fastest in the case of independence. Formally, the result stems from Proposition \ref{prop:rapid-varying-tails} below, which is valid under arbitrary dependence.
The relationship between dependence and the behavior of maxima is discussed next in Section \ref{sec:URS}, where we will see that the phase-transition phenomenon is not limited to just the AGG models and independent errors; such phenomenon exists for all error models with rapidly varying tails, and under a surprisingly large class of dependent structures.

\section{Uniform relative stability and dependent errors}
\label{sec:URS}

We study the asymptotic behavior of error maxima under dependence in this section.
We show in Proposition \ref{prop:rapid-varying-tails} below that the maxima of errors with rapidly varying tails can be bounded above using quantiles of their marginal distribution, regardless of their dependence structure. 
This result is a key step in the proof of Theorem \ref{thm:sufficient}.

On the other hand, a lower bound for the error maxima can be provided for a very general class of dependence structures, as we will see in Section \ref{subsec:URS}. 
Specifically, we characterized this class of dependence structure in the case of Gaussian errors with a transparent necessary and sufficient condition.
This will prepare us to state the converse of Theorem \ref{thm:sufficient}.

Some new results regarding the structure of correlation matrix of high-dimensional random variables may be of independent interest, and are collected in Section \ref{subsec:Ramsey}.

\subsection{Rapid variation and relative stability}
\label{subsec:RS}

The behavior of the maxima has been well-studied in the literature (see, e.g., \cite{leadbetter2012extremes,resnick2013extreme,embrechts2013modelling,de2007extreme} 
and the references therein). The concept of rapid variation plays an important role in the light-tailed case.

\begin{definition}[Rapid variation] \label{def:rapid-variation}
The survival function of a distribution, $\overline{F}(x) = 1 - F(x)$, is said to be rapidly varying if
\begin{equation}\label{e:def:rapid-variation}
\lim_{x\to\infty} \frac{\overline{F}(tx)}{\overline{F}(x)} 
    = \begin{cases}
    0, & t > 1\\
    1, & t = 1\\
    \infty, & 0 < t < 1
\end{cases}.
\end{equation}
\end{definition}

When $F(x)<1$ for all finite $x$, \citet{gnedenko1943distribution} showed that the distribution $F$ has rapidly varying tails if and only if the maxima of independent observations from $F$ are \emph{relatively stable} in the following sense.
\begin{definition}[Relative stability] \label{def:RS}
Let $\epsilon_p = \left(\epsilon_p(j)\right)_{j=1}^p$ be a sequence of random variables with identical marginal distributions $F$. Define the sequence $(u_p)_{p=1}^\infty$ to be the $(1-1/p)$-th generalized quantile of $F$, i.e., 
\begin{equation} \label{eq:quantiles}
    u_p = F^\leftarrow(1 - 1/p).
\end{equation}
The triangular array ${\cal E} = \{\epsilon_p, p\in\N\}$ is said to have relatively stable (RS) maxima if we have
\begin{equation} \label{eq:RS-condition}
    \frac{1}{u_{p}} M_p := \frac{1}{u_{p}} \max_{j=1,\ldots,p} \epsilon_p(j) \xrightarrow{\P} 1,
\end{equation}
as $p\to\infty$.
\end{definition}

In the case of independent and identically distributed $\epsilon_p(j)$'s, \citet{barndorff1963limit} and \citet{resnick1973almost} obtained necessary and sufficient conditions for the \emph{almost sure stability} of maxima, where the convergence in \eqref{eq:RS-condition} holds almost surely.

While relative stability (and almost sure stability) is well-understood in the independent case, the role of dependence has not been fully explored.
We start this exploration with a small refinement of Theorem 2 in \citet{gnedenko1943distribution} valid under arbitrary dependence.

\begin{proposition}[Rapidly variation and relative stability] \label{prop:rapid-varying-tails}
Assume that the array ${\cal E}$ consists of identically distributed random 
variables with distribution $F$, where $F(x)<1$ for all finite $x>0$. 
\begin{enumerate}
    \item If $F$ has rapidly varying right tail, then for all $\delta>0$,
        \begin{equation} \label{eq:rapid-varying-tails}
            \P\left[\frac{1}{u_p} M_p\le1+\delta\right] \to 1.
        \end{equation}
    \item If in addition, the array ${\cal E}$ has independent entries, then it is relatively stable if and only if $F$ has rapidly varying tail.
    \label{prop:rapid-varying-tails_part-ii}
\end{enumerate}
\end{proposition}

\begin{proof}[Proof of Proposition \ref{prop:rapid-varying-tails}] 
By the union bound and the fact that 
$p\overline F(u_p) \le 1$, we have
\begin{align}\label{e:prop:rapid-varying-tails_part-i-1}
\P [ M_p > (1+\delta)u_p] \le p \overline F((1+\delta)u_p)
 \le \frac{\overline F((1+\delta)u_p)}{\overline F(u_p)}.
\end{align}
In view of \eqref{e:def:rapid-variation} (rapid variation) and
since $u_p\to\infty$, as $p\to\infty$, the right-hand side of \eqref{e:prop:rapid-varying-tails_part-i-1} vanishes 
as $p\to\infty$, for all $\delta>0$.  This completes the proof of \eqref{eq:rapid-varying-tails}. Part 2 is a re-statement of the classic result due to Gnedenko in \cite{gnedenko1943distribution}.
\end{proof}

We will see that Gaussian, Exponential, Laplace, Gamma (Example \ref{exmp:AGG}), log-normal (Example \ref{exmp:heavier-than-AGG}), and Gompertz (Example \ref{exmp:lighter-than-AGG}) distributions all have rapidly varying tails. On the other hand, heavy-tailed distributions like the Pareto and t-distribution do not.

%

\begin{example}[Generalized AGG] \label{exmp:AGG}
A distribution is said to have \emph{Generalized AGG} right tail, if $\log{\overline{F}}$ is regularly varying,
\begin{equation} \label{eq:GAGG}
    \log{\overline{F}(x)} = - x^\nu L(x),
\end{equation}
where $\nu>0$ and $L: (0,+\infty)\to(0,+\infty)$ is a slowly varying function. (A function is said to be slowly varying if $\lim_{x\to\infty}L(tx)/L(x) = 1$ for all $t>0$.) Note that the AGG model corresponds to the special case where $L(x)\to 1/\nu$, as $x\to\infty$.

Relation \eqref{eq:rapid-varying-tails} holds for all arrays $\cal E$ with \emph{generalized} AGG marginals; if the entries are independent, the maxima are relatively stable. 
This follows directly from Proposition \ref{prop:rapid-varying-tails}, once we show that $F$ has rapidly varying tail. 
Indeed, by \eqref{eq:GAGG}, we have
$$
\log{\left(\overline{F}(tx)\Big/ \overline{F}(x)\right)} = - L(x)x^\nu\left(t^\nu\frac{L(tx)}{L(x)} - 1\right),
$$
which converges to $-\infty$, 0, and $+\infty$, as $x\to\infty$, when $t>1$, $t=1$, and $t<1$, respectively.
\end{example}

The AGG class encompasses a wide variety of rapidly varying tail models such as Gaussian and double exponential distributions. The larger class \eqref{eq:GAGG} is needed, however, for the Gamma distribution.
Models such as Log-normal have heavier tails than the AGG model, yet still have rapidly varying tails. Therefore Proposition \ref{prop:rapid-varying-tails} is also applicable.

\begin{example}[Heavier than AGG] \label{exmp:heavier-than-AGG}
Let $\gamma>1$, $c>0$, and suppose that
\begin{equation} \label{eq:heavier-than-AGG}
    \log{\overline{F}(x)} = - \left(\log x\right)^\gamma \left(c+M(x)\right),
\end{equation}
where $\lim_{x\to\infty} M(x)\log{x}= 0$. Then, Relation \eqref{eq:rapid-varying-tails} holds under 
model \eqref{eq:heavier-than-AGG}. Further, if the entries in the array are independent, the 
maxima are relatively stable.

The behavior of the quantiles $u_p$ in this model is as follows. As $p\to\infty,$
\begin{equation*}
    u_p \sim \exp{\left\{\left(c^{-1}\log{p}\right)^{1/\gamma}\right\}}
    \iff c\left(\log{u_p}\right)^{\gamma} + o(1) = \log(p) = - \log \overline{F}(u_p).
\end{equation*}
since $u_p$ diverges, and $M(u_p)$ is $o((\log u_p)^{-1})$.
\end{example}

Lastly, Proposition \ref{prop:rapid-varying-tails} applies to error models with lighter tails than the AGG class.

\begin{example}[Lighter than AGG] \label{exmp:lighter-than-AGG}
With $\nu>0$, and $L(x)$ a slowly varying function, the class of distributions
\begin{equation} \label{eq:lighter-than-AGG}
    \log{\overline{F}(x)} = - \exp{\left\{x^\nu L(x)\right\}},
\end{equation}
is rapidly varying.
The quantiles can be derived explicitly in a subclass of \eqref{eq:lighter-than-AGG} where $L(x)\to 1$, or equivalently, when $\log{|\log{\overline{F}(x)}|}\sim x^\nu$,
\begin{equation*}
    u_p \sim \left(\log \log{p}\right)^{1/\nu}
    \iff \exp{\left\{u_p^\nu\left(1+o(1)\right)\right\}} = \log(p) = - \log \overline{F}(u_p).
\end{equation*}
\end{example}
Strong classification boundaries for the classes of models in Examples \ref{exmp:heavier-than-AGG} and \ref{exmp:lighter-than-AGG} will be derived in Appendix \ref{sec:other-boundaries}.

\subsection{Dependence and uniform relative stability}
\label{subsec:URS}

An important ingredient needed for a converse of Theorem \ref{thm:sufficient} is an appropriate characterization of the error dependence structure under which the said boundary is tight.
The notion of \emph{uniform relative stability} turns out to be the key in characterizing such dependence structures when studying the behavior of maxima of dependent light-tailed sequences.

\begin{definition}[Uniform Relative Stability] \label{def:URS}
Under the notations established in Definition \ref{def:RS}, the triangular array ${\cal E}$ is said to have uniform relatively stable (URS) maxima if for \emph{every} sequence of subsets $S_p\subseteq\{1,\ldots,p\}$ such that $|S_p| \to \infty$, we have
\begin{equation} \label{eq:URS-condition}
    \frac{1}{u_{|S_p|}} M_{S_p} := \frac{1}{u_{|S_p|}} \max_{j\in S_p} \epsilon_p(j) \xrightarrow{\P} 1,
\end{equation}
as $p\to\infty$, where $u_q,\ q\in \{1,\ldots,p\}$ is the generalized quantile in \eqref{eq:quantiles}.
The collection of arrays ${\cal E} = \{ \epsilon_p(j) \}$ with URS maxima is 
denoted $U(F)$.
\end{definition}

Uniform relative stability is, as its names suggests, a stronger requirement on dependence than relative stability. 
From the last section we see that, an array with iid components sharing a marginal distribution $F$ with rapid varying tails has relatively stable maxima; it is easy to see that URS also follows, by independence of the entries.

\begin{corollary} \label{cor:AGG-is-URS}
If $F\in\text{AGG}(\nu)$, $\nu>0$, an independent array ${\cal E}$ is URS; in this case, URS holds with $u_{|S_p|} \sim \left(\nu\log{|S_p|}\right)^{1/\nu}$.
\end{corollary}

On the other hand, RS and URS hold under much broader dependence structures than just 
independent errors. In turn, the stability concepts can be used to characterize dependence structures under which the maxima of error sequences {\em concentrate} around the quantiles \eqref{eq:quantiles} in the sense of \eqref{eq:RS-condition}.

The condition on the dependence structure of the array ${\cal E}$ imposed through the language of 
uniform relative stability seems, however, somewhat mysterious and implicit.  Fortunately, this is an 
extremely weak requirement.  In the rest of this section, we illuminate this condition in the case of 
Gaussian errors.  Specifically, we establish a simple necessary and sufficient condition for the 
uniform relative stability in terms of their covariance structure.

\begin{definition}[Uniformly decreasing dependence (UDD)] \label{def:UDD}
Consider a triangular array of jointly Gaussian distributed errors 
${\cal E} = \left\{\left(\epsilon_p(j)\right)_{j=1}^p, p = 1,2,\ldots\right\}$ 
with unit variances,
$$
\epsilon_p \sim \text{N}(0, \Sigma_p), \quad p=1,2,\ldots.
$$
The array ${\cal E}$ is said to be uniform decreasingly dependent (UDD) if 
for every $\delta>0$ there exists a finite $N(\delta)<\infty$, such that for every $j\in\{1,\ldots,p\}$, and $p\in\N$, we have
\begin{equation} \label{eq:UDD-definition}
    \Big|\left\{k\in\{1,\ldots,p\}:\Sigma_p(j,k)>\delta\right\}\Big| \le N(\delta)\quad \text{for all  } \delta>0.
\end{equation}
\end{definition}
That is, for any coordinate $j$, the number of coordinates which are more than $\delta$-correlated with $\epsilon_p(j)$ does not exceed $N(\delta)$. 

Note that the bound in \eqref{eq:UDD-definition} holds uniformly in $j$ and $p$, and only depends on $\delta$.
Also observe that in on the left-hand side of \eqref{eq:UDD-definition}, we merely count in each row of $\Sigma_p$ the number of exceedances of covariances (not their absolute values!) over level $\delta$.

\begin{remark} \label{rmk:choice-of-N(delta)}
Without loss of generality, we may require that $N(\delta)$ be a non-increasing function of $\delta$, for we can take
$$
N(\delta) = \sup_{p,j} \Big|\{k:\Sigma_p(j,k)>\delta\}\Big|,
$$
which is non-increasing in $\delta$.
Definition \ref{def:UDD} therefore states that the array is UDD when $N(\delta)<\infty$ for all $\delta>0$.
\end{remark}

Observe that the UDD condition does not depend on the order of the coordinates in the error 
vector $\epsilon_p = (\epsilon_p(j))_{j=1}^p$.  Often times, however, the errors are thought of 
coming from a stochastic process indexed by time or space.  To illustrate the generality of the 
UDD condition, we formulate next a simple sufficient condition (UDD$^\prime$) that is easier to 
check in a time-series context.

\begin{definition}[UDD\,$^\prime$]\label{d:UDD-prime}
For $\epsilon_p \sim N(0,\Sigma_p)$ with unit variances, an array ${\cal E} = \left(\epsilon_p(j)\right)_{j=1}^p$ is said to satisfy the UDD\,$^\prime$ condition if there 
exist:
\begin{enumerate}
    \item[(i)] permutations $l = l_p$ of $\{1,\ldots,p\}$, for all $p\in\N$, and
    \item[(ii)] a non-negative sequence $(r_n)_{n=1}^\infty$ converging to zero $r_n\to 0$, as $n\to\infty$,
\end{enumerate}
such that 
\begin{equation} \label{eq:weak-correlation}
    \sup_{p\in\N} |\Sigma_p\left(i',j'\right)| \le r_{|i-j|}.
\end{equation}
where $i' = l(i)$, $j' = l(j)$, for all $i,j\in\{1,\ldots,p\}$.
\end{definition}

\begin{remark}
Without loss of generality, we may also require that $r_n$ be non-increasing in $n$, for we can replace $r_n$ with $r'_n = \sup_{m\ge n} r_m$, which is non-increasing in $n$.
\end{remark}

\begin{proposition} \label{prop:UDD-equivalent}
UDD\,$^\prime$ implies UDD.
\end{proposition} 

\begin{proof}


Since $r_n\to 0$, for any $\delta > 0$, there exists an integer 
$M = M(\delta)<\infty$ such that $r_n\le\delta$, for all $n\ge M$. 
Thus, by \eqref{eq:weak-correlation}, for every fixed 
$j' \in\{1,\ldots,p\}$, we can have $|\cov(\epsilon_p(k'),\epsilon_p(j'))| > \delta$,
only if $k'$ belongs to the set:
$$ 
 \left\{ k' \in \{1,\dots,p\} \, :\, j-M \le  k := l_p^{-1}(k') \le j+M \right\},
$$
where $j:= l_p^{-1}(j')$. That is, there are at most $2M+1<\infty$ indices  $k'\in\{1,\dots,p\}$, whose covariances with $\epsilon(j')$ may exceed $\delta$. 
Since this holds uniformly in $j'\in\{1,\ldots,p\}$, Condition UDD follows with 
$N(\delta) = 2M+1$.
\end{proof}

%

We now state the last (and main) result of this section: a Gaussian sequence is URS if and only if it is UDD.
The URS condition essentially requires that the dependencies decay in a uniform fashion, the rate at which dependence decay does \emph{not} matter.

\begin{theorem} \label{thm:Gaussian-weak-dependence}
Let ${\cal E}$ be a Gaussian triangular array with standard normal marginals.  
The array ${\cal E}$ has uniformly relatively stable (URS) maxima if and only if it is uniformly decreasing dependent (UDD).
\end{theorem}

The proof is given in Section \ref{sec:proof-UDD-URS-outline}. 

We finish this subsection with a brief discussion on the relationships between UDD and other dependence conditions in the context of extreme value theory.

Suppose that the array of errors  ${\cal E}$ comes from a stationary Gaussian time series $\epsilon(j),\ j\in \mathbb{N}$, with auto-covariance $r_p=\cov(\epsilon(j+p),\epsilon(j))$. 
One is interested in the asymptotic behavior of the maxima $M_p:=\max_{j=1,\dots,p} \epsilon(j)$.

In this setting, the Berman's condition, introduced in \cite{berman1964limit}, requires that
\begin{equation} \label{eq:Berman}
    r_p \log p \to 0,\ \ \mbox{ as }p\to\infty.
\end{equation}
This condition entails that 
\begin{equation}
    \label{eq:Gauss-max-in-distribution}
  a_p (M_p - b_p) \stackrel{d}{\longrightarrow } Z,\  \ \mbox{ as }p\to\infty,
\end{equation}
with the Gumbel limit distribution $\mathbb P [Z\le x] = \exp\{-e^{-x}\},\ x\in \mathbb R$, 
where 
$$
a_p = \sqrt{2\log p},\quad b_p  = \sqrt{2\log p} - \frac{1}{2}\left(\sqrt{2\log p}\right)^{-1}\left(\log \log (p) + \log(4\pi)\right),
$$ 
are {\em the same} centering and normalization sequences
as in the case of iid $\epsilon(j)$'s.  
The Berman's condition is one of the weakest dependence conditions  in the literature for which this result holds. See, for example, Theorem 4.4.8 in \citet{embrechts2013modelling}, where the Berman's condition is described as ``very weak''.

For dependence conditions weaker than \eqref{eq:Berman}, it is known that the sequences of normalizing and centering constants in \eqref{eq:Gauss-max-in-distribution} are {\em different} from the iid case, and the corresponding limit is no longer Gumbel; see, for example, Theorems 6.5.1 and  6.6.4 in \citet{leadbetter2012extremes}. 
In particular, \citet{mccormick1976weak} derived the normalizing constants when both $r_p\to 0$ monotonically and  $r_p \log p \to \infty$ monotonically, as $p\to\infty$. 
In this case, convergence in distribution still takes place, with the maxima concentrating along a sequence asymptotic to \eqref{eq:quantiles}.

On the other hand, if one is merely interested in the asymptotic relative stability of the Gaussian maxima rather than in their distributional limit, then the Berman's condition
can be relaxed significantly. 
Observe that by Proposition \ref{prop:UDD-equivalent},  the Berman condition \eqref{eq:Berman} implies UDD and hence the relative stability (Theorem \ref{thm:Gaussian-weak-dependence}), i.e., 
\begin{equation} \label{eq:Gaussian-URS}
  \frac{1}{b_p} M_p \stackrel{\mathbb P}{\to} 1,\quad\mbox{as}\quad p\to\infty.
\end{equation}
This {\em concentration of maxima} property can of course be readily deduced from \eqref{eq:Gauss-max-in-distribution}, since $a_p b_p \sim 2\log(p) \to \infty$ as $p\to\infty$.
Our Theorem \ref{thm:Gaussian-weak-dependence} shows that \eqref{eq:Gaussian-URS} holds if and only if the much weaker uniform dependence condition UDD holds. 
Note that this condition is coordinate free, i.e., neither monotonicity of the sequence $r_p$ nor stationarity of the underlying array is required. 
The method of proof is also very different from the results on distributional convergence in the references mentioned above. 
In our high dimensional support estimation context, the notion of relative stability is sufficient and more natural than the finer notions of distributional convergence.

\subsection{Proof of Theorem \ref{thm:Gaussian-weak-dependence}}
\label{sec:proof-UDD-URS-outline}

The proof of the `only if' part is detailed in Section \ref{subsubsec:URS=>UDD}. 
The proof uses a surprising, yet elegant application of Ramsey's Theorem from the study of combinatorics; this application, and its consequences in high-dimensional probability, are presented in Section \ref{subsec:Ramsey}.
The proof of the `if' part is sketched in Section \ref{subsubsec:UDD=>URS},
details are filled in later in Section \ref{subsec:bounding-upper-tails-of-maxima} and \ref{subsec:bounding-lower-tails-of-maxima}.

\subsubsection{{\bf URS implies UDD (`only if' part of Theorem \ref{thm:Gaussian-weak-dependence})}} \label{subsubsec:URS=>UDD}
In view of Remark \ref{rmk:choice-of-N(delta)}, UDD is equivalent to the requirement that
$N(\delta) := 1+\sup_{p} N_p(\delta) < \infty$ for all $\delta\in(0,1)$,
where 
\begin{equation} \label{eq:N_p(c)}
    N_p(\delta) := \max_{j\in\{1,\ldots,p\}} \Big|\{i:i\neq j,\;\Sigma_p(j,i) > \delta\}\Big|.
\end{equation}
Therefore, if ${\cal E}$ is not UDD, then there must exist a constant $c\in (0,1)$ for which $N(c)$ is infinite, i.e., there is a subsequence $\widetilde p\to\infty$ such that $N_{\widetilde p}(c) \to \infty$.
Without loss of generality,  we may assume that $\widetilde{p}=p$.

Let $j_p(c)$ be the maximizers of \eqref{eq:N_p(c)}, and let
\begin{equation} \label{eq:sub-sequence_of_sets}
S_p(c):= \{ i\in\{1,\dots,p\}\, :\, \Sigma_p(j_p(c), i) > c \}.
\end{equation}
Observe that $|S_p(c)| = N_p(c)+1 \to \infty$, as $p\to\infty$ 
(note $j_p(c) \in S_{p}(c)$).

Applying Lemma \ref{lemma:positive-correlation} (see Section \ref{subsec:Ramsey} below) to the set of random variables indexed by $S_p(c)$, we conclude, for $N_p(c) \ge 2^{2\lceil2/c^2\rceil+4}$, there must be a further subset 
\begin{equation} \label{eq:further_sub-sequence_of_sets}
  K_p(c) \subseteq S_p(c),
\end{equation}
of cardinality 
\begin{equation} \label{eq:further_sub-sequence_of_sets_size}
k_p(c) := \left|K_p(c)\right| \ge \log_2{\sqrt{N_p(c)}},
\end{equation}
such that all pairwise correlations of the random variables indexed by $K_p(c)$ are greater than $c^2/2$.
Since the sequence $N_p(c)\to\infty$, by \eqref{eq:further_sub-sequence_of_sets_size}, we have $k_p(c)\to\infty$ as $p\to\infty$.

Therefore, we have identified a sequence of subsets $K_p(c)\subseteq\{1,\ldots,p\}$ with the following two properties:
\begin{enumerate}
  \item $k_p(c) := \left|K_p(c)\right| \to \infty$, as $p\to\infty$, and
  \item For all $i,j\in K_p(c)$, we have
  \begin{equation} \label{eq:further_sub-sequence_of_sets_cor}
    \Sigma_p(i,j) > c^2/2.
  \end{equation}
\end{enumerate}
Without loss of generality, we may assume $K_p(c) = \{1,\ldots,k_p(c)\} \subseteq \{1,\ldots,p\}$, upon re-labeling of the coordinates. 

Now consider a Gaussian sequence $\epsilon^* = \{\epsilon^*(j),\;j = 1,2,\ldots\}$, independent of ${\cal E}$, defined as follows:
$$
\epsilon^*(j):= Z \left(c/\sqrt{2}\right) + Z(j) \sqrt{1-{c^2}/{2}}, \quad j = 1, 2, \ldots,
$$ 
where $Z$ and $Z(j), j = 1, 2, \ldots$ are independent standard normal random variables. 
Hence,
\begin{equation} \label{eq:Slepian-conclusion-condition-1}
    {\rm Var}(\epsilon^*(j)) = 1 = {\rm Var}(\epsilon_p(j)),
\end{equation}
and
\begin{equation} \label{eq:Slepian-conclusion-condition-2}
    \cov(\epsilon^*(i),\epsilon^*(j)) = \frac{c^2}{2} \le \cov(\epsilon_p(i),\epsilon_p(j)),
\end{equation}
for all $p$, and all $i\neq j$, $i,j\in K_p(c)$.
Thus we have, as $p\to\infty$, 
\begin{equation} \label{eq:!UDD=>subsequence-fail}
    \frac{1}{u_{k_p(c)}} \max_{j\in K_p(c)} \epsilon^*(j) = \frac{c/\sqrt{2}}{u_{k_p(c)}}Z + \frac{\sqrt{1-c^2/2}}{u_{k_p(c)}} \max_{j\in K_p(c)} Z(j) \stackrel{\mathbb P}{\to} \sqrt{1-\frac{c^2}{2}},
\end{equation}
where the convergence in probability follows from Proposition \ref{prop:rapid-varying-tails} part \ref{prop:rapid-varying-tails_part-ii}.

Relations \eqref{eq:Slepian-conclusion-condition-1} and \eqref{eq:Slepian-conclusion-condition-2}, by Slepian's Lemma \cite{slepian1962one}, also imply,
\begin{equation}\label{eq:Slepian-conclusion}
  \frac{1}{u_{k_p(c)}} \max_{j\in K_p(c)} \epsilon^*(j) \stackrel{d}{\ge} \frac{1}{u_{k_p(c)}} \max_{j\in K_p(c)} \epsilon_p(j).
\end{equation}
Therefore, by \eqref{eq:Slepian-conclusion} and \eqref{eq:!UDD=>subsequence-fail}, for all $\sqrt{1-c^2/2} \le \delta < 1$, we have,
$$
\P\left[\frac{1}{u_{k_p(c)}} \max_{j\in K_p(c)} \epsilon_p(j) < \delta \right] \to 1 \quad\mbox{as  }p\to\infty.
$$
This contradicts the definition of URS (with the particular choice of $S_p:=K_p(c)$), and the proof of the `only if' part is complete.

\subsubsection{{\bf UDD implies URS (`if' part of Theorem \ref{thm:Gaussian-weak-dependence})}} \label{subsubsec:UDD=>URS}

Recall that our objective is to show \eqref{eq:URS-condition}. 
We will do so in two stages; namely, we will prove that for all $\delta>0$, we have 
\begin{equation} \label{eq:URS-condition-upper-side}
    \P\left[\frac{M_{S_p}}{u_{|S_p|}} > 1+\delta\right] \to 0,
\end{equation}
and
\begin{equation} \label{eq:URS-condition-lower-side}
    \P\left[\frac{M_{S_p}}{u_{|S_p|}} < 1-\delta\right] \to 0,
\end{equation}
for any sequence of subsets $S_p$ such that $|S_p|\to\infty$.
Although the first step \eqref{eq:URS-condition-upper-side} was already shown in Proposition \ref{prop:rapid-varying-tails}, regardless of the dependence structure, we provide in this section a more refined result. 
Specifically, the following result states that for the AGG model, the constant $\delta$ in Proposition \ref{prop:rapid-varying-tails} can be replaced by a vanishing sequence $c_p\to 0$.

\begin{lemma}[Upper tails of AGG maxima] \label{lemma:AGG-maxima-upper-tails}
Let ${\cal E}$ be an array with marginal distribution $F\in\text{AGG}(\nu)$, $\nu>0$. If we pick
\begin{equation} \label{eq:choice-of-c_p}
    c_p = \frac{u_{p\log{p}}}{u_p} - 1,    
\end{equation} 
where $u_p = F^{\leftarrow}(1-1/p)$, then we have $c_p>0$, $c_p\to 0$, and
\begin{equation} \label{eq:AGG-max-upper-bound}
    \P\left[\frac{M_p}{u_p}-(1+c_p) > 0\right] \to 0.
 \end{equation}
\end{lemma}
The proof can be found in Appendix \ref{subsec:bounding-upper-tails-of-maxima}.

Since Lemma \ref{lemma:AGG-maxima-upper-tails} holds regardless of the dependence structure, the same conclusions hold if one replaces $M_p$ by $M_{S_p} = \max_{j\in S_p}\epsilon(j)$ and $p$ by $q = q(p)=|S_p|$, where $S_p$ is any sequence of sets such that $q \equiv |S_p| \to \infty$.
This entails \eqref{eq:URS-condition-upper-side}.

On the other hand, the proof of \eqref{eq:URS-condition-lower-side} uses a more elaborate argument based on the Sudakov-Fernique bound.
We proceed by first bounding the probability by an expectation. 
For all $\delta>0$, we have
\begin{align}
    \P\left[\frac{M_{S_p}}{u_q}<1-\delta\right] 
        &= \P\left[-\left(\frac{M_{S_p}}{u_q} - (1+c_q)\right) > \delta + c_q\right] \nonumber \\
        &\le \P\left[\left(\frac{M_{S_p}}{u_q} - (1+c_q)\right)_->\delta+c_q\right] \nonumber \\
        &\le \frac{1}{\delta + c_q}\E\left[\left(\frac{M_{S_p}}{u_q} - (1+c_q)\right)_-\right], \label{eq:Gaussian-maxima-lower-expectation-bound}
\end{align}
where $(x)_-:=\max\{-x,0\}$ and the last line follows from the Markov inequality.
The next result shows that the upper bound in \eqref{eq:Gaussian-maxima-lower-expectation-bound} vanishes.
\begin{lemma} \label{lemma:Gaussian-maxima-lower-expectation}
  Let ${\cal E}$ be a Gaussian UDD  array  and 
  $S_p\subseteq\{1,\ldots,p\}$ be an arbitrary sequence of sets 
  such that $q = q(p) = |S_p|\to\infty$.  Then, for $M_{S_p}:= \max_{j\in S_p} \epsilon_p(j)$ and $c_q$ as in \eqref{eq:choice-of-c_p}, we have
  \begin{equation} \label{eq:Gaussian-maxima-lower-expectation}
    \E\left[\left(\frac{M_{S_p}}{u_q} - (1+c_q)\right)_-\right] \to 0,\ \ \quad \mbox{ as }p\to \infty.
  \end{equation}
\end{lemma}
The proof of the lemma is given in Appendix \ref{subsec:bounding-lower-tails-of-maxima}.

Going back to the proof of Theorem \ref{thm:Gaussian-weak-dependence}, we observe that Relations \eqref{eq:Gaussian-maxima-lower-expectation-bound} and \eqref{eq:Gaussian-maxima-lower-expectation} imply \eqref{eq:URS-condition-lower-side}, which completes the proof of the `if' part. \qed 

\begin{remark} Only the Sudakov-Fernique minorization argument used in the proof of Lemma \ref{lemma:Gaussian-maxima-lower-expectation}, relies on the Gaussian assumption. We expect the techniques and results here to be useful in extending Theorem \ref{thm:Gaussian-weak-dependence} to more general class of distributions, say, the AGG model.
\end{remark}

\subsection{Ramsey's Coloring Theorem and structure of correlation matrices} \label{subsec:Ramsey}

We provide a general result on the structure of arbitrary correlation matrices in this section, which may be of independent interest. 
Its proof uses the Ramsey Theorem from graph theory, which we briefly review next.

Given any integer $k\ge 1$, there is always an integer $R(k,k)$ called the {\em Ramsey number}:
\begin{equation}\label{eq:Ramsey-number}
k\le R(k,k)\le \binom{2k-2}{k-1}
\end{equation}
such that the following property holds:
every undirected graph with at least $R(k,k)$ vertices will contain {\em either} a clique of size $k$, or an {\em independent set} of $k$ nodes. 
Recall that a clique is a complete sub-graph where all pairs of nodes are connected, and an independent set is a set of nodes where no two nodes are connected.

This result is a consequence of the celebrated work of \citet{ramsey2009problem}, which 
gave birth to Ramsey Theory (see e.g., \citet{conlon2015recent}).  
The Ramsey Theorem and the upper bound \eqref{eq:Ramsey-number} (established first in \cite{erdos1935combinatorial}) are at the heart of the proof of the following result.

\begin{proposition} \label{prop:lower-bound-correlation-Ramsey}
  Fix $\gamma\in(0,1)$ and let $P = \left(\rho(i,j)\right)_{n\times n}$ be an arbitrary correlation
  matrix. If 
  \begin{equation}\label{eq:Ramsey-the-k-def}
   k:= \lfloor \log_2({n})/2 \rfloor  \ge \lceil 1/\gamma \rceil + 1,
  \end{equation}
  then there is a set of $k$ indices $K = \{l_1, \ldots, l_k\}\subseteq \{1,\ldots,n\}$ 
  such that 
  \begin{equation} \label{eq:lower-bound-correlation-Ramsey}
      \rho(i,j) \ge -\gamma, \mbox{ for all } i,j\in K.
  \end{equation}
\end{proposition}

\begin{proof}[Proof of Proposition \ref{prop:lower-bound-correlation-Ramsey}]
By using \eqref{eq:Ramsey-number} and a refinement of the Stirling's formula, 
we will show at the end of the proof that for $k \le \log_2({n})/2$, we have 
\begin{equation}\label{eq:Ramsey-bounds}
 R(k,k) \le n,
\end{equation}
where $R(k,k)$ is the Ramsey number.  

Now, construct a graph with vertices $\{1,\dots,n\}$ such that there is an edge between nodes $i$ and $j$ if and only if $\rho(i,j) > -\gamma$. 
In view of \eqref{eq:Ramsey-bounds} and Ramsey's theorem (see e.g., Theorem 1 in \cite{fox2009lecture} or \cite{conlon2015recent} for a recent survey on Ramsey theory), there is a subset of $k$ nodes $K =\{l_1,\dots,l_k\}$, which is either a {\em complete graph} or an {\em independent set}.

If $K$ is a complete graph, then by our construction of the graph, Relation \eqref{eq:lower-bound-correlation-Ramsey} holds. 

Now, suppose that $K$ is a set of independent nodes.  This means, again by the construction of our graph, that
$$
\rho(i,j) < -\gamma,\quad\mbox{for all }i\not= j\in K.
$$
Let $Z_i,\ i \in K$ be zero-mean random variables such that 
$\rho(i,j) = E [Z_iZ_j]$. Observe that
\begin{equation} \label{eq:Ramsey-proof-contradiction}
    \var\left( \sum_{i\in K} Z_i\right) 
    = \sum_{i\in K} \var(Z_i) + \sum_{\substack{i\not=j\\i,j \in K}} \cov(Z_i, Z_j) 
    <  k - k(k-1)\gamma,
\end{equation}
since $\var(Z_i)=1$ and $\rho(i,j)<-\gamma$ for $i\neq j$.
By our assumption, $k\ge \left(\lceil 1/\gamma \rceil + 1\right)$, or equivalently, $(k-1) \ge 1/\gamma$, the variance in \eqref{eq:Ramsey-proof-contradiction} is negative. 
This is a contradiction showing that there are no independent sets $K$ with cardinality $k$.

To complete the proof, it remains to show that Relation \eqref{eq:Ramsey-bounds} holds.
In view of the upper bound on the Ramsey numbers \eqref{eq:Ramsey-number}, it 
is enough to show that $k \le \log_2(\sqrt{n})$ implies
$$
 \binom{2k-2}{k-1} \le n.
$$
This follows from a refinement of the Stirling formula, due to \citet{robbins1955remark}:
$$
 \sqrt{2\pi} m^{m+1/2} e^{-m} e^{\frac{1}{(12 m +1)}} \le  m! \le \sqrt{2\pi} m^{m+1/2} e^{-m} 
 e^{\frac{1}{12 m}}.
$$
Indeed, letting $\widetilde k:= k-1$, and applying the above upper and lower bounds 
to the  terms $(2\widetilde k)!$ and $\widetilde k!$, respectively, we obtain:
\begin{align*}
\binom{2k-2}{k-1} \equiv \frac{(2\widetilde k)!}{ (\widetilde k!)^2 }
\le \frac{2^{2\widetilde k}}{\sqrt{\pi \widetilde k}}\exp\left\{ \frac{1}{24 \widetilde k} -
\frac{2}{ 12 \widetilde k +1}\right\} < 2^{2 k}
\end{align*}
where the last two inequalities follow by simply dropping positive factors less than $1$.
Since $2k \le \log_2(n)$, the above bound implies Relation \eqref{eq:Ramsey-bounds} 
and the proof is complete.
\end{proof}

Using Proposition \ref{prop:lower-bound-correlation-Ramsey}, we establish an auxiliary result used in the proof of Theorem \ref{thm:Gaussian-weak-dependence}.

\begin{lemma} \label{lemma:positive-correlation}
  Let $c\in(0,1)$, and $P = \left(\rho(i,j)\right)_{(n+1)\times(n+1)}$ be a correlation matrix such that \begin{equation} \label{eq:positive-correlation-lemma-condition}
      \rho(1,j) > c \quad \mbox{for all } j = 1,\ldots,n+1.
  \end{equation}
  If $n \ge 2^{2\lceil2/c^2\rceil+4}$, then there is a set of indices $K = \{l_1, \ldots, l_k\}\subseteq \{2,\ldots,n+1\}$ of cardinality $k = |K| = \lfloor\log_2{\sqrt{n}}\rfloor$, such that 
  \begin{equation} \label{eq:positive-correltation-lemma-conclusion}
      \rho(i,j) > \frac{c^2}{2} \quad\mbox{for all } i,j\in K.
  \end{equation}
  That is, all entries of the $k\times k$ sub-correlation matrix $P_K:=\left(\rho(i,j)\right)_{i,j\in K}$ are larger than $c^2/2$.
\end{lemma}

\begin{proof}[Proof of Lemma \ref{lemma:positive-correlation}]
    Let $Z_1, \ldots, Z_{n+1}$ be random variables with covariance matrix $P$.
    Denote $\rho_j = \rho(1,j)$ and define 
    \begin{equation}
      R(j) = 
      \begin{cases}
        \frac{1}{\sqrt{1-\rho_j^2}}\left(Z(j) - \rho_j Z(1)\right), &\mbox{if } \rho_j<1,\\
        R^* &\mbox{if } \rho_j=1,
      \end{cases}
    \end{equation}
    where $R^*$ is an arbitrary zero-mean, unit-variance random variable.
    It is easy to see that $\var(R(j)) = 1$, and
    \begin{align*}
    \cov\left(Z(i), Z(j)\right) &= \cov\left(\rho_i Z(1) + \sqrt{1-\rho_i^2} R(i), \; \rho_j Z(1) + \sqrt{1-\rho_j^2} R(j)\right) \\
        &= \rho_i\rho_j + \sqrt{1-\rho_i^2}\sqrt{1-\rho_j^2} \;\cov\left(R(i), R(j)\right) \\
        &\ge c^2 + \min\left\{\cov\left(R(i), R(j)\right), 0\right\}.
    \end{align*}
    
    Therefore, Relation \eqref{eq:positive-correltation-lemma-conclusion} would hold if we can find a set of indices $K = \{l_1,\ldots,l_k\}$ such that $\cov\left(R(i),R(j)\right)>-c^2/2$ for all $i,j\in K$, where $k=|K|=\lfloor\log_2\sqrt{n}\rfloor$.
    This, however, follows from Proposition \ref{prop:lower-bound-correlation-Ramsey} applied to $\left(R(j)\right)_{j=2}^{n+1}$ with $\gamma = c^2/2$, provided that 
    $$
    k = \lfloor\log_2\sqrt{n}\rfloor \ge \lceil 2/c^2 \rceil + 1.
    $$
    The last inequality indeed follows form the assumption that $n \ge 2^{2\lceil2/c^2\rceil+4}$.
\end{proof}

\section{Necessary conditions for exact support recovery}
\label{sec:necessary}

With the preparations from Section \ref{subsec:URS}, we are ready to state the necessary conditions for exact support recovery for thresholding procedures. 
It turns out that under the general dependence structure characterized by URS, the boundary \eqref{eq:strong-classification-boundary} is tight.

\begin{theorem} \label{thm:necessary}
    Let ${\cal E}$ be a triangular array with AGG marginal $F$ with parameter $\nu > 0$, and the signals $\mu$ be as described in Theorem \ref{thm:sufficient}. 
    Assume further that the errors ${\cal E}$ have uniform relatively stable maxima and minima, i.e., ${\cal E}\in U(F)$, and $(-{\cal E}) = \{-\epsilon_{p}(j)\} \in U(F)$.
    If 
    \begin{equation} \label{eq:signal-below-boundary}
        \overline{r} < g(\beta) = \left(1+(1-\beta)^{1/\nu}\right)^\nu,
    \end{equation}
    then no thresholding procedure can succeed in the exact support recovery problem.
    That is, for any $\widehat{S}_p = \{j\in[p]\;|\;x(j)\ge t_p(x)\}$, where the threshold $t_p(x)$ may depend on the data, 
    \begin{equation} \label{eq:classification-impossible-dependent}
        \lim_{p\to\infty}\mathbb P[\widehat{S}_p = S_p] = 0. 
    \end{equation}
\end{theorem}

We shall first comment on the role of error dependence.

\begin{remark} \label{rmk:dependence-assumptions}
Paraphrasing Theorems \ref{thm:sufficient} and Theorem \ref{thm:necessary}: 
if we consider only thresholding procedures, then for a very large class of dependence structures, we cannot improve upon the Bonferroni procedure $\widehat{S}_p^{\text{Bonf}}$. 
Specifically, for all ${\cal E}\in U(F)$, and for all $S_p\in {\cal S}$, we have
\begin{equation}
    \lim_{p\to\infty} \P[\widehat{S}_p^{\text{Bonf}}\neq S_p]
    = \begin{cases}
    \inf_{\widehat{S}_p \in {\cal T}} \liminf_{p\to\infty} \P[\widehat{S}_p\neq S_p] = 1, & \text{if}\quad \overline{r} < g(\beta)\\
    \inf_{\widehat{S}_p \in {\cal T}} \limsup_{p\to\infty} \P[\widehat{S}_p\neq S_p] = 0, & \text{if}\quad \underline{r} > g(\beta),\\
    \end{cases}
\end{equation}
where $\cal T$ is the set of all thresholding procedures (recall \eqref{eq:thresholding-procedure}). 

In view of Theorem \ref{thm:necessary}, we also answer the question raised in \citet{butucea2018variable}.
In particular, the authors of \citep{butucea2018variable} commented that independent error is the  `least favorable model' in the problem of support recovery, and conjectured that the support recovery problem may be easier to solve in the case of dependence, similar to how the problem of signal detection is easier under dependent errors (see \citep{hall2010innovated}).
Our results here state that asymptotically, \emph{all} error dependence structures in the URS class are equally difficult for thresholding procedures. 
Therefore, the phase-transition behavior is universal in the class of dependence structures characterized by URS.
\end{remark}

Returning to the more concrete example of Gaussian errors, Theorem \ref{thm:Gaussian-weak-dependence} yields the following corollary.

\begin{corollary} \label{cor:weakly-dependent-errors}
For UDD Gaussian errors, the conclusions in Theorem \ref{thm:necessary} hold.
\end{corollary}

Following Remark \ref{rmk:dependence-assumptions}, we will demonstrate the tightness of the dependence conditions in Theorem \ref{thm:necessary}.
Specifically, we demonstrate that if the URS dependence condition is violated, then it may be possible to recover the support of weaker signals below the boundary.

\begin{example} \label{exmp:counter-example}
Suppose ${\cal E} = \left(\epsilon_p(j)\right)_{j=1}^p$ is Gaussian, and is comprised of $\lfloor p^{1-\beta}\rfloor$ blocks, each of size at least $\lfloor p^\beta \rfloor$; 
let the elements of each block have correlation 1, and let elements from different blocks be independent. 
If $\underline{r} \ge 4(1-\beta)$, then the procedure $\widehat{S} = \big\{j:x(j)>\sqrt{2(1-\beta)\log{p}}\big\}$ yields $\mathbb P[\widehat{S} = S] \to 1$. 
This requirement on signal size is strictly weaker than that of the strong classification boundary, since $4(1-\beta) < (1 + \sqrt{1-\beta})^2$ on $\beta\in(0,1)$.

The above example shows that if the correlations of the Gaussian errors do not decay in a uniform fashion (UDD fails), then we can do substantially better in terms of support recovery.
\end{example} 

Proof of the claims in the example can be found in Appendix \ref{subsec:proofs-examples}; numerical simulations of this example can be found in Section \ref{sec:numerical}.

Our second comment is on the role of error tail assumptions.

\begin{remark} \label{rmk:tail-assumptions}
Theorems \ref{thm:sufficient} and \ref{thm:necessary} state that the phase-transition phenomenon is universal in all distributions in the AGG class for thresholding procedures.
We will see in Section \ref{sec:optimality} that thresholding procedures are in fact \emph{sub-optimal} for AGG models with $\nu<1$. 
Therefore, the minimax optimality of the thresholding procedures only applies to AGG models with $\nu\ge1$.

On the other hand, the phase-transition phenomenon for thresholding procedures is universal in all error models with rapidly varying tails, which includes AGG models {\it for all} $\nu>0$.
In contrast, errors with heavy (regularly varying) tails do not exhibit this phenomenon (see Appendix \ref{sec:heavy-tailed} Theorem \ref{thm:heavy-tails}).
We summarize the role of thresholding procedures in Table \ref{table:role-of-thresholding}.
\end{remark}

\begin{table}[ht]
    \caption{Properties of thresholding procedures under different error distributions (in brackets).}
    \begin{tabular}{p{35mm}p{32mm}p{32mm}} \toprule
        Thresholding procedure & Minimax-optimality &  Phase-transition \\ 
        (Error distributions) &  (Log-concave density) & (Rapidly-varying tails) \\ \midrule
        AGG($\nu$), $\nu\ge1$ & Yes (Yes) & Yes (Yes) \\ \cmidrule{1-3}
        AGG($\nu$), $0<\nu<1$ & No (No) & Yes (Yes) \\ \cmidrule{1-3}
        Power laws & No (No) & No (No) \\ \bottomrule
    \end{tabular}
  \label{table:role-of-thresholding}
\end{table}

We conclude this section with the proof of Theorem \ref{thm:necessary}.

\begin{proof} [Proof of Theorem \ref{thm:necessary}]
Since the estimator $\widehat{S}_p = \{x(j) \ge t_p(x)\}$ is thresholding, the threshold must separate the signals and null part, i.e.,
\begin{equation*}
    \P[\widehat{S}_p = S_p] 
    = \P\left[\max_{j\in S^c}x(j) < t_p(x) \le \min_{j\in S}x(j)\right]
    \le \P\left[\max_{j\in S^c}x(j) < \min_{j\in S}x(j)\right].
\end{equation*}
Using the assumption that the signal sizes are no greater than $\overline{\Delta}$, we have
\begin{align}
\P\left[\max_{j\in S^c}x(j) < \min_{j\in S}x(j)\right]  
  &= \P\left[\frac{\max_{j\in S^c}x(j)}{u_p} < \frac{\min_{j\in S}x(j)}{u_p}\right] \nonumber \\
  &\le  \P\left[\frac{\max_{j\in S^c}\epsilon(j)}{u_p} < \frac{\min_{j\in S}\overline{\Delta} + \epsilon(j)}{u_p}\right] \nonumber \\
  &= \P\left[ \frac{M_{S^c}}{u_p} < \frac{\overline{\Delta} - m_S}{u_p} \right], \label{eq:classification-possible-dependent-proof-0}
\end{align}
where $M_{S^c} = \max_{j\in S^c}\epsilon(j)$ and $m_{S} = \max_{j\in S}\left(-\epsilon(j)\right)$.
Since the error arrays ${\cal E}$ and $(-{\cal E})$ are URS by assumption, using the expression for the AGG quantiles \eqref{eq:AGG-quantiles}, we have
\begin{equation} \label{eq:classification-possible-dependent-proof-1}
    \frac{M_{S^c}}{u_p} = \frac{M_{S^c}}{u_{|S^c|}} \frac{u_{|S^c|}}{u_p} \xrightarrow{\P} 1,
\quad \text{and} \quad
\frac{m_{S}}{u_p} = \frac{m_{S}}{u_{|S|}} \frac{u_{|S|}}{u_p} \xrightarrow{\P} (1-\beta)^{1/\nu},
\end{equation}
so that the two random terms in probability \eqref{eq:classification-possible-dependent-proof-0} converge to constants.

Since signal sizes are bounded above by $\overline{r} < \left(1 + (1-\beta)^{1/\nu}\right)^{\nu}$, we can write $\overline{r}^{1/\nu} = 1 + (1-\beta)^{1/\nu} - d$ for some $d > 0$, by our parametrization of $\overline{\Delta}$, we have
\begin{equation} \label{eq:classification-possible-dependent-proof-2}
    \frac{\overline{\Delta}}{u_p} = \left(1+(1-\beta)^{1/\nu}-d\right)(1+o(1)).
\end{equation}
Combining \eqref{eq:classification-possible-dependent-proof-1} and \eqref{eq:classification-possible-dependent-proof-2}, we conclude that the right-hand-side of the probability \eqref{eq:classification-possible-dependent-proof-0} converges in probability to a constant strictly less than 1, that is, 
\begin{equation}
    \frac{\overline{\Delta} - m_S}{u_p} \xrightarrow{\P} 1 - d.
\end{equation}
Therefore, the last probability in \eqref{eq:classification-possible-dependent-proof-0} must go to 0.
\end{proof}

\section{The (sub)optimality of thresholding procedures}
\label{sec:optimality}

In support recovery problems under models such as \eqref{eq:model}, thresholding procedures seem to be the only ``reasonable'' choice for estimating the support set $S$.
In Theorem \ref{thm:sufficient} we saw that a (FWER-controlling) thresholding procedure is indeed consistent for support recovery. 
It is natural to ask if it will be sufficient to restrict attention to only thresholding procedures.
We show in this section that, perhaps surprisingly, for general error models, thresholding procedures are not always optimal.

We pause the discussion on additive models \eqref{eq:model}, and start by characterizing the optimal procedure for support recovery problems under general settings.

We adopt here a Bayesian framework to measure statistical risks.
Specifically, we assume there are $s$ signals independent with (not necessarily equal) densities $f_{1}, \ldots, f_{s}$, located on $\{1,\ldots,p\}$ uniformly at random.
That is, let the ordered support $P=(j_1,\ldots,j_s)$ have prior
\begin{equation} \label{eq:uniform-ordered}
\pi^{\text{ord}}((j_1,\ldots, j_s)) := \P\left[P = (j_1,\ldots, j_s)\right] = {(p-s)!}/{p!},
\end{equation}
for all distinct $1\le j_1, \ldots, j_s\le p$; and $x(j_i)$ has density $f_{i}$ for $i=1,\ldots,s$.

The unordered support $S=\{j_1,\ldots,j_s\}$ is consequently distributed uniformly in the collection of all set of size $s$, denoted $\mathcal{S} = \left\{S\subseteq\{1,\ldots,p\};|S|=s\right\}$. That is,
\begin{equation} \label{eq:uniform}
\pi(\{j_1,\ldots, j_s\}) := \P\left[S = \{j_1,\ldots, j_s\}\right] = {(p-s)!s!}/{p!},
\quad\text{for all}\;S\in\mathcal{S}.
\end{equation}
Let the rest $(p-s)$ locations be independent and have common density $f_0$, i.e., $x(j)$ has density $f_0$ for $j\not\in S$.

If we consider the 0-1 loss function,
$$
\ell(\widehat{S},S) = 1 - \mathbbm{1}[\widehat{S} = S] = \mathbbm{1}[\widehat{S} \neq S],
$$
then the optimal procedures should minimize the Bayes risk $\E_{P,{\cal E}} [\ell(\widehat{S},S)]$, where the expectation is taken over all configurations $P$, with a uniform distribution $\pi^{\text{ord}}$ as specified in \eqref{eq:uniform-ordered}, and over the errors ${\cal E}$.
For a sequence of estimators $\widehat{S} = \widehat{S}_p$, the asymptotically optimal procedures are defined to be the ones that minimize the asymptotic Bayes risk,
\begin{equation} \label{eq:Bayes-risk}
    R(\widehat{S}) := \limsup_{p\to\infty} \E_{P,{\cal E}} [\ell(\widehat{S},S)] = \limsup_{p\to\infty} \P_{P,{\cal E}}[\widehat{S} \neq S].
\end{equation}

\subsection{Optimality of thresholding procedures and monotone likelihood ratios}

Recently, \citet{butucea2018variable} derived non-asymptotic bounds on the Hamming loss risk under Gaussian errors, and also investigated their generalizations in the case of non-Gaussian distributions under the following \emph{monotone likelihood ratio} (MLR) property.
\begin{definition}[Monotone Likelihood Ratio]
A family of positive densities on $\R$, $\{f_\delta, \delta \in U\}$, is said to have the MLR property if, for all $\delta_0, \delta_1\in U\subseteq\R$ such that $\delta_0 < \delta_1$, the likelihood ratio $\left(f_{\delta_1}(x)/f_{\delta_0}(x)\right)$ is an increasing function of $x$.
\end{definition}

If the sparsity $s = |S|$ of the problem is known, then a natural estimator for $S$ would be based on the set of top $s$ order statistics.
\begin{definition}[Oracle data thresholding]
We call $\widehat{S}^* = \{j\,|\, x(j)\ge x_{[s]}\}$ the oracle data thresholding procedure, where $x_{[1]} \ge \ldots \ge x_{[p]}$ are the order statistics of $x$.
\end{definition}

The MLR property is intimately linked with data thresholding procedures.

\begin{lemma} \label{lemma:optimal-oracle-procedures}
Let the observations $x(j)$, $j=1,\ldots,p$ have $s$ signals as prescribed in \eqref{eq:uniform-ordered}, and let the rest $(p-s)$ locations have common density $f_0$.
If each of $\{f_0, f_{1}\},\ldots,\{f_0,f_{s}\}$ form an MLR family, then the oracle data thresholding procedure $\widehat{S}^* = \{j\,|\, x(j)\ge x_{[s]}\}$ is finite-sample optimal in terms of Bayes risk $\E_{P,{\cal E}} [\ell(\widehat{S},S)]$.
\end{lemma} 

The proof of Lemma \ref{lemma:optimal-oracle-procedures} is found in Appendix \ref{sec:proofs}.

We emphasize that the oracle thresholding procedures are in fact \emph{finite-sample optimal} in the Bayesian context.
Additionally, our setup allows for different alternative distributions, and relaxes the assumptions of \citet{butucea2018variable}, where the alternatives are assumed to be identically distributed; 
this is indeed an important practical concern.
It remains to understand when the key MLR property holds. 
We elaborate on this question next.

\subsection{Optimal procedure under sub-exponential errors}

Returning to the more concrete signal-plus-noise model \eqref{eq:model}, it turns out that the error tail behavior is what determines the optimality of data thresholding procedures.
In this setting, log-concavity of the error densities is \emph{equivalent} to the MLR property (Lemma \ref{lemma:MLR-log-concavity}). This, in turn, yields the finite-sample optimality of data thresholding procedures (Proposition \ref{prop:log-concave}).

\begin{lemma} \label{lemma:MLR-log-concavity}
Let $\delta$ be the magnitude of the non-zero signals in the signal-plus-noise model \eqref{eq:model} with positive error density $f_0$, and let $f_\delta(x) = f_0(x-\delta)$.
The family $\{f_\delta, \delta \in \R\}$ has the MLR property if and only if the error density $f_0$ is log-concave.
\end{lemma}

\begin{proof}[Proof of Lemma \ref{lemma:MLR-log-concavity}]
Suppose MLR holds, we will show that $f_0(t) = \exp\{\phi(t)\}$ for some concave function $\phi$.
By the assumption of MLR, for any $x_1 < x_2$, setting $\delta_0 = 0$, and $\delta_1 = (x_2 - x_1)/2 > 0$, we have
\begin{equation*}
    \log{\frac{f_{\delta_1}(x_2)}{f_{\delta_0}(x_2)}}
    = \phi\left(\frac{(x_1+x_2)}{2}\right)- \phi(x_2) 
    \ge \phi(x_1)- \phi\left(\frac{(x_1+x_2)}{2}\right) 
    = \log{\frac{f_{\delta_1}(x_1)}{f_{\delta_0}(x_1)}}.
\end{equation*}
This implies that the log-density $\phi(t)$ is midpoint-concave, i.e., for all $x_1$ and $x_2$, we have,
\begin{equation}
    \phi\left(\frac{(x_1+x_2)}{2}\right) 
    \ge \frac{1}{2} \phi(x_1) + \frac{1}{2} \phi(x_2).
\end{equation}
For Lebesgue measurable functions, midpoint concavity is equivalent to concavity by the Sierpinki Theorem (see, e.g., Sec I.3 in \citep{donoghue2014distributions}). This proves the `only-if' part.

For the `if' part, when  $\phi(t) = \log{(f_0(t))}$ is log-concave, then for any $\delta_0 < \delta_1$, and any $x<y$, we have
\begin{equation} \label{eq:concavity-implies-MLR}
    \log{\frac{f_{\delta_1}(y)}{f_{\delta_0}(y)}} - \log{\frac{f_{\delta_1}(x)}{f_{\delta_0}(x)}}
    = \phi(y-\delta_1) - \phi(y-\delta_0) - \phi(x-\delta_1) + \phi(x-\delta_0) \ge 0,
\end{equation}
where the last inequality is a simple consequence of concavity (see Lemma \ref{lemma:four-point-concavity}). This proves the `if' part.
\end{proof}

Lemmas \ref{lemma:optimal-oracle-procedures} and \ref{lemma:MLR-log-concavity} yield immediately the following.

\begin{proposition} \label{prop:log-concave}
Consider the additive error model \eqref{eq:model}.
Let the errors $\epsilon$ be independent with common distribution $F$.
Let the signal $\mu$ have $s$ positive entries with magnitudes $0<\delta_1\le\ldots\le\delta_s$, located on $\{1,\ldots,p\}$ as prescribed in \eqref{eq:uniform-ordered}.
If $F$ has a positive, log-concave density $f$, then the oracle thresholding procedure $\widehat{S}^* = \{j\,;\,x(j)\ge x_{[s]}\}$ is finite-sample optimal in terms of Bayes risk.
\end{proposition} 
Remarkably, under MLR (or equivalently, log-concavity), the oracle thresholding procedure is finite-sample optimal even in the case where the signals have difference (positive) sizes.

The assumption of log-concavity of the densities is compatible with the AGG model when $\nu\ge1$, as demonstrated in the next example.

\begin{example}
The generalized Gaussian density $f(x)\propto \exp\{-|x|^\nu/\nu\}$ is log-concave for all $\nu\ge1$.
In this case, by Proposition \ref{prop:log-concave}, the oracle thresholding procedure is Bayes optimal.
\end{example}


Consider the asymptotic Bayes risk as defined in \eqref{eq:Bayes-risk}, the statement for the necessary condition of support recovery, with the help of Proposition \ref{prop:log-concave}, can be strengthened to include all procedures, regardless of whether they are thresholding.

\begin{theorem} \label{thm:necessary-strengthened}
In the context of Theorem \ref{thm:necessary}, where the signal is below the strong classification boundary \eqref{eq:signal-below-boundary},
if the $\epsilon_p(j)$'s are independent and identically distributed with log-concave densities, and the signals are uniformly distributed as in \eqref{eq:uniform}, then no procedure succeeds. That is, for any sequence of estimators $\widehat{S} = \widehat{S}_p$, we have $R(\widehat{S}) = 1.$
\end{theorem}

\begin{proof}[Proof of Theorem \ref{thm:necessary-strengthened}]
When errors are independent with log-concave density, the oracle thresholding procedure, by Proposition \ref{prop:log-concave}, is optimal.
But all thresholding procedures fail according to Theorem \ref{thm:necessary}.
\end{proof}

\begin{remark}
In view of Theorem \ref{thm:necessary-strengthened}, Bonferroni's thresholding procedure $\widehat{S}_p^{\text{Bonf}}$ described in Theorem \ref{thm:sufficient} is an asymptotically minimax procedure over the class of error models $U(F)$.
That is, in the class of errors with AGG($\nu$) marginals $F$ where $\nu\ge1$, we have
\begin{equation} \label{eq:minimax}
    \widehat{S}_p^{\text{Bonf}} \in \argmin_{\widehat{S}_p} \sup_{{\cal E}\in U(F)} R(\widehat{S}_p),
\end{equation}
where the minimum is taken over \emph{all} procedures $\widehat{S}$.
It can be shown that the supremum in the Bayes risk formulation \eqref{eq:minimax}, where $S$ is random and uniformly distributed, can be replaced by a supremum over all possible (deterministic) configurations of support sets.
This can be done by applying Theorem 1.1 in the supplement of \citep{butucea2018variable}.
We content ourselves with a Bayesian minimax result here.
\end{remark}

\subsection{Optimal procedure under super-exponential errors}

When MLR fails (or equivalently, when errors in model \eqref{eq:model} do not have log-concave densities), the optimal procedures are no longer data thresholding (Proposition \ref{prop:log-convex}). 
Instead, the following result (Lemma \ref{lemma:likelihood-ratio-thresholding}) shows that a different type of thresholding procedures which we refer to as \emph{likelihood ratio thresholding}, is in fact optimal. 

\begin{lemma} \label{lemma:likelihood-ratio-thresholding}
Let the observations $x(j)$, $j=1,\ldots,p$ have $s$ signals as prescribed in \eqref{eq:uniform-ordered} having common density $f_a$, and let the rest $(p-s)$ locations have common density $f_0$.
Define the likelihood ratios 
$$
L(j) := {f_a(x(j))}\big/{f_0(x(j))},
$$
and let $L_{[1]} \ge L_{[2]} \ge \ldots \ge L_{[p]}$ be the order statistics of $L(j)$'s.
Then the procedure $\widehat{S}_{\text{opt}} = \{j\,|\,L(j) \ge L_{[s]}\}$ is finite-sample optimal in terms of Bayes risk $\E_{P,{\cal E}} [\ell(\widehat{S},S)]$.
\end{lemma} 

The proof of Lemma \ref{lemma:likelihood-ratio-thresholding} is found in Appendix \ref{sec:proofs}.

\begin{remark}
While the assumption that the alternatives have a common density $f_a$ is slightly more restrictive than in Lemma \ref{lemma:optimal-oracle-procedures}, it is not excessively so if one considers a fully Bayesian alternative, where $f_a = p_1f_1 + p_2f_2 + \ldots + p_sf_s$ is a random mixture of alternatives $f_1,\ldots,f_s$ with probability weights $p_1,\ldots,p_s$.
\end{remark}

The characterization of optimal likelihood ratio thresholding procedures in Lemma \ref{lemma:likelihood-ratio-thresholding} may not always yield practical estimators, as the density of alternatives, and number of signals are typically unknown.
Nevertheless, some insights can still be obtained by virtue of Lemma \ref{lemma:likelihood-ratio-thresholding}.

\begin{proposition} \label{prop:log-convex}
Consider the additive error model \eqref{eq:model}.
Let the errors $\epsilon$ be independent with common distribution $F$.
Let each of the $s$ signals be located on $\{1,\ldots,p\}$ uniformly at random, and take on magnitudes $\delta_1\le\ldots\le\delta_s<\infty$ independently with probability $p_1,\ldots,p_s$.
If errors $\epsilon$ are iid with density $f$ log-convex on $[K, +\infty)$, then whenever $j\in\widehat{S}_{\text{opt}}$ for some $x(j) > K+\delta_s$, we must have $j'\in\widehat{S}_{\text{opt}}$ for all $K+\delta_s \le x(j') < x(j)$.
\end{proposition} 

Specifically, if there are $m$ observations exceeding $K+\delta_s$, with $m>s$, then the top $m-s$ observations will \emph{not} be included in the optimal estimator $\widehat{S}_{\text{opt}}$.
This shows that, in the case when the errors have super-exponential tails, the optimal procedures are in general \emph{not} data thresholding.

\begin{proof}[Proof of Proposition \ref{prop:log-convex}]
Since all $f_{\delta i}(t) = {f(t-\delta_i)}$ are log-convex on $[K+\delta_s, \infty)$, so is the alternative $f_a(t) = \sum_{i=1}^s p_if_{\delta i}(t)$.
By Relation \eqref{eq:concavity-implies-MLR} in the proof of Lemma \ref{lemma:MLR-log-concavity} and appealing to log-convexity (rather than log-concavity), the likelihood ratios
$$
L(j) := \frac{f_a(x(j))}{f_0(x(j))} = \sum_{i=1}^s \frac{p_if_{\delta i}(x(j))}{f_0(x(j))},
$$
are decreasing in $x(j)$ on $[K+\delta_s, \infty)$.
The proof is complete by applying Lemma \ref{lemma:likelihood-ratio-thresholding}.
\end{proof}

\begin{example}[Sub-optimality of data thresholding] \label{exmp:log-convex}
Let the errors have generalized Gaussian density with $\nu=1/2$, i.e., $\log{f_0(x)}\propto -x^{1/2}$. Let dimension $p=2$, sparsity $s=1$ with uniform prior, and signal size $\delta=1$.
If the observations take on values $x = (x_1, x_2)^\mathrm{T} = (1,2)^\mathrm{T}$, we see from a comparison of the likelihoods
$$
\log \frac{f(x|\{1\})}{f(x|\{2\})} = 2x_1^{1/2} + 2(x_2 - 1)^{1/2} - 2x_2^{1/2} - 2(x_1 - 1)^{1/2} = 4 - 2\sqrt{2} > 0,
$$
that even though $x_1<x_2$, the set $\{1\}$ is a better estimate of support than $\{2\}$, i.e., $\P[S=\{1\}\,|\,x] > \P[S=\{2\}\,|\,x]$.
In other words, $\widehat{S}_{\text{opt}} \neq \widehat{S}^*$.
\end{example}

It should be remarked that Proposition \ref{prop:log-convex} assumes the existence and knowledge of a maximum signal size $\delta_s$. 
When maximum signal sizes are unknown a priori, a large observation may be attributed to either a truly large signal, or to the error.
Common practice has been to attribute large observations to signals, and not errors.

As an example, consider the linear regression
\begin{equation} \label{eq:regression}
 Y = X\mu + \xi,
\end{equation}
where $\mu$ is a vector of regression coefficients of interest to be inferred from observations of $X$ and $Y$.
If the design matrix $X$ is of full column rank, then the OLS estimator of $\mu$ can be formed 
\begin{equation*}
    \widehat{\mu} = \left(X'X\right)^{-1}X'Y = \mu + \epsilon,
\end{equation*}
where $\epsilon := (X'X)^{-1}X'\xi$.
Hence we recover the generic problem \eqref{eq:model}. 
The support recovery problem is therefore equivalent to the fundamental model selection problem.
Often an investigator calculates the t-scores of each coefficient as 
\begin{equation} \label{eq:linear-model-selection}
    \widehat{\mu}(j) \Big/ \widehat{\mathrm{s.e.}}(\widehat{\mu}(j)),
\end{equation}
where $\widehat{\mathrm{s.e.}}(\widehat{\mu}(j))$ is the estimated standard error of $\widehat{\mu}(j)$.
The investigator then chooses indices with large t-scores to enter the model.
If the errors in model \eqref{eq:regression} are iid Gaussian, the expression \eqref{eq:linear-model-selection} is t-distributed and have power-law tails; the discussion above suggests that this commonplace procedure may be sub-optimal for bounded signals.

\section{Numerical illustrations}
\label{sec:numerical}

We examine numerically the boundaries in finite dimensions, illustrating the boundary \eqref{eq:strong-classification-boundary} under several error tail assumptions, as well as error dependence structures.

\subsection{Exact support recovery and tail assumptions}

To demonstrate the phase-transition phenomenon under different error tail densities, we simulate from the additive error model \eqref{eq:model} with
\begin{itemize}
    \item Gaussian errors, where the density is given by
    $$f(x) = \frac{1}{\sqrt{2\pi}}\exp{\left\{-x^2/2\right\}}.$$
    \item Laplace errors, where the density is given by $$f(x) = \frac{1}{2}\exp{\left\{-\left|x\right|\right\}}.$$
    \item Generalized Gaussian errors with $\nu=1/2$, where the density is given by 
    $$f(x) = \frac{1}{2}\exp{\left\{-2\left|x\right|^{1/2}\right\}}.$$
\end{itemize}
The sparsity and signal size of the sparse mean vector are parametrized as in equations \eqref{eq:sparsity-parametrized} and \eqref{eq:signal-size-parametrized}, respectively.
The support set $S$ is estimated with 
$\widehat{S} = \left\{j:x(j)>\sqrt{2\log{p}}\right\}$ 
under the Gaussian errors, 
$\widehat{S} = \left\{j:x(j)>\log{p} + (\log{\log{p}})/2\right\}$ 
under the Laplace errors, and with
$\widehat{S} = \{j:x(j)> \frac{1}{4}\left(W\left(-c/(ep\log{p})\right) + 1\right)^2\}$
under the generalized Gaussian ($\nu = 1/2$) errors. Here $W$ is the Lambert W function, i.e., $W=f^{-1}$ where $f(x)=x\exp{(x)}$.
The choices of thresholds correspond to Bonferroni's procedures with FWER decreasing at a rate of $1/\sqrt{\log{p}}$, therefore satisfying the assumptions in Theorem \ref{thm:sufficient}.
Experiments were repeated 1000 times under each sparsity-and-signal-size combination.

The results of the numerical experiments are shown in Figure \ref{fig:phase-simulated}.
The numerical results illustrate that the predicted boundaries are not only accurate in high-dimensions ($p=10000$, right panels of Figure \ref{fig:phase-simulated}), but also practically meaningful even at moderate dimensions ($p=100$, left panels of Figure \ref{fig:phase-simulated}).

\begin{figure}
      \centering
      \includegraphics[width=0.4\textwidth]{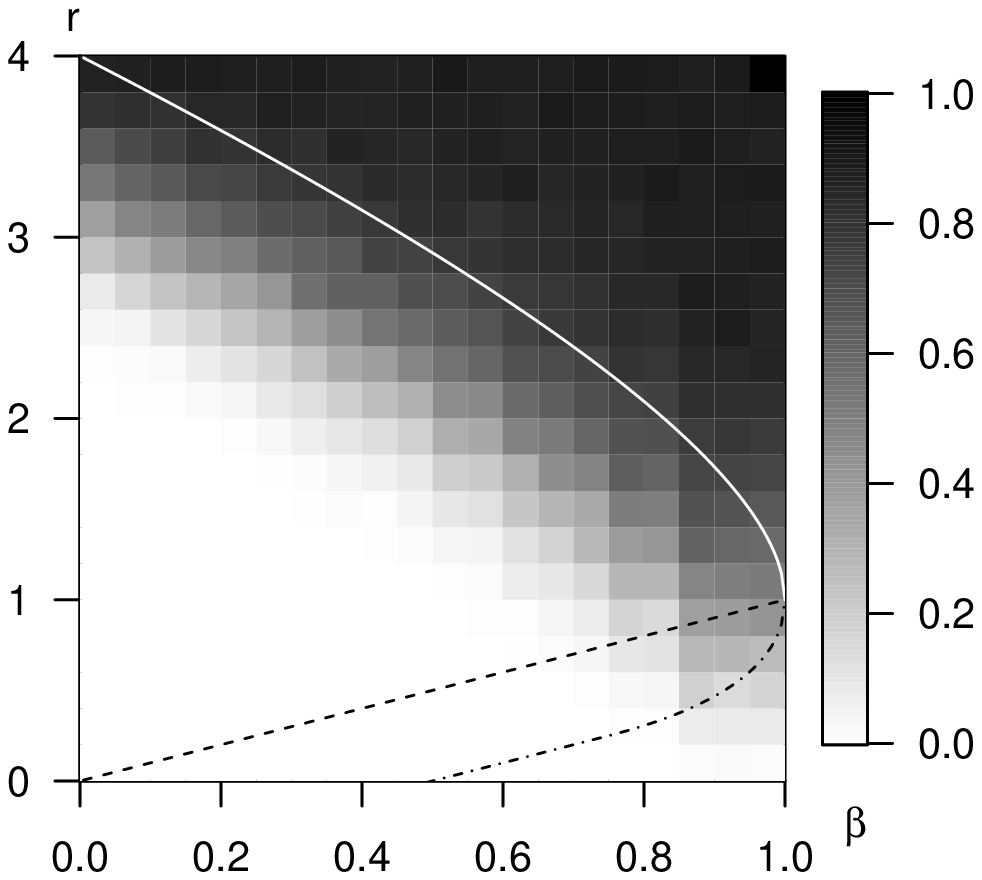}
      \includegraphics[width=0.4\textwidth]{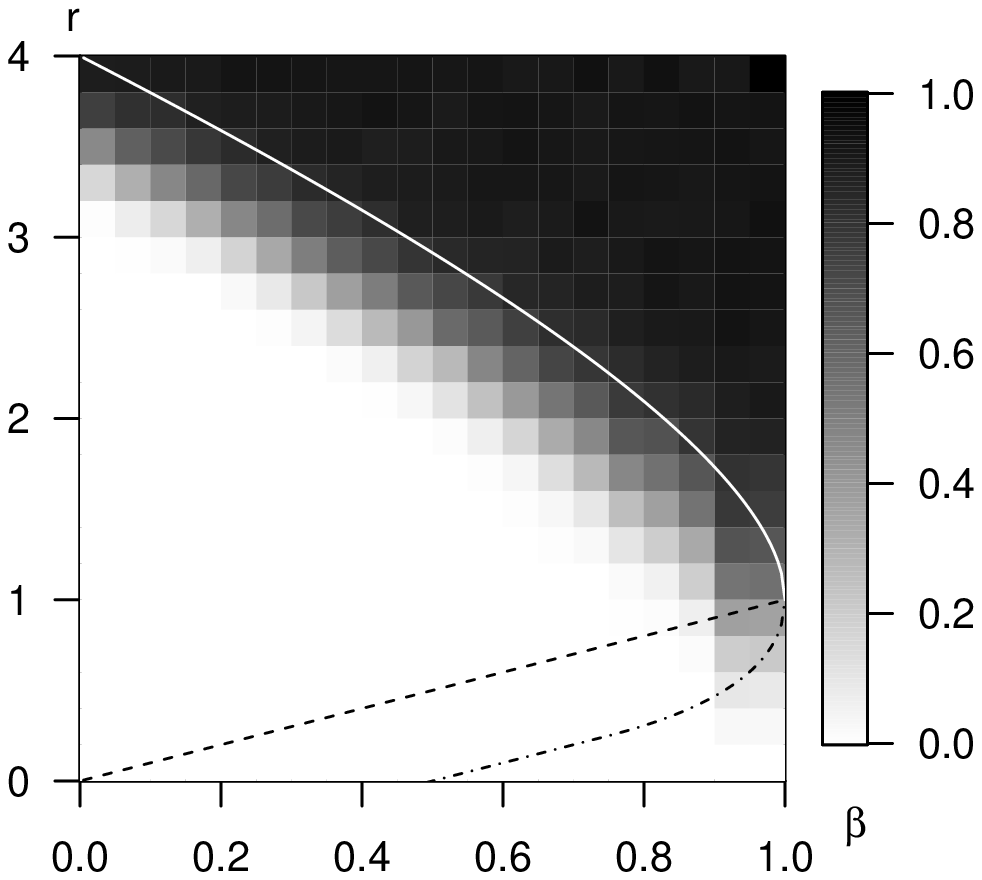}
      \includegraphics[width=0.4\textwidth]{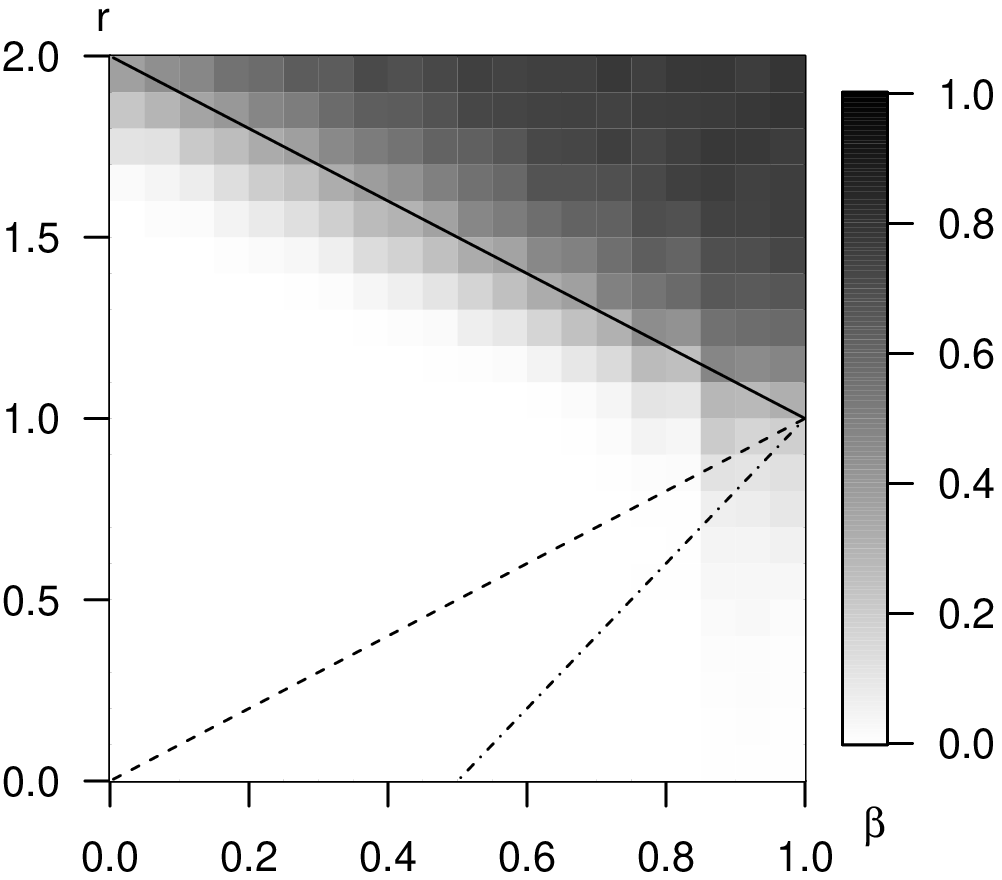}
      \includegraphics[width=0.4\textwidth]{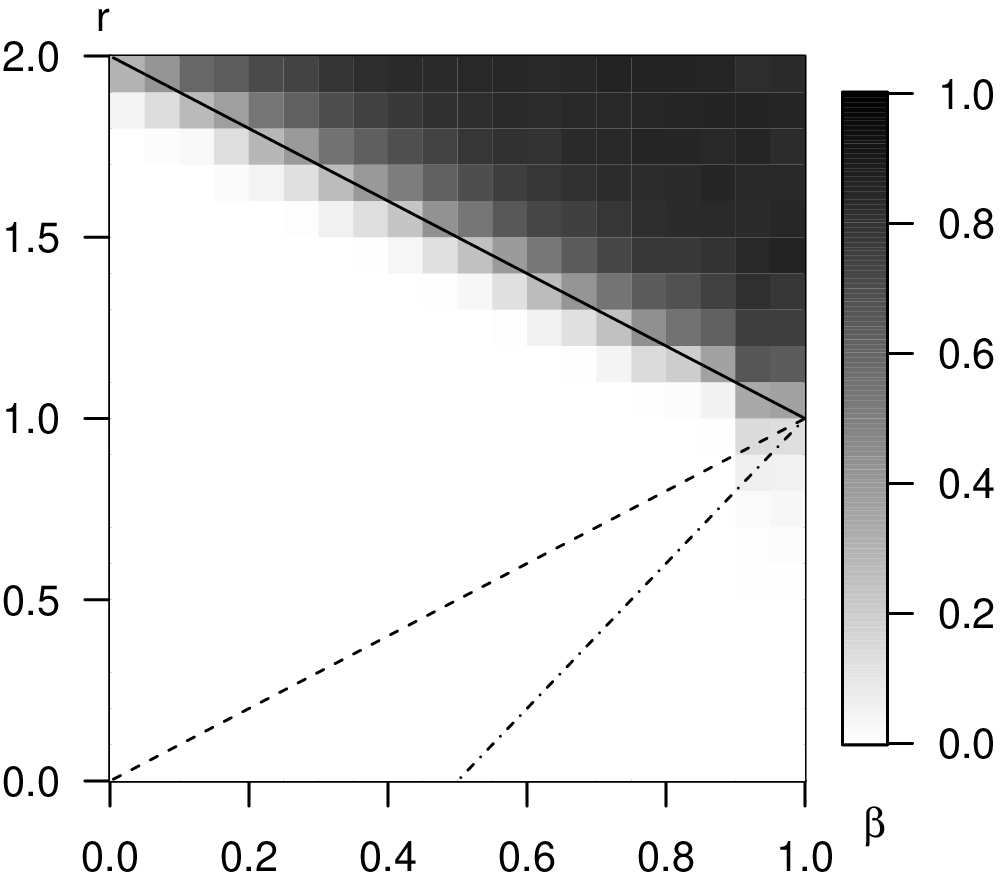}
      \includegraphics[width=0.4\textwidth]{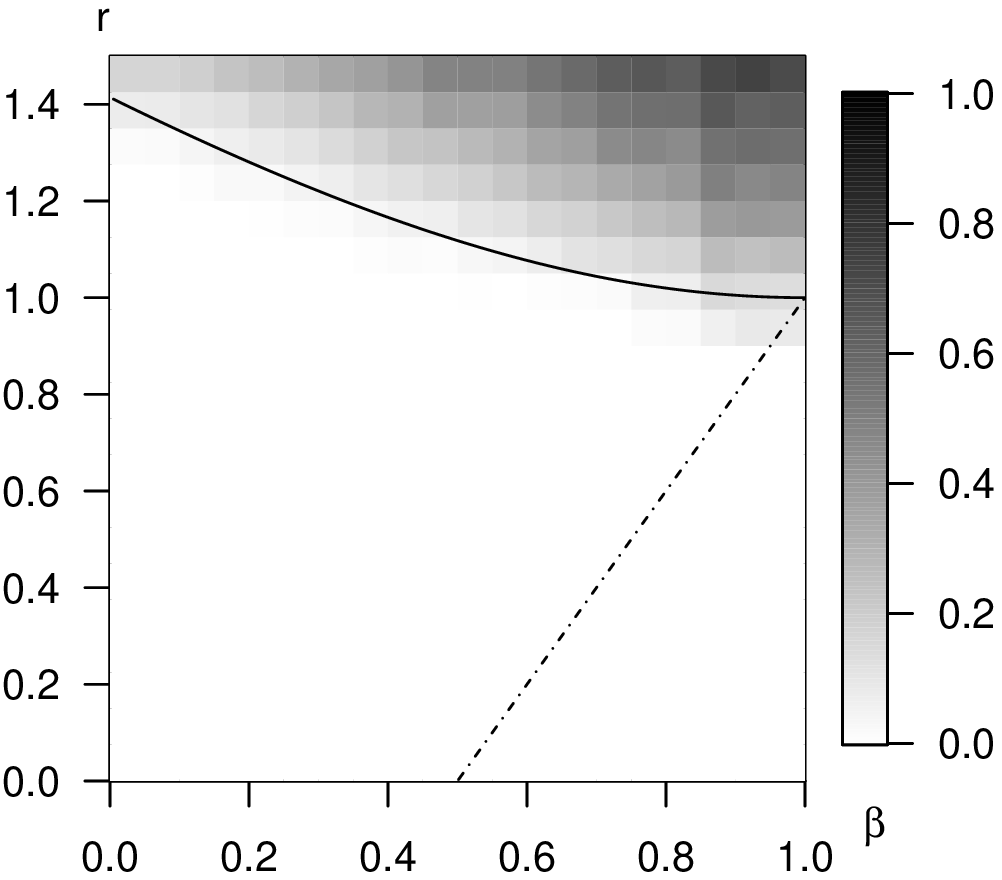}
      \includegraphics[width=0.4\textwidth]{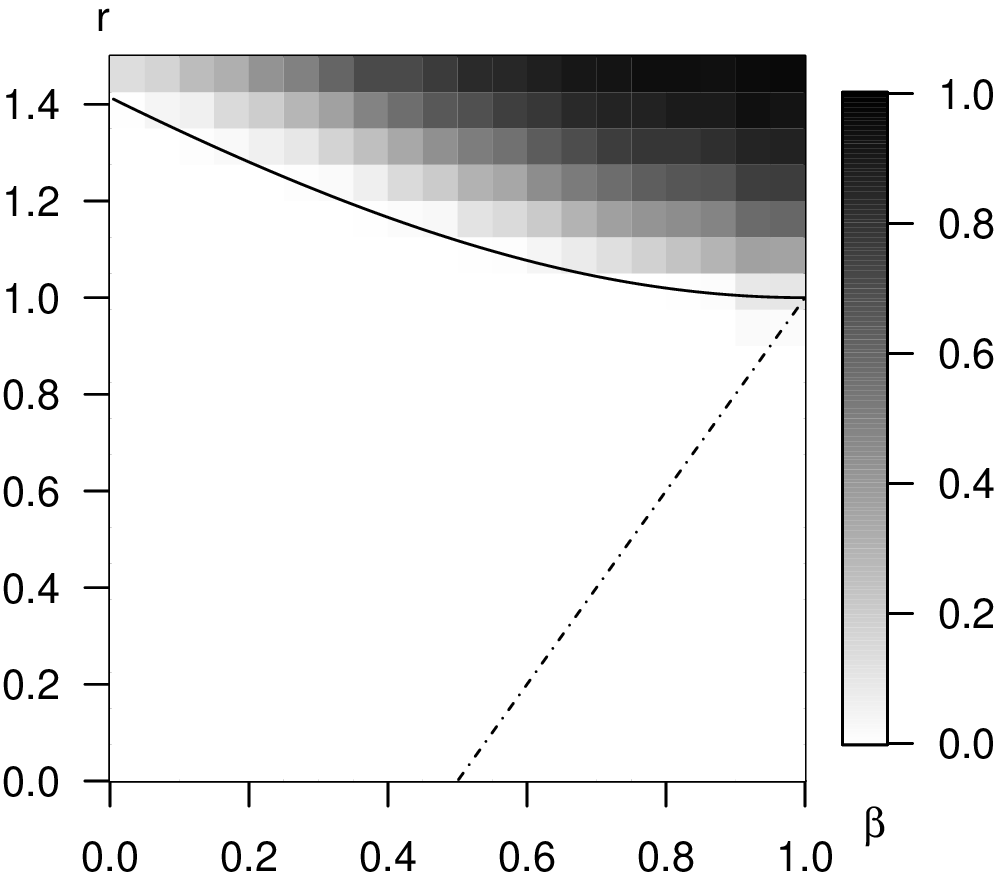}
      \caption{The empirical probability of exact support recovery from numerical experiments, as a function of sparsity level $\beta$ and signal sizes $r$, from Gaussian error models (upper panels), Laplace error models (middle panels), and generalized Gaussian with $\nu=1/2$ (lower panels); darker color indicates higher probability of exact support recovery. 
      The experiments were repeated 1000 times for each sparsity-signal size combination, and for dimensions $p=100$ (left panels) and $p=10000$ (right panels). Numerical results agree with the boundaries described in Theorem \ref{thm:sufficient}; convergence is noticeably slower for under generalized Gaussian ($\nu=1/2$) errors.
      For reference, the dashed and dash-dotted lines represent the weak classification and detection boundaries (see Sec \ref{sec:sufficient}, and \citep{haupt2011distilled, arias2017distribution, ingster1998minimax, donoho2004higher}).}
      \label{fig:phase-simulated}
\end{figure}

\subsection{Exact support recovery and dependence}

We also illustrate the effect of dependence on the phase-transition behavior in finite dimensions.
We shall compare the performance of the Bonferroni's procedure, which is agnostic to both sparsity and signal size, with the oracle procedure which picks the top-s observations.

The first set of experiments involves dependent errors from an auto-regressive (AR) model.
\begin{itemize}
    \item AR(1) Gaussian errors with parameter $\rho = -0.5$, $\rho = 0.5$, and $\rho = 0.9$,
    where the autocovariance functions decay exponentially
    $$\rho_{k} = \rho^{k}.$$
\end{itemize}
The results of the experiment are shown in Figure \ref{fig:phase-simulated-dependent}.

The second set of experiments explores heavily dependent, and non-UDD errors.
In particular we simulate
\begin{itemize}
    \item Fractional Gaussian noise (fGn) with Hurst parameter $H = 0.75$ and $H = 0.9$. 
    The autocovariance functions are 
    $$\rho_{k} \sim 0.75k^{-0.6},$$
    and, respectively,
    $$\rho_{k} \sim 1.44k^{-0.2},$$
    as $k\to\infty$.
    Both fGn models represent the regime of long-range dependence, where covariances decay very slowly to zero, so that $\sum|\rho_k| = \infty$; see, e.g., \citep{taqqu:2003-livre}.
    \item The non-UDD Gaussian errors described in Example \ref{exmp:counter-example}.
\end{itemize}
We will apply both the the sparsity-and-signal-size-agnostic Bonferroni's procedure, i.e., $\widehat{S} = \{j:x(j)>\sqrt{2\log{p}}\}$, as well as the oracle procedure $\widehat{S}^* = \{j:x(j)\ge x_{[s]}\}$, $s=|S|$, to all settings.
Results of the numerical experiments for the AR models are shown in Figure \ref{fig:phase-simulated-dependent}; the fGn and non-UDD models are shown in Figure \ref{fig:phase-simulated-very-dependent}.

Notice that the oracle procedure sets its thresholds more aggressively (at roughly $\sqrt{2\log s}$) than the Bonferroni procedure (at $\sqrt{2\log p}$).
Although this difference vanishes as $p\to\infty$, in finite dimensions ($p=10\,000$) the advantage can be felt. 
Indeed, in all our experiments the oracle procedures is able to recover support of signals with higher probability than the Bonferroni procedures; compare left and right columns of Figures \ref{fig:phase-simulated-dependent} and \ref{fig:phase-simulated-very-dependent}.
Notice also that there is an increase in probability of recovery near $\beta=0$ for oracle procedures.
This is an artifact due to the fact that $s = \lfloor p^{1-\beta}\rfloor < p/2$, and there are more signals than nulls. The oracle procedure is able to adjust to this reversal by lowering its threshold accordingly.

For UDD errors, Theorem \ref{thm:necessary} predicts that exact recovery of the support is impossible when signal sizes are below the boundary \eqref{eq:strong-classification-boundary}, even with oracle procedures. 
Both the AR and the fGn models generate UDD Gaussian errors, and should demonstrate the same phase-transition boundary.
However, the rate of this convergence (i.e., $\P[\widehat{S}^*=S]\to0\;\text{or}\;1$) can be very slow when the errors are heavily dependent,
even though all AR and fGn models demonstrate qualitatively the same behavior in line with the predicted boundary \eqref{eq:strong-classification-boundary}. 
In finite dimensions ($p=10\,000$), as dependence in the errors increases (AR(1) to fGN(H=0.9)), the oracle procedure becomes more powerful at recovering signal support with high probability for weaker signals. 
See right columns of Figures \ref{fig:phase-simulated-dependent} and \ref{fig:phase-simulated-very-dependent}.

On the other hand, as demonstrated in Example \ref{exmp:counter-example}, non-UDD errors yield qualitatively different behavior; exact support recovery is possible for signal sizes strictly weaker than that in the UDD case. 
Lower-right panel of Figure \ref{fig:phase-simulated-very-dependent} demonstrates in this example that the signal support can be recovered as long as the signal sizes are larger than $4(1-\beta)$.

\begin{figure}
    \centering
    \includegraphics[width=0.4\textwidth]{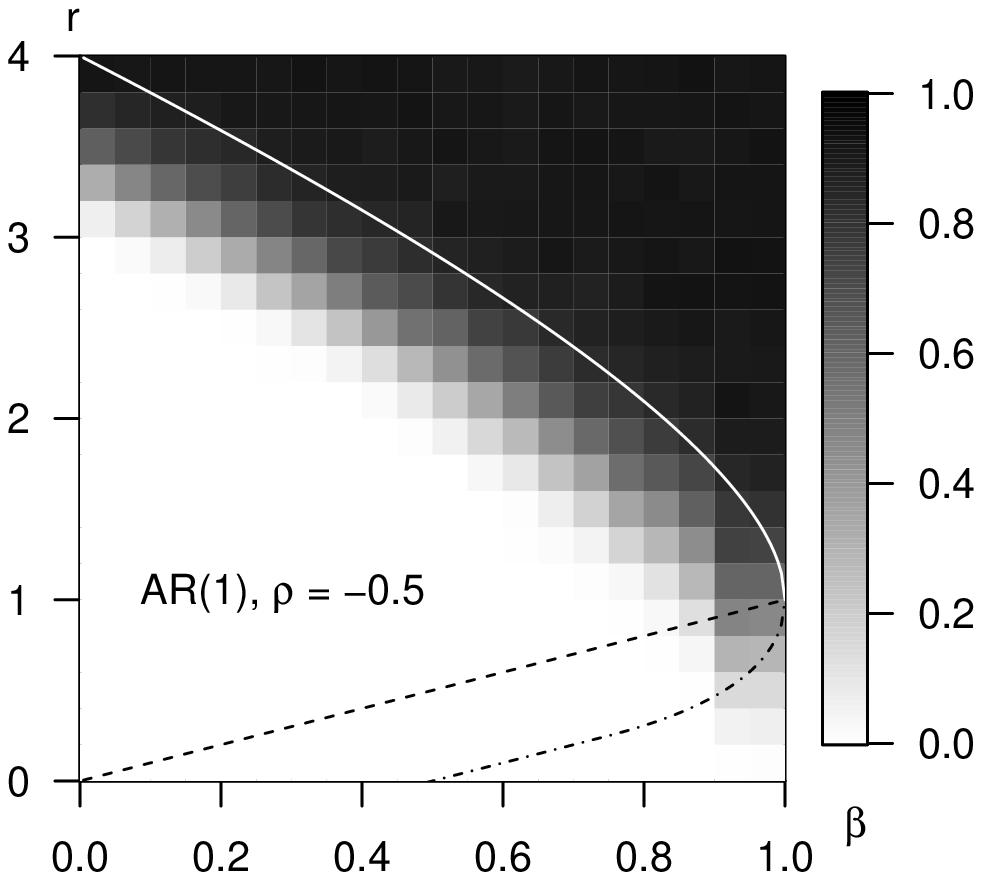}
    \includegraphics[width=0.4\textwidth]{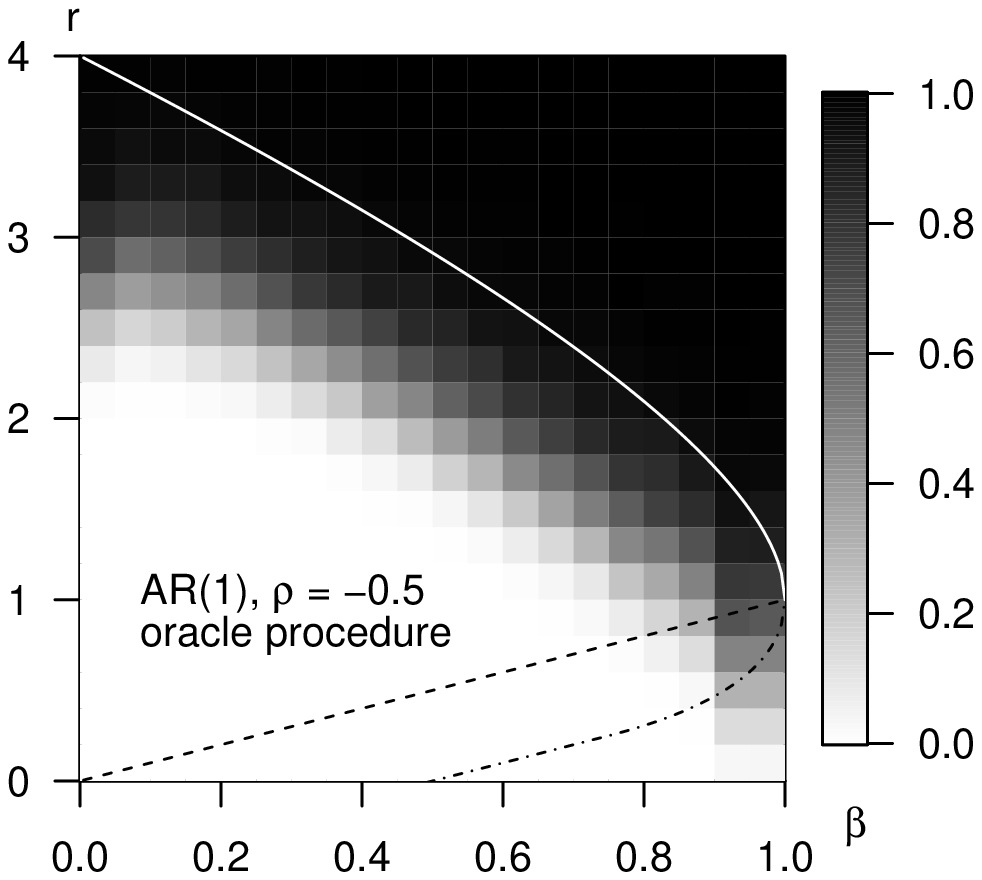}
    \includegraphics[width=0.4\textwidth]{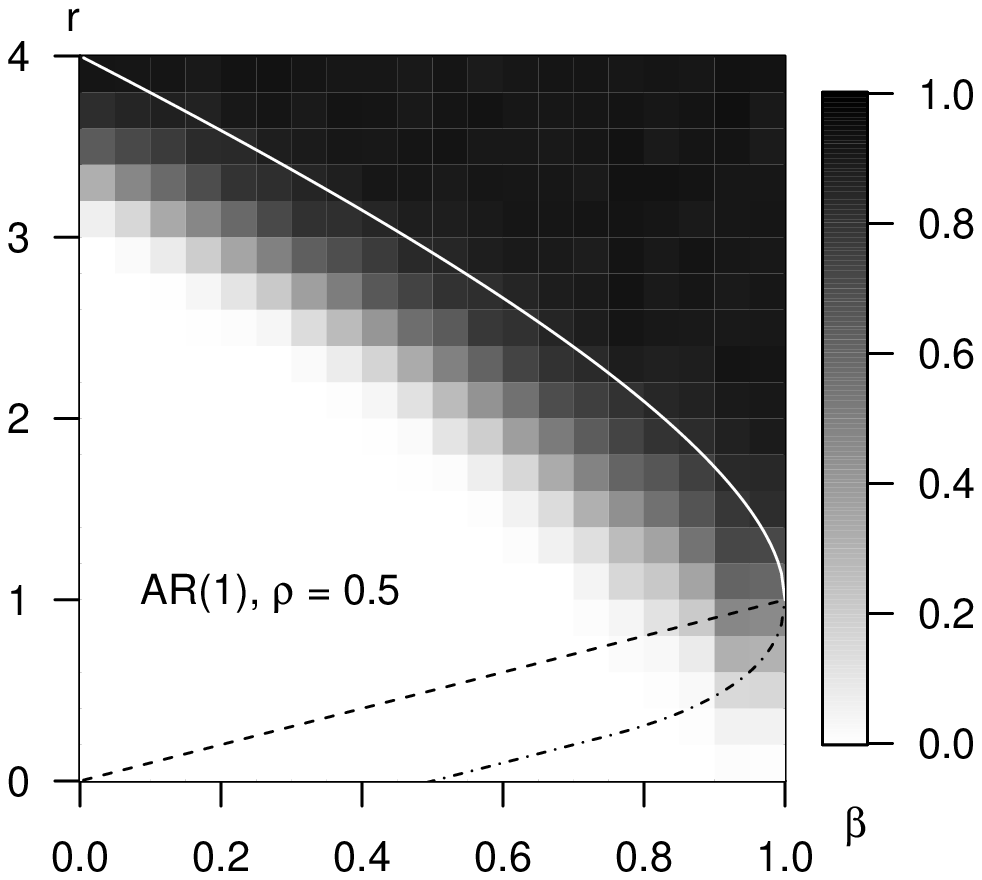}
    \includegraphics[width=0.4\textwidth]{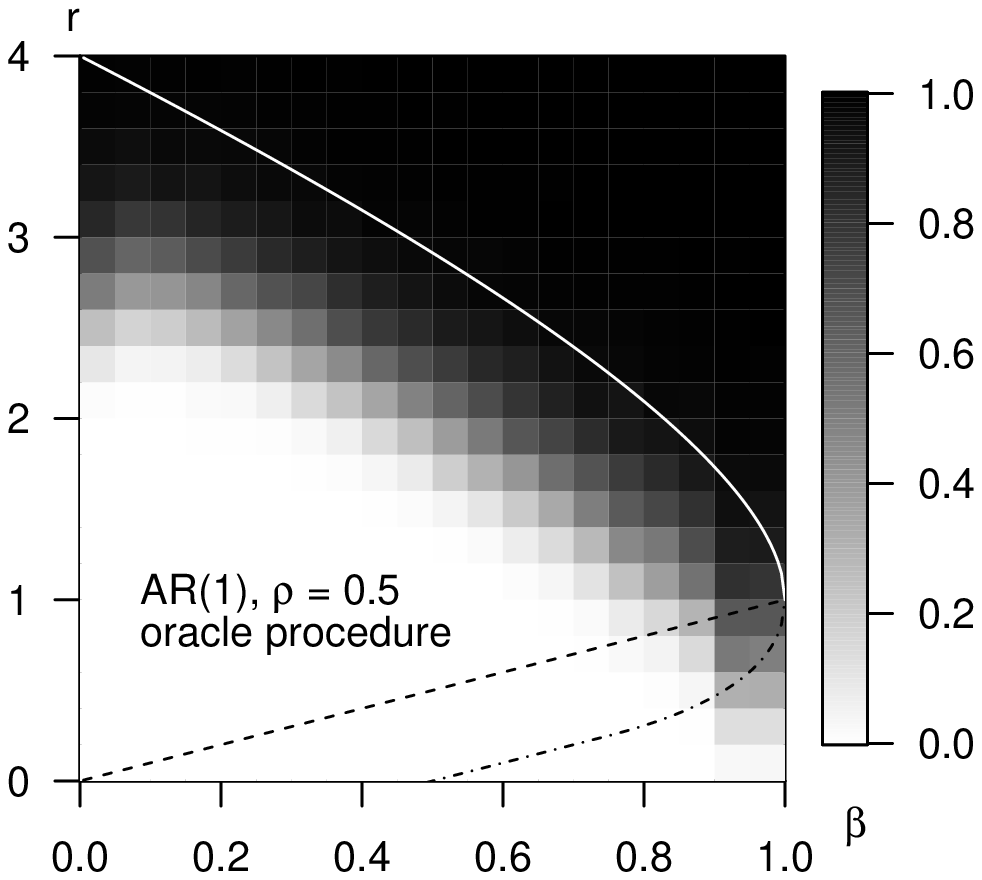}
    \includegraphics[width=0.4\textwidth]{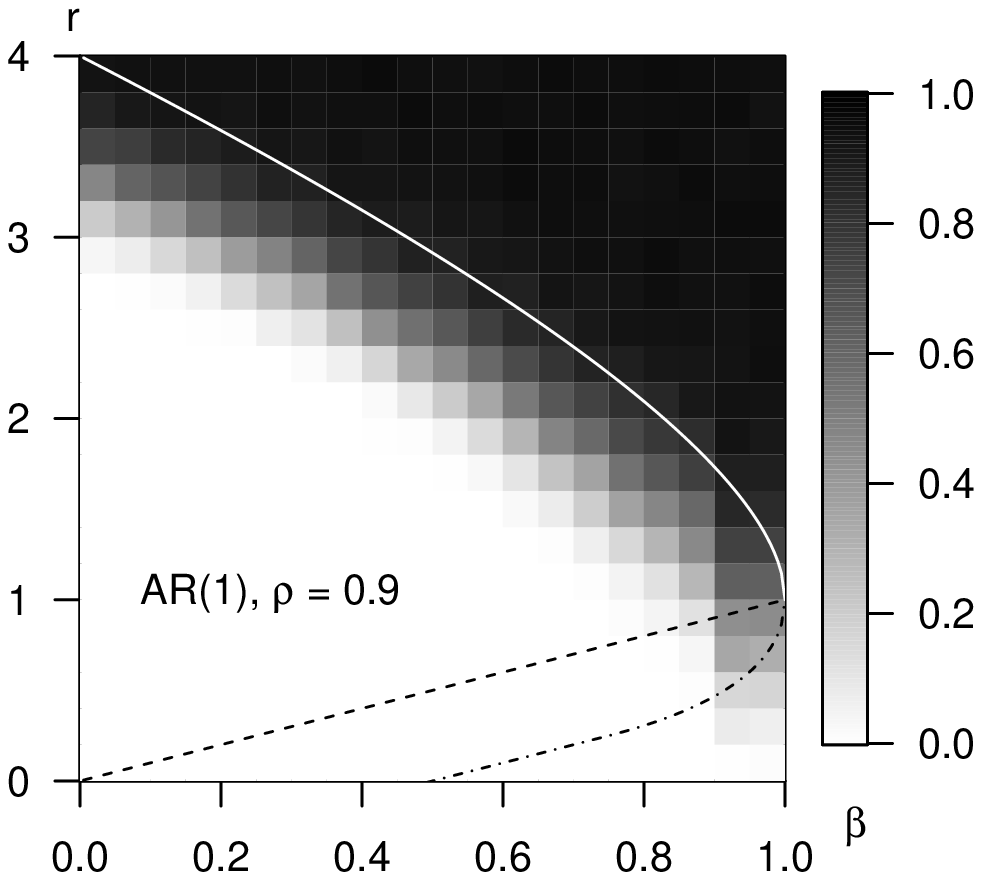}
    \includegraphics[width=0.4\textwidth]{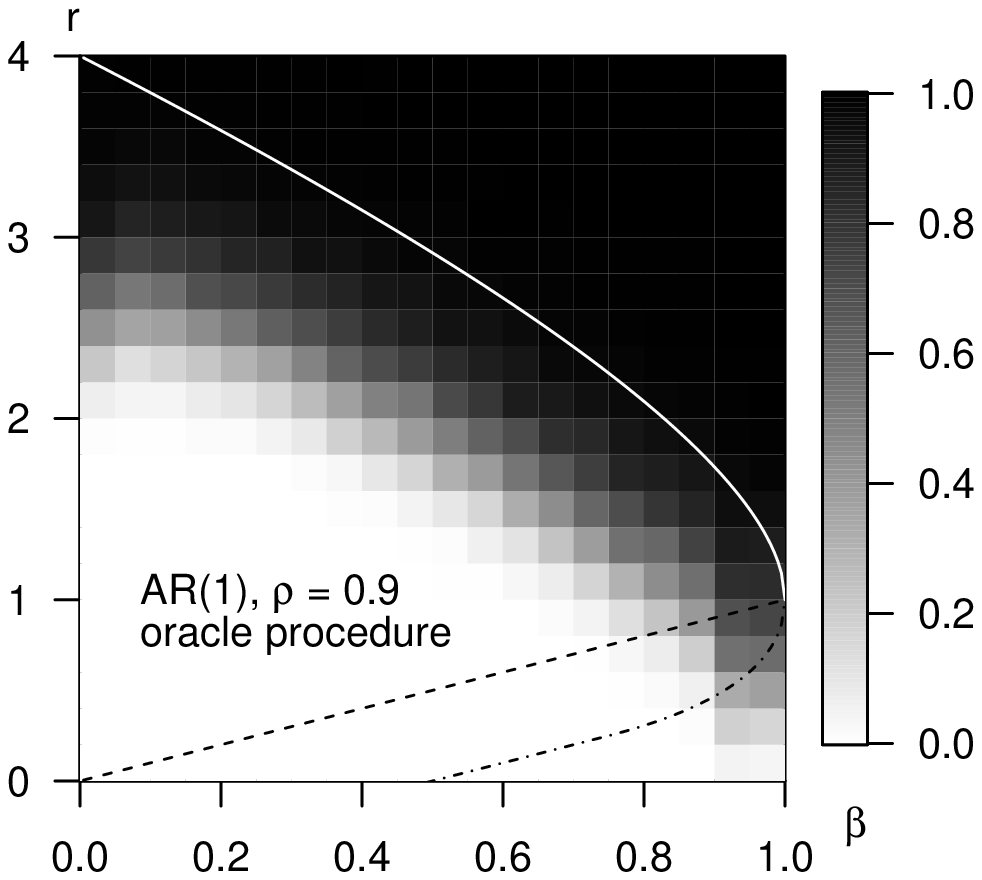}
    \caption{The empirical probability of exact support recovery from numerical experiments, as a function of sparsity level $\beta$ and signal sizes $r$. Darker colors indicate higher probability of exact support recovery. 
    Three AR(1) models with autocorrelation functions $(-0.5)^k$ (upper), 
    $0.5^k$ (middle), and $0.9^k$ (lower) are simulated.
    The experiments were repeated 1000 times for each sparsity-signal size combination.
    In finite dimensions ($p=10000$), the Bonferroni procedures (left) suffers small loss of power compared to the oracle procedures (right).
    A phase-transition in agreement with the predicted boundary \eqref{eq:strong-classification-boundary} can be seen in the AR models.
    The boundaries (solid, dashed, and dash-dotted lines) are as in Fig \ref{fig:phase-simulated}.
    See text for additional comments.}
    \label{fig:phase-simulated-dependent}
\end{figure}

\begin{figure}
    \centering
    \includegraphics[width=0.4\textwidth]{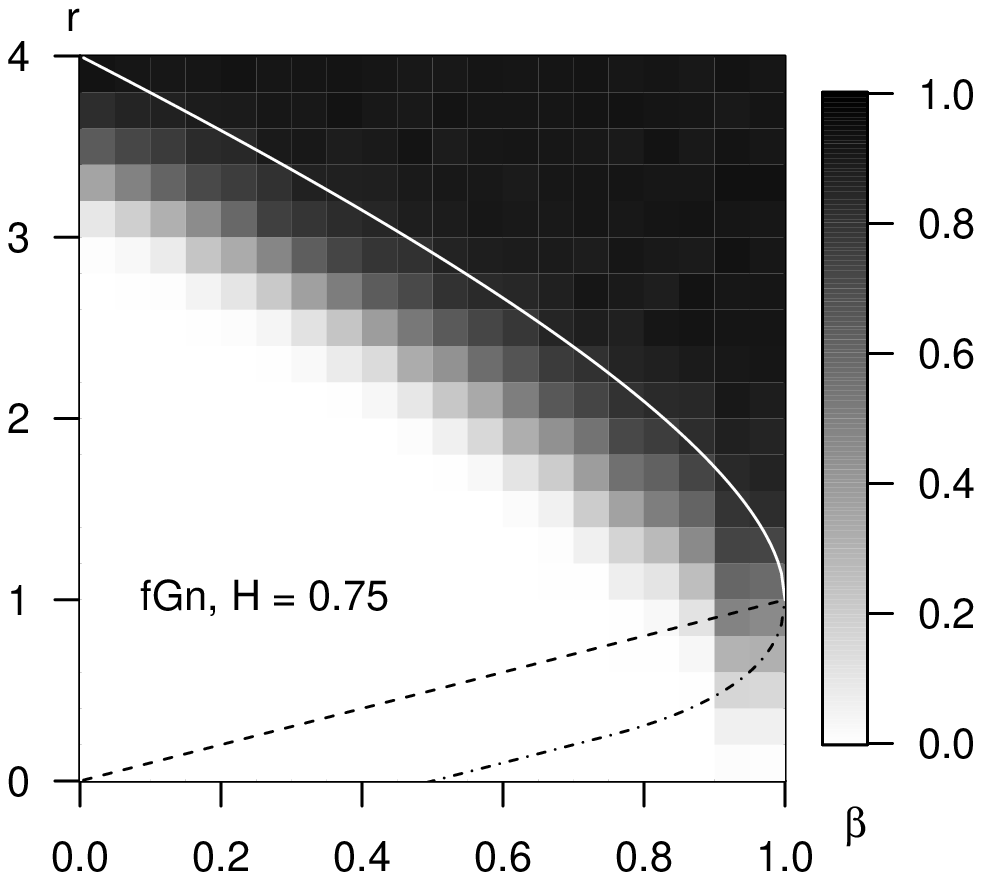}
    \includegraphics[width=0.4\textwidth]{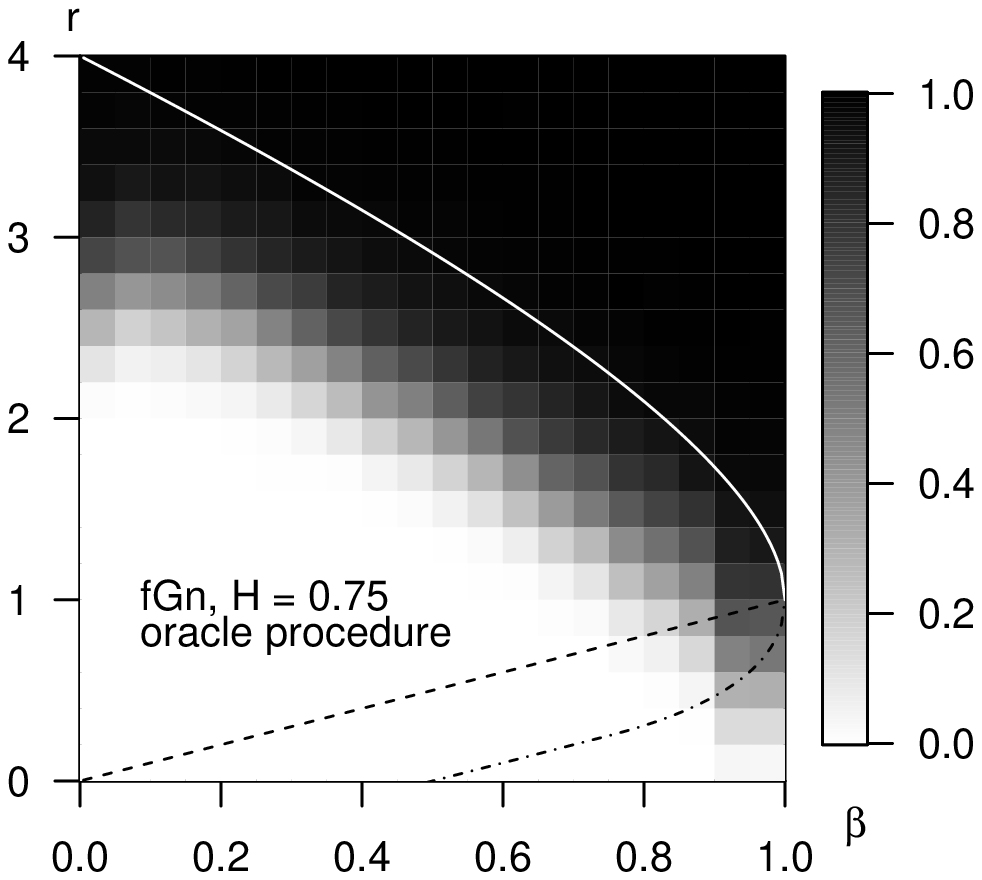}
    \includegraphics[width=0.4\textwidth]{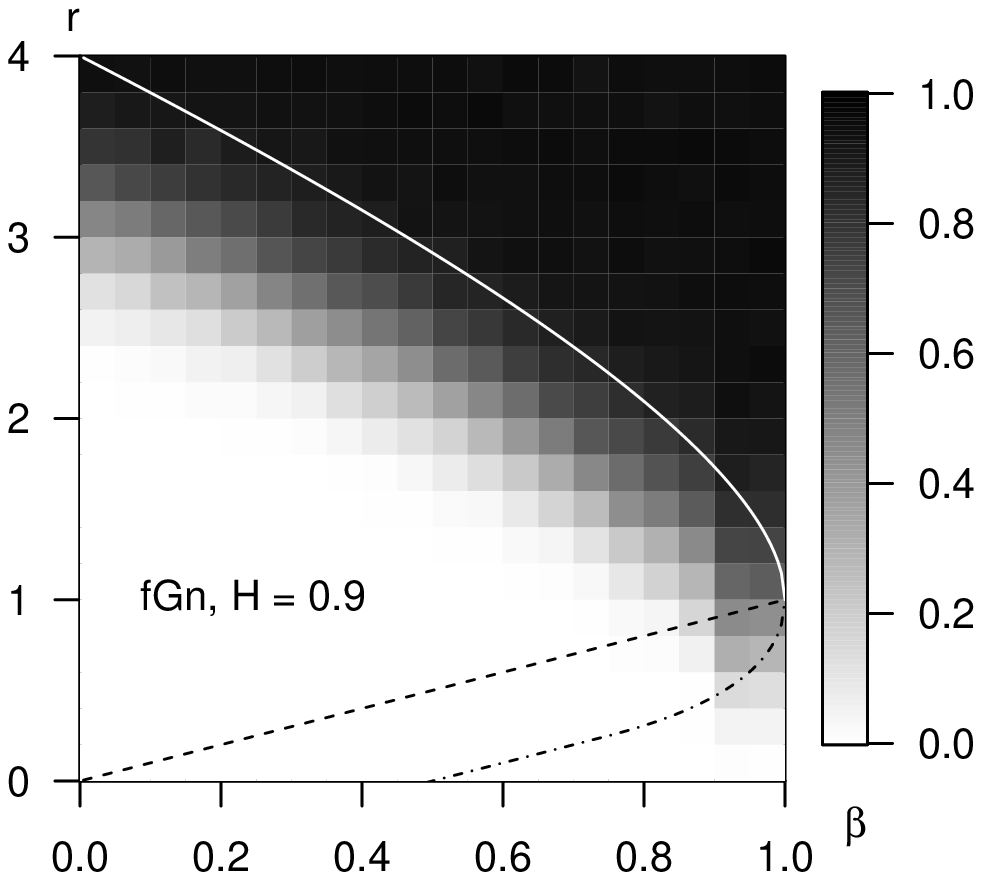}
    \includegraphics[width=0.4\textwidth]{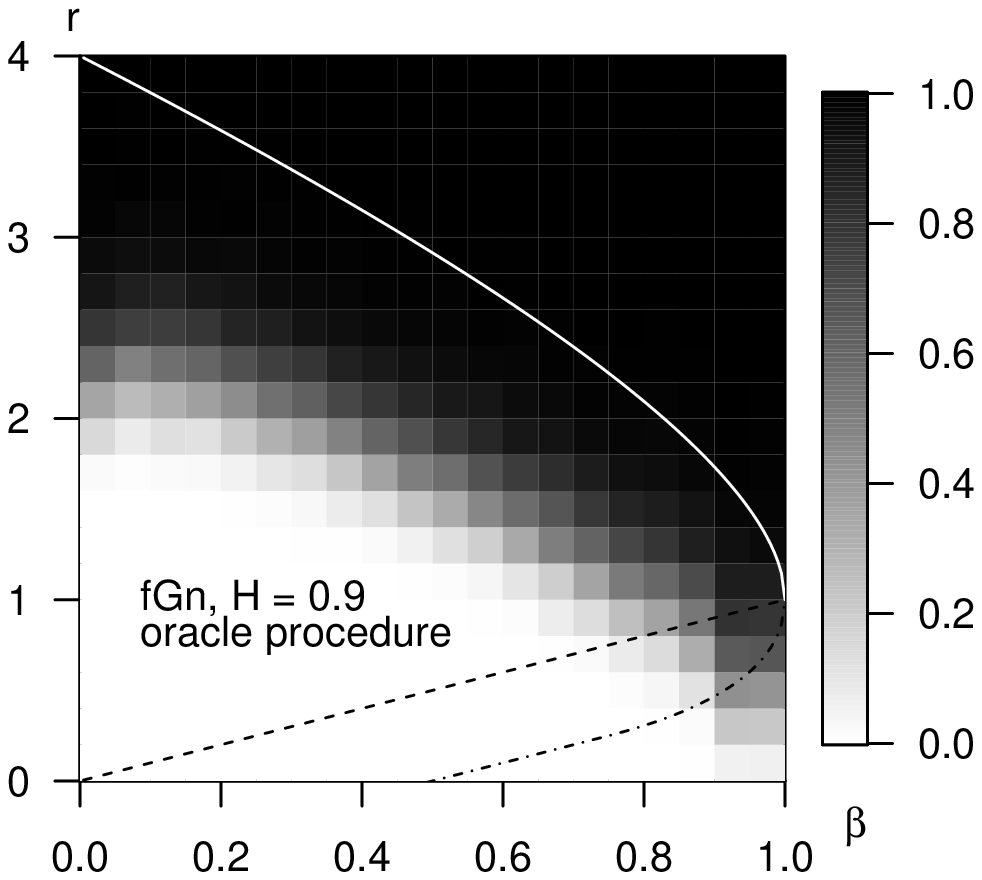}
    \includegraphics[width=0.4\textwidth]{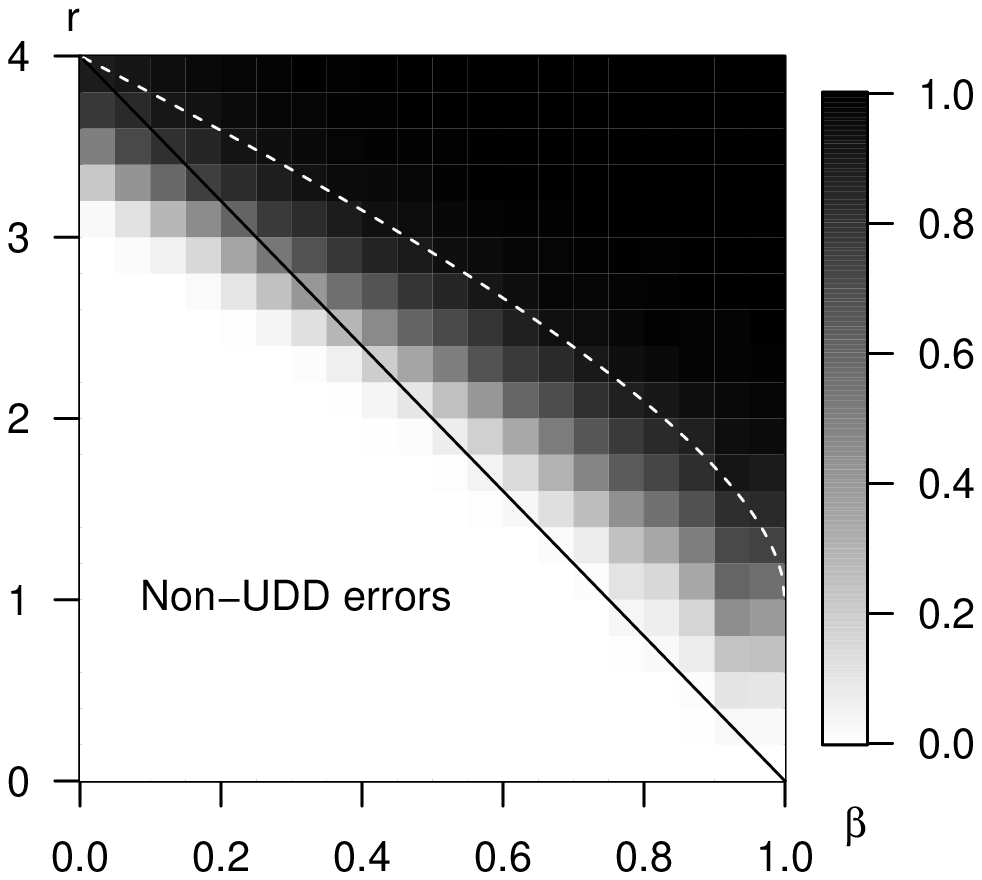}
    \includegraphics[width=0.4\textwidth]{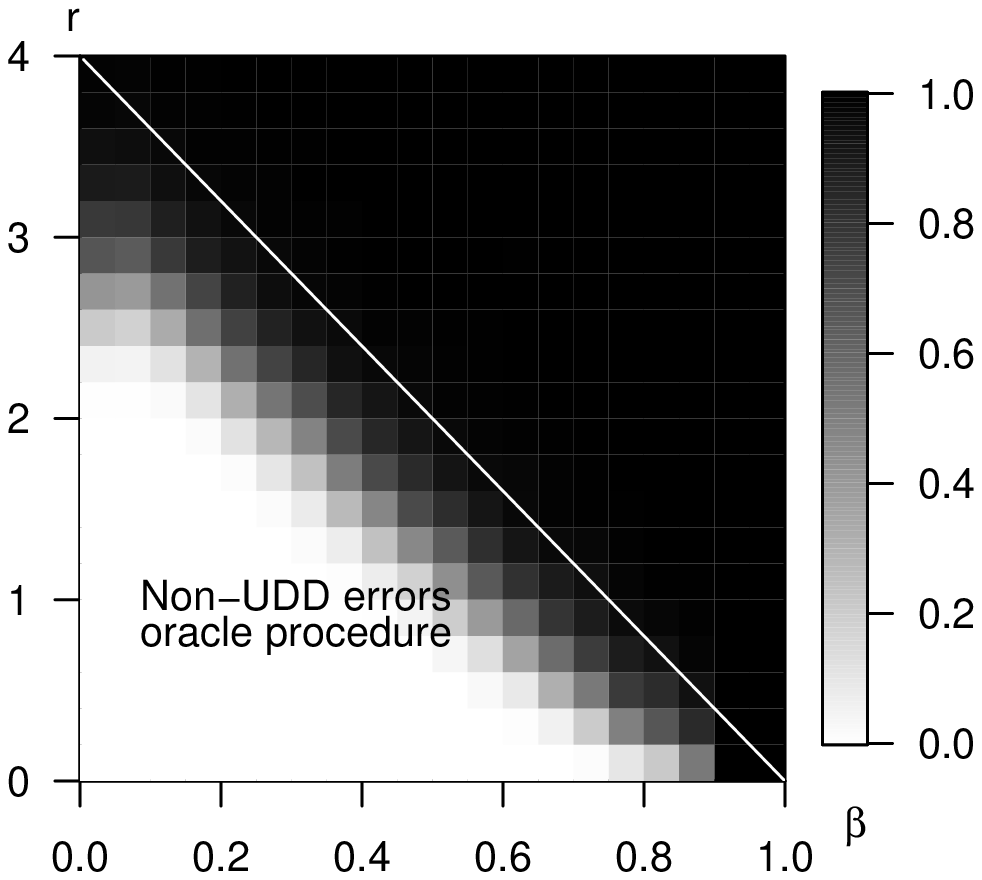}
    \caption{The empirical probability of exact support recovery from numerical experiments, as a function of sparsity level $\beta$ and signal sizes $r$. Darker colors indicate higher probability of exact support recovery. 
    Two fGn models with Hurst parameter $H = 0.75$ (upper), 
    $H = 0.9$ (middle), and the non-UDD errors in Example \ref{exmp:counter-example} (lower) are simulated.
    The experiments were repeated 1000 times for each sparsity-signal size combination.
    In finite dimensions ($p=10000$), the oracle procedures (right) is able to recover support for weaker signals than the Bonferroni procedures (left) when errors are heavily dependent, although they have the same phase-transition limit.
    The non-UDD errors demonstrate qualitatively different behavior, enabling support recovery for strictly weaker signals.
    The boundaries (solid, dashed, and dash-dotted lines) are as in Fig \ref{fig:phase-simulated}.
    In the non-UDD example, dashed lines represent the limit attained by Bonferroni's procedures.
    See text for additional comments.}
    \label{fig:phase-simulated-very-dependent}
\end{figure}


\appendix

\section{Proofs}
\label{sec:proofs}
\subsection{Bounding the upper tails of AGG maxima}
\label{subsec:bounding-upper-tails-of-maxima}

\begin{proof}[Proof of Lemma \ref{lemma:AGG-maxima-upper-tails}]
Recall by \eqref{eq:AGG-quantiles} that 
\begin{equation*}
    u_q\sim\left(\nu\log{q}\right)^{1/\nu}, \quad q\to\infty,
\end{equation*}
so that
\begin{equation} 
c_p  = \frac{u_{p\log{p}}}{u_p} -1 = \left(\frac{\log{p}+\log{\log{p}}}{\log{p}}\right)^{1/\nu}(1+o(1)) - 1 \rightarrow 0 \quad \mbox{as } p\to\infty.
\end{equation} 
By the union bound, we have
\begin{align}
    \P\left[\frac{M_p}{u_p} > 1+c_p\right] 
        &\le \sum_{j=1}^p \P\left[\frac{\epsilon_p(j)}{u_p} > 1+c_p\right] 
        = p \overline{F}\left(u_{p\log{p}}\right) \label{eq:AGG-maxima-upper-tails-proof-1} \\
        &= p \overline{F}\left(F^{\leftarrow}\left(1-\frac{1}{p\log{p}}\right)\right) \le \frac{1}{\log{p}} \rightarrow 0. \nonumber
\end{align}
where the last inequality follows from the fact that $F\left(F^{\leftarrow}(u)\right)\ge u$ for all $u\in[0,1]$.
\end{proof}

In addition to Lemma \ref{lemma:AGG-maxima-upper-tails}, which says the upper tail vanishes in probability, we will also prepare a result which states that the upper tail also vanishes in expectation.

\begin{lemma} \label{lemma:AGG-max-uniform-integrability}
Let $M_p$ and $c_p$ be as in Lemma \ref{lemma:AGG-maxima-upper-tails}, and denote 
$$
\xi_p := \frac{M_p}{(1+c_p)u_p}.
$$
Then there exists $p_0,t_0 > 0$, and absolute constant $C>0$ such that
\begin{equation} \label{eq:lemma-AGG-uniform-integrability}
    \P\left[\xi_p > t \right]\le \exp{\{-Ct^\nu\}}, \quad \text{for all} \quad p>p_0, t>t_0.
\end{equation}
In particular, the set of random variables $\{\left(\xi_p\right)_+,p\in\N\}$ is uniformly integrable.
\end{lemma}
\begin{proof}[Proof of Lemma \ref{lemma:AGG-max-uniform-integrability}]
Recalling that $(1+c_p)u_p = u_{p\log{p}}$, and by applying the union bound as in \eqref{eq:AGG-maxima-upper-tails-proof-1}, we have
\begin{align}
    \log \P\left[\xi_p > t\right] 
        &\le \log p + \log{\overline{F}\left(u_{p\log{p}}t\right)} \nonumber \\
        &\le \log p - \frac{1}{\nu}\left(u_{p\log{p}}t\right)^\nu(1-\delta). \label{eq:lemma-AGG-uniform-integrability-proof-1}
\end{align}
for $t > t_0(\delta)>0$, where $\delta\in(0,1)$ is an arbitrarily small number fixed in advance. 
This follows from the assumption that $F\in\text{AGG}(\nu)$ and the Definition \ref{def:AGG} of AGG tails.
Using in \eqref{eq:lemma-AGG-uniform-integrability-proof-1} the explicit expressions for quantiles in \eqref{eq:AGG-quantiles}, we obtain
\begin{equation} \label{eq:lemma-AGG-uniform-integrability-proof-2}
    \log \P\left[\xi_p > t\right] \le \log p - \underbrace{\left(1+o(1)\right)(1-\delta)t^\nu}_{\text{greater than }1\text{ for large }t}\log{p} - t^\nu\underbrace{\log{\log{p}}\left(1+o(1)\right)(1-\delta)}_{\text{greater than }C\text{ for large }p}.
\end{equation}
For large $t$, we have $\left(1+o(1)\right)(1-\delta)t^\nu > 1$, so that sum of the first two terms on the right-hand side of \eqref{eq:lemma-AGG-uniform-integrability-proof-2} is negative.
Also, for $p$ larger than some constant $p_0(\delta)$, we have $\log{\log{p}}\left(1+o(1)\right)(1-\delta) > C$ for some constant $C$ that does not depend on $p$.
Therefore \eqref{eq:lemma-AGG-uniform-integrability} holds for $t>t_0(\delta)$ and $p>p_0(\delta)$, and the proof is complete.
\end{proof}

\begin{corollary} \label{cor:AGG-max-upper-bound-expectation}
The upper tails of AGG maxima vanish in expectation, i.e.,
    \begin{equation} \label{eq:AGG-max-bound-upper-expectation}
    \E\left[\left(\frac{M_p}{u_p} - (1+c_p)\right)_+\right]
    \to 0 \quad\text{as }\; p\to\infty,
\end{equation}
where $(a)_+ := \max\{a,0\}$.
\end{corollary}

\begin{proof}[Proof of Corollary \ref{cor:AGG-max-upper-bound-expectation}]
Since $c_p\ge0$ is a sequence converging to 0, we have $c_p < 1$ for $p \ge p_0$. Hence for any $t>0$, we have
\begin{align}
    \P\left[\left(\frac{M_p}{u_p} - (1+c_p)\right)_+ > t\right] 
    &= \P\left[(1+c_p)\left(\xi_p-1\right)_+ > t\right] \nonumber \\
    &\le \P\left[\left(\xi_p-1\right)_+ > t/2\right] 
    \le \P\left[\xi_p > t/2 \right]. \label{eq:AGG-max-bound-upper-expectation-proof}
\end{align}
By Lemma \ref{lemma:AGG-max-uniform-integrability}, $\{\left(\xi_p\right)_+\}$ is u.i., therefore by Relation \eqref{eq:AGG-max-bound-upper-expectation-proof}, $\{\left(M_p/u_p - (1+c_p)\right)_+,\;p\in\N\}$ is u.i. as well.
Since by Lemma \ref{lemma:AGG-maxima-upper-tails}, $\left(M_p/u_p - (1+c_p)\right)_+\to 0$ in probability, Relation \eqref{eq:AGG-max-bound-upper-expectation} follows from the established uniform integrability.
\end{proof}

\subsection{Bounding the lower tails of Gaussian maxima}
\label{subsec:bounding-lower-tails-of-maxima}

The main goal of this section is to establish the following result. 

\begin{proposition} \label{prop:Gaussian-maxima-expectation-lower-bound}
For every UDD Gaussian array $\cal E$, and any sequence of subsets
$S_p\subseteq\{1,\ldots,p\}$ such that $q = q(p) = |S_p|\to \infty$, we have
\begin{equation} \label{eq:AGG-max-bound-expectation}
    \liminf_{p\to\infty} \E\left[\frac{M_{S_p}}{u_q}\right] \ge 1,
\end{equation}
where $M_S = \max_{j\in S}\epsilon(j)$.
\end{proposition}

Lemma \ref{lemma:Gaussian-maxima-lower-expectation}, which is the key to the proof of the `if' part of Theorem \ref{thm:Gaussian-weak-dependence}, follows immediately from this proposition.

\begin{proof}[Proof of Lemma \ref{lemma:Gaussian-maxima-lower-expectation}]
We start with the identity
$$
\E\left[\frac{M_{S_p}}{u_q}-(1+c_q)\right] = \E\left[\left(\frac{M_{S_p}}{u_q}-(1+c_q)\right)_+\right] - \E\left[\left(\frac{M_{S_p}}{u_q}-(1+c_q)\right)_-\right].
$$
By re-arranging terms and taking limsup/liminf, we obtain
\begin{align} \label{e:lim-sup-vanishes}
    0 \le &\limsup_{p\to\infty}\E\left[\left(\frac{M_{S_p}}{u_q}-(1+c_q)\right)_-\right]\nonumber \\
        \le & \limsup_{p\to\infty}\E\left[\left(\frac{M_{S_p}}{u_q}-(1+c_q)\right)_+\right] - \liminf_{p\to\infty}\E\left[\frac{M_{S_p}}{u_q}-(1+c_q)\right]\\
        = & - \liminf_{p\to\infty}\E\left[\frac{M_{S_p}}{u_q}-(1+c_q)\right],
        \label{e:lim-inf-bound}
\end{align}
where the last equality follows from the fact that the lim-sup in \eqref{e:lim-sup-vanishes} vanishes by Corollary \ref{cor:AGG-max-upper-bound-expectation}.
On the other hand, since $c_q\to 0$, we have
$$
\liminf_{p\to\infty}\E\left[\frac{M_{S_p}}{u_q}-(1+c_q)\right] 
= \liminf_{p\to\infty}\E\left[\frac{M_{S_p}}{u_q}-1\right] \ge 0,
$$
where the last inequality follows from Proposition \ref{prop:Gaussian-maxima-expectation-lower-bound}.  This shows that 
the right-hand side of \eqref{e:lim-inf-bound} is non-positive and hence
\eqref{eq:Gaussian-maxima-lower-expectation} holds. 
\end{proof}

We now prove Proposition \ref{prop:Gaussian-maxima-expectation-lower-bound}. This is where the UDD dependence assumption is used.

\begin{proof}[Proof of Proposition \ref{prop:Gaussian-maxima-expectation-lower-bound}]
Define the canonical (pseudo) metric on $S_p$,
$$
d(i,j) = \sqrt{\E\left[\left(\epsilon(i)-\epsilon(j)\right)^2\right]}.
$$
It can be easily checked that the canonical metric takes values between 0 and 2.
For arbitrary $\delta\in(0,1)$, take $\gamma = \sqrt{2(1-\delta)}$, and
let $\mathcal{N}$ be a $\gamma$-packing of $S_p$. That is, let $\mathcal{N}$ be a subset of $S_p$, such that for any $i,j\in\mathcal{N}$, $i\neq j$, we have $d(i,j)\ge\gamma$, i.e.,
\begin{equation}\label{e:gamma-packing-def}
d(i,j) = \sqrt{2\left(1-\Sigma_p(i,j)\right)} \ge \gamma = \sqrt{2(1-\delta)},
\end{equation}
or equivalently, $\Sigma_p(i,j) \le \delta$.
We claim that we can find a $\gamma$-packing $\mathcal{N}$ whose number of elements is at least 
\begin{equation} \label{eq:packing-number-lower-bound}
    |\mathcal{N}| \ge q/N(\delta).
\end{equation}
Indeed, $\mathcal{N}$ can be constructed iteratively as follows:
\begin{itemize}[
    align=left,
    leftmargin=4em,
    itemindent=0pt,
    labelsep=0pt,
    labelwidth=4em
    ]
    \raggedright
    \item[{\bf Step 1:}] Set $S_p^{(1)}:=S_p$ and $\mathcal{N}:=\{j_1\}$, where $j_1\in S_p^{(1)}$ is an arbitrary element. Set $k:=1$.\\
    \item[{\bf Step 2:}] Set $S_p^{(k+1)}:=S_p^{(k)}\setminus B_\gamma(j_k)$, where
    $$
    B_\gamma(j_k) := \{i\in S_p: \;d(i,j_k) < \gamma \equiv \sqrt{2(1-\delta)}\}.
    $$
    \item[{\bf Step 3:}] If $S_p^{(k)} \neq \emptyset$, pick an arbitrary $j_{k+1}\in S_p^{(k)}$, set $\mathcal{N}:=\mathcal{N}\cup\{j_{k+1}\}$, and $k:=k+1$, go to step 2; otherwise, stop.
\end{itemize}
By the definition of UDD (see Definition \ref{def:UDD}), there are at most $N(\delta)$ coordinates whose covariance with $\epsilon(j)$ exceed $\delta$. 
Therefore at each iteration, $\left|B_\gamma(j_k)\right|\le N(\delta)$, and hence
$$
\left|S_p^{(k+1)}\right| \ge \left|S_p^{(k)}\right| - \left| B_\gamma(j_k)\right| \ge q - kN(\delta).
$$
The construction can continue for at least $q/N(\delta)$ iterations, and we have $|\mathcal{N}| \ge \lfloor q/N(\delta) \rfloor$ as desired.
    
Now we define on this $\gamma$-packing $\mathcal{N}$ an independent Gaussian process $\left(\eta(j)\right)_{j\in\mathcal{N}}$, 
$$\eta(j) = \frac{\gamma}{\sqrt{2}}Z(j) \quad j\in \mathcal{N},$$
where $Z(j)$'s are i.i.d. standard normal random variables.
Observe that by the definition of $\gamma$-packing in \eqref{e:gamma-packing-def}, the increments of the new process are smaller than those of the original process in the following sense, 
$$
\E\left[\left(\eta(i)-\eta(j)\right)^2\right] = \gamma^2 \le d^2(i,j) = \E\left[\left(\epsilon(i)-\epsilon(j)\right)^2\right]
$$
for all $i\neq j$, $i,j\in\mathcal{N}$. Applying the Sudakov-Fernique inequality (see, e.g., Theorem 2.2.3 in \citep{adler2009random}) to  $\left(\eta(j)\right)_{j\in\mathcal{N}}$ and  $\left(\epsilon(j)\right)_{j\in\mathcal{N}}$, we have
\begin{equation}\label{eq:Sudakov-1}
\E\left[\max_{j\in\mathcal{N}}\eta(j)\right] \le \E\left[\max_{j\in\mathcal{N}}\epsilon(j)\right] \le \E\left[\max_{j\in{S_p}}\epsilon(j)\right].
\end{equation}
Since the $\left(\eta(j)\right)_{j\in\mathcal{N}}$ are independent Gaussians, Lemma \ref{lemma:expectation-lower} yields the lower bound,
\begin{equation}\label{eq:Sudakov-2}
\liminf_{p\to\infty} \E\left[\frac{\max_{j\in\mathcal{N}}\eta(j)}{u_{|\mathcal{N}|}}\right] \ge \frac{\gamma}{\sqrt{2}} = \sqrt{1-\delta}.
\end{equation}
Using the expressions \eqref{eq:AGG-quantiles} for the quantiles 
of AGG models (with $\nu=2$ here), we have
\begin{equation}\label{eq:Sudakov-3}
\frac{u_{|\mathcal{N}|}}{u_q} 
\ge \left(\frac{\log q-\log{N(\delta)}}{\log{q}}\right)^{1/2}\left(1+o(1)\right)\to 1,
\end{equation}
since $N(\delta)$ does not depend on $q= q(p)\to \infty$, and that $|\mathcal{N}|\ge q/N(\delta)$.

By combining \eqref{eq:Sudakov-1}, \eqref{eq:Sudakov-2} and \eqref{eq:Sudakov-3}, we conclude that
\begin{align*}
    \liminf_{p\to\infty} \E\left[\frac{\max_{j\in{S_p}}\epsilon(j)}{u_q}\right] 
    &\ge \liminf_{p\to\infty} \E\left[\frac{\max_{j\in\mathcal{N}}\eta(j)}{u_q}\right] &\text{by } \eqref{eq:Sudakov-1}\\
    &\ge \liminf_{p\to\infty} \E\left[\frac{\max_{j\in\mathcal{N}}\eta(j)}{u_{|\mathcal{N}|}}\right]  &\text{by } \eqref{eq:Sudakov-3} \\
    &\ge \sqrt{1-\delta}.  &\text{by } \eqref{eq:Sudakov-2}
\end{align*} 
Since $\delta>0$ is arbitrary, \eqref{eq:AGG-max-bound-expectation} follows as desired.
\end{proof}


\subsection{Proof of the claims in Examples \ref{exmp:FWER-controlling_procedures} and \ref{exmp:counter-example}}
\label{subsec:proofs-examples}

\begin{proof}[Proof of claims in Example \ref{exmp:FWER-controlling_procedures}]
By the Mill's ratio for the standard Gaussian distribution,
$$
\frac{t_p \P\left[Z>t_p\right]}{\phi(t_p)} \to 1,\quad \text{as}\quad t_p\to\infty,
$$
where $Z\sim \text{N}(0,1)$. 
Using the expression for $t_p = \sqrt{2\log{p}}$, we have
$$
p \;\P\left[Z>t_p\right] \sim \sqrt{2\pi}^{-1}\left(2\log{p}\right)^{-1/2} \to 0,
$$
as desired. The rest of the claims follow from Corollary \ref{cor:FWER-controlling_procedures}.
\end{proof}

\begin{proof}[Proof of claims in Example \ref{exmp:counter-example}]
Recall that $\widehat{S}^* = \{j:x(j)>t_p^*\}$, where $t_p^* = \sqrt{2(1-\beta)\log{p}}$. 
Analogous to \eqref{eq:Bonferroni-FWER-control} in the proof of Theorem \ref{thm:sufficient}, we have
\begin{align*}
    \P\left[\widehat{S} \subseteq S\right] 
        &= 1 - \P\left[\max_{j\in S^c}x(j) > t_p^*\right] 
        = 1 - \P\left[\max_{j\in S^c}\epsilon(j) > t_p^*\right] \nonumber \\
        &\ge 1 - \P\left[\max_{j\in[p]}\epsilon(j) > t_p^*\right] 
        \ge 1 - \P\left[\max_{j\in\{1,\ldots,\lfloor p^{1-\beta}\rfloor\}}\widetilde{\epsilon}(j) > t_p^*\right]
\end{align*}
where $\left(\widetilde{\epsilon}\right)_{j=1}^{\lfloor p^{1-\beta}\rfloor}$'s are independent Gaussian errors; in the last inequality we used the assumption that there are at most $\lfloor p^{1-\beta}\rfloor$ independently distributed Gaussian errors in $\left(\epsilon_p(j)\right)_{j=1}^p$.
By Example \ref{exmp:FWER-controlling_procedures} (with $\lfloor p^{1-\beta}\rfloor$ taking the role of $p$), we know that the FWER goes to 0 at a rate of 
$\left(2\log{\lfloor p^{1-\beta}\rfloor}\right)^{-1/2}$.
Therefore, the probability of no false inclusion converges to 1.

On the other hand, since the signal sizes are no smaller than $(\nu\underline{r}\log p)^{1/\nu}$, similar to \eqref{eq:sufficient-proof-eq1}, we obtain
\begin{align}
    \P\left[\widehat{S} \supseteq S\right] 
    &\ge \P\left[\min_{j\in S}\epsilon(j) > \sqrt{2(1-\beta)\log{p})} - \sqrt{2\underline{r}\log{p}} \right] \nonumber \\
    &= \P\left[\max_{j\in S}\left(-\epsilon(j)\right) < \sqrt{2\log{p})}\left(\sqrt{\underline{r}}-\sqrt{1-\beta}\right) \right] \nonumber \\
    &= \P\left[\frac{\max_{j\in S}(-\epsilon(j))}{u_{|S|}} < \frac{\sqrt{\underline{r}}-\sqrt{1-\beta}}{\sqrt{1-\beta}}\left(1+o(1)\right) \right], \label{eq:sufficient-proof-counter-example}
\end{align}
where in the last line we used the quantiles \eqref{eq:AGG-quantiles}.
Since the minimum signal size is bounded below by $\underline{r} > 4(1-\beta)$, the right-hand-side of the inequality in \eqref{eq:sufficient-proof-counter-example} converges to a constant strictly larger than 1. While the left-hand-side, by Slepian's Lemma \cite{slepian1962one}, is stochastically smaller than a r.v. going to 1, i.e.,
\begin{equation}
  \frac{1}{u_{|S|}} \max_{j\in S}(-\epsilon(j)) \stackrel{d}{\le} \frac{1}{u_{|S|}} \max_{j\in S)} \epsilon^*(j) \stackrel{\P}{\longrightarrow} 1,
\end{equation}
where $\left({\epsilon^*}\right)_{j=1}^{\lfloor p^{1-\beta}\rfloor}$'s are independent Gaussian errors.
Therefore the probability in \eqref{eq:sufficient-proof-counter-example} must also converge to 1.
\end{proof}

\subsection{Proofs of Lemma \ref{lemma:optimal-oracle-procedures} and Lemma \ref{lemma:likelihood-ratio-thresholding}}

\begin{proof}[Proof of Lemma \ref{lemma:optimal-oracle-procedures}]
The problem of support recovery can be equivalently stated as a classification problem, where the discrete parameter space is $\mathcal{S} = \{S\subseteq[p]:|S|=s\}$, and the observation $x \in\R^p$ has likelihood $f(x|S)$ indexed by the support set $S$.

By the optimality of the Bayes classifier (see, e.g., \citep{domingos1997optimality}), a set estimator that maximizes the probability of support recovery is such that
$$
\widehat{S} \in \argmax_{S\in \mathcal{S}} f(x|S) \pi(S).
$$
Since we know from \eqref{eq:uniform} that $\pi(S)$ is uniform, the problem in our context reduces to showing that $f(x|\widehat{S}^*) = f(x|\widehat{S})$, where $f(x|S)$ is the conditional distribution of data given the unordered support $S$,
$$
f(x|S) 
= \sum_{P\in\sigma(S)} f(x|P) \pi^{\text{ord}}(P|S) 
= \frac{1}{s!} \left(\sum_{P\in\sigma{(S)}} \prod_{i=1}^s {f_{i}(x(P(i)))}\right) \prod_{k\not\in S}{f_0(x(k))},
$$
where $\sigma(S)$ is the set of all permutations of the indices in the support set $S$.

Suppose $\widehat{S} \neq \widehat{S}^*$, then there must be indices $j \in \widehat{S}$ and $j' \in \widehat{S}^c$ such that $x(j) \le x(j')$.
We exchange the labels of $x(j)$ and $x(j')$, and form a new estimate $\widehat{S}\,' = \big(\widehat{S}\setminus\{j\}\big)\cup\{j'\}$.
Comparing the likelihoods under $\widehat{S}$ and $\widehat{S}\,'$, we have
\begin{align}
    f(x|\widehat{S}) - f(x|\widehat{S}\,') 
    &= \frac{1}{s!} \sum_{P\in\sigma{(\widehat{S})}} \prod_{i=1}^s {f_{i}(x(P(i)))} f_0(x(j'))\prod_{k\not\in \widehat{S}\cup\{j'\}}{f_0(x(k))} - \nonumber \\
    &\quad\quad\quad - \frac{1}{s!} \sum_{P'\in\sigma{(\widehat{S}')}} \prod_{i=1}^s {f_{i}(x(P'(i)))} f_0(x(j)) \prod_{k\not\in \widehat{S}'\cup\{j\}}{f_0(x(k))} \nonumber \\
    &= \frac{1}{s!} \left(\sum_{i=1}^s a_i  \Big(f_i(x(j)) f_0(x(j')) - f_i(x(j')) f_0(x(j))\Big) \right) \prod_{k\not\in \widehat{S}\cup\{j'\}}{f_0(x(k))}, \label{eq:MLR-optimality-proof}
\end{align}
where the last equality follows by first summing over all permutations fixing $P(i) = j$ and $P'(i) = j'$, and setting $a_i = \sum_{P\in\sigma{(\widehat{S}\setminus\{j\})}} \prod_{i'\neq i} {f_{i'} (x(P(i')))}$. Notice that the $a_i$'s are non-negative.

Since $x(j) \le x(j')$, and that each of $\{f_0, f_{i}\}$ is an MLR family, we have
$$
\frac{f_i(x(j))}{f_0(x(j))} - \frac{f_i(x(j'))}{f_0(x(j'))} \le 0 \implies f_i(x(j)) f_0(x(j')) - f_i(x(j')) f_0(x(j)) \le 0.
$$
Using Relation \eqref{eq:MLR-optimality-proof}, we conclude that $f(x|\widehat{S}) \le f(x|\widehat{S}\,')$.
This implies that any estimator that is not $\widehat{S}^*$ may be improved, and the optimality follows.
\end{proof}

\begin{lemma} \label{lemma:four-point-concavity}
Let $\phi$ be any concave function on $\R$. For any $x<y\in\R$, and $\delta>0$ we have
$$
\phi(x) + \phi(y+\delta) \le \phi(y) + \phi(x+\delta).
$$
\end{lemma}

\begin{proof}[Proof of Lemma \ref{lemma:four-point-concavity}]
Pick $\lambda = \delta/(y-x+\delta)$, by concavity of $f$ we have
\begin{equation} \label{eq:four-point-concavity-1}
    \lambda \phi(x) + (1-\lambda) \phi(y+\delta) 
    \le \phi(\lambda x + (1-\lambda)(y+\delta)) 
    = \phi(y),
\end{equation}
and
\begin{equation} \label{eq:four-point-concavity-2}
    (1-\lambda) \phi(x) + \lambda \phi(y+\delta)
    \le \phi((1-\lambda) x + \lambda(y+\delta)) 
    = \phi(x+\delta).
\end{equation}
Summing up \eqref{eq:four-point-concavity-1} and \eqref{eq:four-point-concavity-2} and we arrive at the conclusion as desired.
\end{proof}

\begin{proof}[Proof of Lemma \ref{lemma:likelihood-ratio-thresholding}]
The proof is entirely analogous to that of Lemma \ref{lemma:optimal-oracle-procedures}.
Since we know from \eqref{eq:uniform} that $\pi(S)$ is uniform, the problem reduces to showing that $f(x|\widehat{S}_{\text{opt}}) = f(x|\widehat{S})$, where 
$$
\widehat{S} \in \argmax_{S \in \mathcal{S}} f(x|S) \pi(S).
$$
and
$f(x|S)$ is the conditional distribution of data given the unordered support $S$,
\begin{equation} \label{eq:likelihood-ratio-thresholding-proof}
    f(x|S) = \sum_P f(x|P) \pi^{\text{ord}}(P|S) = \prod_{j\in S} f_a(x(j)) \prod_{j\not\in S}{f_0(x(j))}.
\end{equation}
Suppose $\widehat{S} \neq \widehat{S}_{\text{opt}}$, then there must be indices $j \in \widehat{S}$ and $j' \in \widehat{S}^c$ such that $L(j) \le L(j')$.
If we exchange the labels of $L(j)$ and $L(j')$, that is, we form a new estimate $\widehat{S}\,' = \big(\widehat{S}\setminus\{j\}\big)\cup\{j'\}$,
comparing the log-likelihoods under $\widehat{S}$ and $\widehat{S}\,'$, we have
\begin{equation*}
    \log{f(x|\widehat{S})} - \log{f(x|\widehat{S}\,')} 
    = \log{f_a(x(j))} + \log{f_0(x(j'))} - \log{f_a(x(j'))} - \log{f_0(x(j))}.
\end{equation*}
By the definition of $L(j)$'s, and the order relations, we obtain
\begin{equation*}
    \log{f(x|\widehat{S})} - \log{f(x|\widehat{S}\,')} 
    = \log{L(j)} - \log{L(j')} \le 0
\end{equation*}
This implies that any estimator that is not $\widehat{S}_{\text{opt}}$ may be improved, and the optimality follows.
\end{proof}

\begin{remark}
In the non-log-concave setting, where we know that thresholding procedures are suboptimal, likelihood thresholding procedures are promising, thanks to Lemma \ref{lemma:likelihood-ratio-thresholding}. 
However, in the case where signals have difference sizes, likelihood thresholding procedures are undefined;
in such settings, existence of an optimal procedure is an open problem.

Indeed, in the proof of Lemma \ref{lemma:likelihood-ratio-thresholding}, identical the signal densities are needed so that the Relation \eqref{eq:likelihood-ratio-thresholding-proof} holds.
\end{remark}

\section{Properties of asymptotic generalized Gaussian models and auxiliary facts}
\label{sec:AGG}
We collect some facts about AGG models in this section.
In particular, we give approximate upper quantiles of the AGG tails (which extends naturally to lower quantiles), and show that iid AGG sequences have uniform relatively stable (URS) maxima (which extends naturally to URS minima).

%
\begin{proof}[Proof of Proposition \ref{prop:quantile}]
By definition of AGG, for any $\epsilon>0$, there is a constant $C(\epsilon)$ such that for all $x\ge C$, we have
$$
-\frac{1}{\nu}x^\nu(1+\epsilon) \le \log{\overline{F}(x)} \le -\frac{1}{\nu}x^\nu(1-\epsilon).
$$
Therefore, for all $x < x_l := \left((1+\epsilon)^{-1}\nu\log{p}\right)^{1/\nu}$, we have
\begin{equation} \label{eq:AGG-quantiles-proof-1}
    -\log{p} = -\frac{1}{\nu}x_l^\nu(1+\epsilon) \le \log{\overline{F}(x_l)} \le \log{\overline{F}(x)},
\end{equation}
and for all $x > x_u := \left((1-\epsilon)^{-1}\nu\log{p}\right)^{1/\nu}$, we have
\begin{equation} \label{eq:AGG-quantiles-proof-2}
    \log{\overline{F}(x)} \le \log{\overline{F}(x_u)} \le -\frac{1}{\nu}x_u^\nu(1-\epsilon) = -\log{p}.
\end{equation}
By definition of generalized inverse,
\begin{equation*}
    u_p := F^\leftarrow(1-1/p) = \inf\{x:\overline{F}(x)\le 1/p\} = \inf\{x:\log{\overline{F}(x)} \le -\log{p}\}.
\end{equation*}
We know from relations \eqref{eq:AGG-quantiles-proof-1} and \eqref{eq:AGG-quantiles-proof-2} that 
$$
[x_u, +\infty) \subseteq \{x:\log{\overline{F}(x)} \le -\log{p}\} \subseteq [x_l, +\infty),
$$
and so $x_l\le u_p \le x_u$, and the expression for the quantiles follow.
\end{proof}

A final auxiliary result is used in Relation \eqref{eq:Sudakov-2} in the proof of Theorem \ref{thm:Gaussian-weak-dependence},
where it is applied to the Normal distributed random variables.
We give here a general statement which may be of independent interest.

\begin{lemma} \label{lemma:expectation-lower}
Let $(X_i)_{i=1}^p$ be $p$ i.i.d. random variables with distribution $F$ such that $\E[(X_i)_-]$ exists, i.e.,
$$
\E[\max\{-X_i, 0\}] < \infty.
$$
Let $M_p = \max_{i=1,\ldots,p}X_i$. Assume that $F$ has a density $f$, which is eventually decreasing. 
More precisely, we suppose there exists a 
$C_0$ such that $0<F(C_0)<1$, and $f(x_1) \ge f(x_2)$ whenever $C_0 < x_1 \le x_2$. 
Under these assumptions, we have,
$$
\liminf_{p\to\infty}\frac{\E M_p}{u_{p+1}} \ge 1,
$$
where $u_{p+1} = F^{\leftarrow}(1 - 1/(p+1))$.
\end{lemma}

\begin{proof}[Proof of Lemma \ref{lemma:expectation-lower}]
Write 
$$
X_i = F^{\leftarrow}(U_i)
$$
where $U_i$ are i.i.d. uniform random variables on $(0,1)$.
Denote $M^{U}_p$ as the maximum of the $U_i$'s, we have $\E M_p = \E\left[F^{\leftarrow}(M^{U}_p)\right]$, and by conditioning, we obtain
\begin{align} \label{eq:lemma:expectation-lower-proof-1}
    \E M_p &= \E\left[F^{\leftarrow}(M^{U}_p)\;\big|\;M^{U}_p \ge F(C_0)\right] \P\left[M^{U}_p \ge F(C_0)\right] + \nonumber \\ 
           &\quad\quad +\E\left[F^{\leftarrow}(M^{U}_p)\;\big|\;M^{U}_p < F(C_0)\right] \P\left[M^{U}_p < F(C_0)\right]. 
\end{align} 
We first handle the first term in the summation. Since $f$ is decreasing beyond $C_0$, $F$ is concave on $(C_0, \infty)$, and $F^{\leftarrow}$ is convex on $(F(C_0), 1)$. By Jensen's inequality, we have
\begin{equation*}
    \E\left[F^{\leftarrow}(M^{U}_p)\;\big|\;M^{U}_p \ge F(C_0)\right] 
        \ge F^{\leftarrow}\left(\E[M^{U}_p\;|\;M^{U}_p\ge F(C_0)]\right).
\end{equation*}
With a direct calculation, one can show that 
\begin{equation*}
    F^{\leftarrow}\left(\E[M^{U}_p\;|\;M^{U}_p\ge F(C_0)]\right)
    = F^{\leftarrow}\left(\left(1-\frac{1}{p+1}\right)\left(\frac{1-F(C_0)^{p+1}}{1-F(C_0)^{p}}\right)\right),
\end{equation*}
and hence
\begin{align*}
    \E\left[F^{\leftarrow}(M^{U}_p)\;\big|\;M^{U}_p \ge F(C_0)\right] 
        &\ge F^{\leftarrow}\left(\left(1-\frac{1}{p+1}\right)\left(\frac{1-F(C_0)^{p+1}}{1-F(C_0)^{p}}\right)\right) \\
        &\ge F^{\leftarrow}\left(1-\frac{1}{p+1}\right) = u_{p+1}.
\end{align*}
Since $\P[M^{U}_p\le m\big|\;M^{U}_p < F(C_0)] = \left(m/F(C_0)\right)^p \le m/F(C_0)$ for $m\le F(C_0))$, we have
$$\left(M^{U}_p\;\big|\;M^{U}_p < F(C_0)\right)\stackrel{\text{d}}{\ge} \left(U_1\;\big|\;U_1 < F(C_0)\right),$$
where and the latter is the uniform distribution on $(0,F(C_0))$.
Therefore, for the second term of the sum in \eqref{eq:lemma:expectation-lower-proof-1}, by monotonicity of $F^{\leftarrow}$, we obtain
\begin{align*}
    \E\left[F^{\leftarrow}(M^{U}_p)\;\big|\;M^{U}_p < F(C_0)\right] 
        &\ge \E\left[F^{\leftarrow}(U_1)\;\big|\;U_1 < F(C_0)\right] \\
        &= \E\left[X_1\;\big|\;X_1 < C_0\right].
\end{align*}
Finally, since $\P\left[M^{U}_p < F(C_0)\right] = F(C_0)^p = 1 - \P\left[M^{U}_p \ge F(C_0)\right]$, by \eqref{eq:AGG-quantiles}, we have
\begin{align*}
    \frac{\E M_p}{u_{p+1}} 
        &\ge \left(1 - F(C_0)^p\right) + \frac{\E\left[X_1\;\big|\;X_1 < C_0\right]}{u_{p+1}} F(C_0)^p.
\end{align*}
The conclusion follows since the right-hand-side of the last inequality converges to 1.
\end{proof}


\section{Strong classification boundaries in other light-tailed models}
\label{sec:other-boundaries}
The strong classification boundaries extend beyond the AGG models.
As our analysis in Section \ref{sec:URS} suggests, all error models with URS maxima demonstrate this phase transition phenomenon under appropriate parametrization of the sparsity and signal sizes.
We derive explicit boundaries for two additional classes of models under the general form of the additive noise models \eqref{eq:model}, with heavier and lighter tails than the AGG models, respectively. 

We would like to point out that the sparsity and signal sizes can be re-parametrized for the the boundaries to have different shapes.
For example in the case of Gaussian errors, if we re-parametrize sparsity $s$ with 
$\widetilde{\beta} = 2 - \left(1 + \sqrt{1-\beta}\right)^2$ where $\widetilde{\beta}\in(0,1)$, then the signal sparsity would have a slightly more complicated form:
$$
\left|S_p\right| = \left\lfloor p^{1-\beta} \right\rfloor = \left\lfloor p^{\left(\sqrt{2 - \widetilde{\beta}} - 1\right)^2}\right\rfloor,
$$
while the strong classification boundary would take on the simpler form:
\begin{equation}\label{eq:altenative-parametrization}
g(\beta) = \widetilde{g}(\widetilde{\beta}) = 2 - \widetilde{\beta}.
\end{equation}
In the next two classes of models we will adopt parametrizations such that the boundaries are of the form $\widetilde{g}$ in \eqref{eq:altenative-parametrization}.

\subsection{Additive noise models with heavier-than-AGG tails}
Following Example \ref{exmp:heavier-than-AGG}, assume that the errors in Model \eqref{eq:model} have rapidly varying right tails
\begin{equation} \label{eq:heavier-than-AGG-boundary-1}
    \log{\overline{F}(x)} = - \left(\log x\right)^\gamma \left(c+M(x)\right),
\end{equation}
as $x\to \infty$, and left tails
\begin{equation} \label{eq:heavier-than-AGG-boundary-2}
    \log{{F}(x)} = - \left(\log{(-x)}\right)^\gamma \left(c+M(-x)\right),
\end{equation}
as $x\to -\infty$.

\begin{theorem} \label{thm:heavier-than-AGG}
Suppose the marginals $F$ follows \eqref{eq:heavier-than-AGG-boundary-1} and \eqref{eq:heavier-than-AGG-boundary-2}.
Let
$$
k(\beta) = \log{p} - \left((\log{p})^{1/\gamma} + \log{(1-\beta)}\right)^\gamma,
$$
and let the signal $\mu$ have 
$$|S_p| = \left\lfloor pe^{-k(\beta)} \right\rfloor$$
non-zero entries. Assume the magnitudes of non-zero signal entries are in the range between
$$\underline{\Delta} = \exp{\left\{(\log{p})^{1/\gamma}\right\}}\underline{r}
\quad\text{and}\quad
\overline{\Delta} = \exp{\left\{(\log{p})^{1/\gamma}\right\}}\overline{r}.$$
If $\underline{r} > \widetilde{g}(\beta) = 2 - \beta$, then Bonferroni's procedure $\widehat{S}_p$ (defined in \eqref{eq:Bonferroni-procedure}) with appropriately calibrated FWER $\alpha\to 0$ achieves asymptotic perfect support recovery, under arbitrary dependence of the errors.

On the other hand, when the errors are uniformly relatively stable, if $\overline{r} < \widetilde{g}(\beta) = 2 - \beta$, then no thresholding procedure can achieve asymptotic perfect support recovery with positive probability.
\end{theorem}

\subsection{Additive noise models with lighter-than-AGG tails}
Following Example \ref{exmp:lighter-than-AGG}, assume that errors in Model \eqref{eq:model} has rapidly varying right tails
\begin{equation} \label{eq:lighter-than-AGG-boundary-1}
        \log{\overline{F}(x)} = - \exp{\left\{x^\nu L(x)\right\}},
\end{equation}
where $L(x)$ is a slowly varying function, as $x\to\infty$, and left tails
\begin{equation} \label{eq:lighter-than-AGG-boundary-2}
        \log{\overline{F}(x)} = - \exp{\left\{x^\nu L(x)\right\}},
\end{equation}
as $x\to -\infty$.

\begin{theorem} \label{thm:lighter-than-AGG}
Suppose marginals $F$ follow \eqref{eq:lighter-than-AGG-boundary-1} and \eqref{eq:lighter-than-AGG-boundary-2}.
Let
$$
k(\beta) = \log{p} - \left(\log(p)\right)^{(1-\beta)^\nu},
$$
and let the signal $\mu$ have 
$$|S_p| = \left\lfloor pe^{-k(\beta)} \right\rfloor$$
non-zero entries. Assume the magnitudes of non-zero signal entries are in the range between
$$\underline{\Delta} = \log{\log{p}}^{1/\nu}\underline{r}
\quad\text{and}\quad
\overline{\Delta} = \log{\log{p}}^{1/\nu}\overline{r}.$$
If $\underline{r} > \widetilde{g}(\beta) = 2 - \beta$, then Bonferroni's procedure $\widehat{S}_p$ (defined in \eqref{eq:Bonferroni-procedure}) with appropriately calibrated FWER $\alpha\to 0$ achieves asymptotic perfect support recovery, under arbitrary dependence of the errors.

On the other hand, when the errors are uniformly relatively stable, if $\overline{r} < \widetilde{g}(\beta) = 2 - \beta$, then no thresholding procedure can achieve asymptotic perfect support recovery with positive probability.
\end{theorem}

\section{Thresholding procedures under heavy-tailed errors}
\label{sec:heavy-tailed}
We analyze the performance of thresholding estimators under heavy-tailed models in this section, and illustrate its lack of phase transition.
Suppose we have iid errors with Pareto tails in Model \eqref{eq:model}, that is, $\epsilon(j)$'s have common marginal distribution $F$ where
\begin{equation} \label{eq:pareto-tails}
    \overline{F}(x) \sim x^{-\alpha} \quad \text{and} \quad F(-x) \sim x^{-\alpha},    
\end{equation}
as $x\to\infty$. 
It is well-known (see, e.g., Theorem 1.6.2 of \citep{leadbetter2012extremes}) that the maxima of iid Pareto random variables have Frechet-type limits.
Specifically, we have
\begin{equation} \label{eq:Frechet-limit-1}
    \frac{\max_{j\in[p]}\epsilon(j)}{u_p} \implies Y,
\end{equation}
in distribution, where $u_p = F^{\leftarrow}(1-1/p)\sim p^{1/\alpha}$, and $Y$ is an $\alpha$-Frechet random variable. By symmetry in our assumptions, the same argument applies to the minima as well.

\begin{theorem} \label{thm:heavy-tails}
Let errors in Model \eqref{eq:model} be as described in Relation \eqref{eq:pareto-tails}.
Let the signal have $s = |S| = fp$ non-zero entries, with magnitude $\Delta = rp^{1/\alpha}$, where both $f$ and $r$ may depend on $p$, so that no generality is lost.

Under these assumptions, the necessary and sufficient condition for a thresholding procedures $\widehat{S}$ to achieve perfect support recovery ($\P[\widehat{S}=S]\to1$) is 
$$
\liminf_{p\to\infty} r\to\infty.
$$
Conversely, the necessary and sufficient condition  for a thresholding procedures to fail perfect support recovery ($\P[\widehat{S}=S]\to0$) is 
$$
\limsup_{p\to\infty} r\to 0.
$$
That is, there does not exist a non-trivial phase transition for thresholding procedures when errors have (two-sided) $\alpha$-Pareto tails.
\end{theorem}

\begin{proof}[Proof of Theorem \ref{thm:heavy-tails}]
Recall the oracle thresholding procedure $\widehat{S}^* = \left\{j:x(j) \ge x_{[s]}\right\}$.
The probability of exact support recovery by any thresholding procedure $\widehat{S}$ is bounded above by that of $\widehat{S}^*$, that is,
\begin{align}
    \P[\widehat{S}=S] &\le \P[\widehat{S}^*=S] 
        = \P\Big[\max_{j\in S^c}x(j) \le \min_{j\in S}x(j)\Big] \nonumber \\
        &= \P\Big[\frac{\max_{j\in S^c}x(j)}{u_p} \le \frac{\min_{j\in S}x(j)}{u_p}\Big] \nonumber \\
        &= \P\Big[\frac{M_{S^c}}{u_p} \le \frac{m_S}{u_p} + r\Big], \label{eq:heavy-tailed-case-proof-0}
\end{align}
where $M_{S^c} = \max_{j\in S^c}\epsilon(j)$ and $m_S = \min_{j\in S}\epsilon(j)$.
For any $\alpha > 0$, the following elementary relations hold,
\begin{equation*}
    0 < L \le (1-f)^{1/\alpha} + f^{1/\alpha} \le U < \infty, \quad \text{for all} \; f\in(0,1),
\end{equation*}
where $L = \min\left\{1, 2(1/2)^{1/\alpha}\right\}$ and $U = \max\left\{1, 2(1/2)^{1/\alpha}\right\}$.
Therefore we have,
\begin{equation} \label{eq:heavy-tailed-case-proof-1}
    U\max\left\{\frac{M_{S^c}}{u_p}, -\frac{m_S}{u_p}\right\} < r
    \implies
    (1-f)^{1/\alpha}\frac{M_{S^c}}{u_p} - f^{1/\alpha}\frac{m_S}{u_p} < r,
\end{equation}
and 
\begin{equation} \label{eq:heavy-tailed-case-proof-2}
    L\min\left\{\frac{M_{S^c}}{u_p}, -\frac{m_S}{u_p}\right\} < r
    \impliedby
    (1-f)^{1/\alpha}\frac{M_{S^c}}{u_p} - f^{1/\alpha}\frac{m_S}{u_p} < r.
\end{equation}
Putting together \eqref{eq:heavy-tailed-case-proof-0}, \eqref{eq:heavy-tailed-case-proof-1}, and \eqref{eq:heavy-tailed-case-proof-2}, we have
\begin{equation} \label{eq:heavy-tailed-case-proof-3}
    \P\Big[\max\left\{\frac{M_{S^c}}{u_p}, -\frac{m_S}{u_p}\right\} < r/U\Big]
    \le \P[\widehat{S}^*=S]
    \le \P\Big[\min\left\{\frac{M_{S^c}}{u_p}, -\frac{m_S}{u_p}\right\} < r/L\Big].
\end{equation}
But we know from \eqref{eq:Frechet-limit-1} that 
\begin{equation} \label{eq:heavy-tailed-case-proof-4}
    \P\Big[\max\left\{Y^{(1)}, Y^{(2)}\right\} < \frac{\liminf r}{U}\Big]
    \le \limsup \P\Big[\max\left\{\frac{M_{S^c}}{u_p}, -\frac{m_S}{u_p}\right\} < r/U\Big],
\end{equation}
and
\begin{equation} \label{eq:heavy-tailed-case-proof-5}
    \liminf \P\Big[\min\left\{\frac{M_{S^c}}{u_p}, -\frac{m_S}{u_p}\right\} < r/U\Big]
    \le \P\Big[\min\left\{Y^{(1)}, Y^{(2)}\right\} < \frac{\limsup r}{L}\Big].
\end{equation}
where $Y^{(1)}$ and $Y^{(2)}$ are independent $\alpha$-Frechet random variables.
The conclusions of Theorem \ref{thm:heavy-tails} follow immediately from \eqref{eq:heavy-tailed-case-proof-3}, \eqref{eq:heavy-tailed-case-proof-4} and \eqref{eq:heavy-tailed-case-proof-5}.
\end{proof}

The probability of exact recovery can be approximated if the parameters $r$ and $f$ converge.
The following lemma follows with a small modification of the arguments in the proof of Theorem \ref{thm:heavy-tails}.

\begin{corollary}
Under the assumptions in Theorem \ref{thm:heavy-tails}, if 
$\lim r = r^*$, and $\lim f = f^*$, for some constant $r^*\ge0$ and $f^*\in[0,1]$, then 
$$
\lim \P[\widehat{S}^*=S] = \P\Big[(1-f^*)^{1/\alpha}Y^{(1)} + (f^*)^{1/\alpha}Y^{(2)} < {r^*}\Big].
$$
where $Y^{(1)}$ and $Y^{(2)}$ are independent $\alpha$-Frechet random variables.
\end{corollary}

\comment{One might wonder if it would be meaningful to derive a ``phase transition'' under a different parametrization of the signal sizes, say 
\begin{equation} \label{eq:Pareto-parametrization-with-boundary}
    \Delta = p^{r/\alpha}.
\end{equation}
In this case, Theorem \ref{thm:heavy-tails} suggests that a ``phase transition'' takes place at $r=1$.
However, this non-multiplicative parametrization of the signal sizes would make power analysis (like in Example \ref{exmp:gap-when-signal-sparse}) dimension-dependent. 

To illustrate, in the case of Gaussian errors with variance 1, if we were interested in small signals of size $\sqrt{2r\log{p}}$, where $r<1$ is below the boundary \eqref{eq:strong-classification-boundary}, then we only need $n > 2/r$ samples to guarantee discovery of their support.
In the Pareto case with parametrization \eqref{eq:Pareto-parametrization-with-boundary}, however, if we were interested in small signals of size $p^{r/\alpha}$, where $r<1$, then the ``boundary'' says that we will need $n > p^{2(1-r)/\alpha}$ samples, which is exponential in the dimension $p$ and quickly diverges.
Recall that the ``boundary'' is really an asymptotic result in $p$. 
Such an approximation in finite dimensions becomes invalid}

\section*{Acknowledgements}
The authors thank Jinqi Shen and Yuanzhi Li for inspiring discussions.


\bibliography{references}

\end{document}